\documentclass[12pt]{amsart}
\usepackage[style=numeric,natbib=true]{biblatex}
\usepackage[utf8]{inputenc}
\usepackage[T1]{fontenc}
\usepackage{hyperref}
\usepackage{amsmath}
\usepackage{amsfonts}
\usepackage{amssymb}
\usepackage{amsthm}
\usepackage{mathtools}
\usepackage{latexsym}
\usepackage{euscript}
\usepackage{xspace}
\usepackage{ifthen}
\usepackage{xstring}
\usepackage{txfonts}
\usepackage{calrsfs}
\usepackage[capitalize]{cleveref}
\usepackage{refcount}
\usepackage{xcolor}
\usepackage{caption}
\usepackage{subcaption}
\usepackage[hyperfirst=false,nopostdot,nogroupskip]{glossaries}
\newcommand\ORCiD[1]{\href{https://orcid.org/#1}{\mbox{\scalerel*{
\begin{tikzpicture}[yscale=-1,transform shape]
\pic{orcidlogo};
\end{tikzpicture}
}{|}}}}
\renewcommand\ORCiD[1]{}
\newcommand{\amorcid}{0000-0003-0661-4495}
\newcommand{\sinaiorcid}{0000-0002-0463-0013}

{\vspace{1em}%\begin{Sbox}
\begin{minipage}{12cm}
\rule{12cm}{1pt}\sffamily\noindent}%
{\newline\rule{12cm}{1pt}
\end{minipage}%\end{Sbox}\fbox{\TheSbox}
 \vspace{1em}}

\newtheorem{teo}{Theorem}[section]
\newtheorem{lem}[teo]{Lemma}
\newtheorem{pro}[teo]{Proposition}
\newtheorem{cor}[teo]{Corollary}
\newcounter{numex}
\newtheorem{problem}{Problem}
\newtheoremstyle{exa}{3pt}{3pt}{\slshape}{1pt}{\scshape}{:}{1pt}{}
\newtheoremstyle{conj}{3pt}{3pt}{\slshape}{1pt}{\scshape}{:}{1pt}{}
{\theoremstyle{exa}
\newtheorem{Example}{Example}[section]}
{\theoremstyle{conj}
}

\newtheorem*{Rstat}{\rotulo}
\newcommand{\rotulo}{}
\newenvironment{rstat}[1]
   {\renewcommand\rotulo{\cref{#1}}\begin{Rstat}}{\end{Rstat}}
\newcommand{\Item}[1]{(\ref{#1})}

{\begin{enumerate}}%
{\end{enumerate}}

\newcommand{\R}{\ensuremath{\mathbb{R}}\xspace}
\newcommand{\Q}{\ensuremath{\mathbb{Q}}\xspace}

\newcommand{\C}{\ensuremath{\mathbb{C}}\xspace}
\newcommand{\CC}{\ensuremath{\mathcal{C}}\xspace}

\newcommand{\N}{\ensuremath{\mathbb{N}}\xspace}
\newcommand{\w}{\ensuremath{\omega}\xspace}
\providecommand\T{}
\renewcommand{\T}{\ensuremath{\mathbb{T}}\xspace}
\newcommand{\eps}{\ensuremath{\varepsilon}\xspace}
\newcommand{\st}[1]{\ensuremath{#1^\ast}\xspace}
\newcommand{\Ifor}{Interlace Formula\xspace}

\DeclareMathOperator{\Rep}{Re}
\DeclareMathOperator{\Disc}{Disc}
\DeclareMathOperator{\geom}{ge}
\DeclarePairedDelimiter\ceil{\lceil}{\rceil}
\DeclarePairedDelimiter\floor{\lfloor}{\rfloor}
\newcommand{\ccomma}{\mathbin{\raisebox{0.5ex}{,}}}

\newcommand{\ip}[2]{\ensuremath{\langle #1,\!#2\rangle}\xspace}
\newcommand{\norm}[1]{\ensuremath{\parallel\!#1\!\parallel}\xspace}
\newcommand{\sqf}{\sqrt{5}}

\crefname{cor}{Corollary}{Corollaries}
\crefname{teo}{Theorem}{Theorems}
\crefname{lem}{Lemma}{Lemmas}
\crefname{pro}{Proposition}{Propositions}
\crefname{conjec}{Conjecture}{Conjectures}
\crefname{Example}{Example}{Examples}
\crefname{Section}{Section}{Sections}
\newcommand{\conj}[2]{\ensuremath{\left\{\left.#1\,\right|\;#2\right\}}}
\makeatletter
\DeclareFontEncoding{LS1}{}{}
\DeclareFontSubstitution{LS1}{stix}{m}{n}
\DeclareMathAlphabet{\mathscr}{LS1}{stixscr}{m}{n}
\makeatother

\newcommand{\textscr}[1]{%
  \text{\usefont{LS1}{stixscr}{m}{n}#1}%
}

\newcommand{\il}[1]{\ensuremath{\textscr{i\!l}\ifthenelse{\equal{#1}{}}{}{\!\mspace{-2mu}\left(#1\right)}}\xspace}
\newcommand{\cn}[1]{\ensuremath{\textscr{c\!n}\ifthenelse{\equal{#1}{}}{}{\!\mspace{-2mu}\left(#1\right)}}\xspace}
\newcommand{\be}[1]{\ensuremath{\textscr{b\!e}\ifthenelse{\equal{#1}{}}{}{\!\mspace{-1.5mu}\left(#1\right)}}\xspace}
\newcommand{\darga}[1]{\ensuremath{\mathrm{dg}\!\left(#1\right)}\xspace}

\newcommand{\St}[1]{\ensuremath{\mathrm{S}_\w\!\left(#1\right)}\xspace}

\newcommand{\sage}{\texttt{SageMath}\xspace}
\newcommand{\ucirc}{\T}
\let\sig\tilde

\usepackage{graphicx}
\hypersetup{colorlinks,% 
citecolor=cyan,% 
linkcolor=red,% 
pdftex}
\usepackage{color}
\DeclareMathOperator{\Aut}{Aut}
\newcommand{\palpha}{\ensuremath{p_\alpha}\xspace}
\newcommand{\In}{\ensuremath{\mathbb{I}_n}\xspace}
\newcommand{\un}{\ensuremath{U_n}\xspace}
\newcommand{\vn}{\ensuremath{V_n}\xspace}
\newcommand{\tn}[1]{\ifthenelse{\equal{#1}{^}}{{\theta_n\!}^}{\ensuremath{\theta_n} #1}}%\xspace #1}}
\newcommand{\paln}[1]{\ensuremath{\mathcal{P}^{(#1)}}\xspace}
\newcommand{\Trim}[1]{\ensuremath{\mathcal{T}^{(#1)}}\xspace}
\newcommand{\Cpaln}[1]{\ensuremath{C\!\mathcal{P}^{(#1)}}\xspace}
\newcommand{\CTrim}[1]{\ensuremath{C\!\mathcal{T}^{(#1)}}\xspace}
\DeclareMathOperator{\trim}{trim}
\newcommand{\palsym}{\ensuremath{\sigma\!}}
\newcommand{\oldpal}[1]{
  \StrCut{#1}{;}\palARG\palX
  \IfStrEq{\palX}{\empty}{\def\palX{x}}{}
  \ensuremath{
    \IfSubStr{\palARG}{,}{
        \palsym^{\left(\StrBefore{\palARG}{,}\right)}_{\StrBehind{\expandafter\palARG}{,}}\left(\palX\right)
      }
      {\palsym_{\palARG}\left(\palX\right)}
    }
  }

\let\pal\newpal

\newcommand{\eipi}[1]{\ensuremath{
    \IfSubStr{#1}{;}{
      \text{e}^{\frac{2\pi i\StrBefore{#1}{;}}{\raisebox{2pt}{$\scriptscriptstyle \StrBehind{#1}{;}$}}}}
    {\text{e}^{2i\pi#1}}
}}

\newcommand{\lilhalf}{\raisebox{2pt}{\(\scriptscriptstyle\frac12\)}}

\newcommand{\dfn}{\dotfill}%\nopostdesc}
\newglossaryentry{circnum}{name=circle-number,
description={for \(p\) trim self-inversive,\\ \(\cn{p}=\inf\conj{\beta}{\palpha\
  \text{is circle rooted for all}\ \alpha\geq\beta}\)\dfn}}
\newglossaryentry{circroot}{name={circle rooted},
description={polynomial with all its roots are in the unit circle\dfn}}
\newglossaryentry{realroot}{name={real rooted},
description={polynomial all whose roots are real\dfn}}
\newglossaryentry{darga}{name={darga},%symbol={\(\darga{p}\)},
  description={\(\darga{p}=\min_j\{p_j\neq0\} +\max_j \{p_j\neq 0\}\)}\dfn}
\newglossaryentry{palpha}{name={\palpha},symbol={\palpha},sort={palpha},
  description={\(\palpha(x)=\alpha(x^{\darga{p}}+1)+p(x)\)\dfn}}
\newglossaryentry{palindromic}{name={palindromic},description={real polynomial where \(p_j=p_{\darga{p}-j}\) for all \(j\)\dfn}}
\newglossaryentry{self-inversive}{name={self-inversive},description={complex polynomial where \(p_j=\bar{p}_{\darga{p}-j}\) for all \(j\)\dfn}}
\newglossaryentry{strictlyangleinterlaces}{name={strictly angle interlaces},
description={\(p\quad\) s.a.i. a finite set \(A\) if\\ it has one root inside each sector determined by \(A\)\dfn}}
\newglossaryentry{thetan}{name={\(\theta_n\)},sort={theta},
description={\(\eipi{;n}\)\dfn}}
\newglossaryentry{Un}{name={\un},sort={Un},
description={The $n$'th roots of unity\dfn}}
\newglossaryentry{Vn}{name={\vn},sort={Vn},
description={The  roots of unity in the closed upper half plane\dfn}}
\newglossaryentry{trim}{name={trim polynomial},
description={darga > degree\dfn}}
\newglossaryentry{full}{name={full polynomial},
description={darga = degree\dfn}}
\newglossaryentry{trimmed}{name={trimmed part},
description={of self-inversive \(p\) is \(\trim\!p\!=\!p(x)\!-\bar{p}(0)x^{\darga{p}}-p(0)\)\dfn}}
\newglossaryentry{ilnum}{name={interlace number},
description={for \(p\) trim self-inversive,\\ \(\il{p}=\inf\conj{\alpha>0}{\palpha\ \text{strictly interlaces}\ \un }\)\dfn}}
\newglossaryentry{exact}{name={exact polynomial},
description={circle number equals the interlace number\dfn}}
\newglossaryentry{berror}{name={bounding error},
description={for \(p\) trim self-inversive,\\ \(\be{p}=\frac{\il{p}}{\cn{p}}-1\)\dfn}}
\newglossaryentry{Ifor}{name={Interlace Formula},
description={\cref{ilnumber} or \cref{siilnumber}\dfn}}
\newglossaryentry{icert}{name={interlace cert},
description={a root of unity which yields the interlace number\\ in the \Ifor\dfn}}
\newglossaryentry{ccert}{name={circle cert},
description={a double root of   \(p_{\cn{p}}(x)\)\dfn}}
\newglossaryentry{trimn}{name={\Trim{n}},sort={Trim},
description={space of trim palindromic polynomials of darga \(n\)\dfn}}
\newglossaryentry{ctrimn}{name={\CTrim{n}},sort={Ctrim},
description={space of trim self-inversive polynomials of darga \(n\)\dfn}}
\newglossaryentry{paln}{name={\paln{n}},sort={pal},
description={vector space of all palindromic polynomials of darga $n$\dfn}}
\newglossaryentry{cpaln}{name={\Cpaln{n}},sort={cpal},
description={space of all self-inversive polynomials of darga $n$\dfn}}
\newglossaryentry{pbasis}{name={\palsym-basis},sort={sigma},
description={basis of \paln{n} consisting of\\ \(\pal{n,j}(x)   :=x^j+x^{n-j}\), \(j=0,1,...,\floor*{n/2}\)\dfn}}
\newglossaryentry{prep}{name={\palsym-representation},sort={sigmarep},
description={for \(p\in\paln{n}\), \(p(x)=\sum_{j=0}^{\floor*{n/2}}\sig{p}_j\pal{n,j}\)\dfn}}
\newglossaryentry{geompol}{name={geometric polynomial},
description={\(ge_n(x) = x+ x^2 + \cdots+ x^{n-1}\)\dfn}}
\newglossaryentry{coneC}{name={\(C_j\)},sort={cone},
description={the simplex cone in \(\R^{\floor*{n/2}}\) of trim polynomials\\ which have \(\tn^j\) as an interlace cert\dfn}}
\makeglossaries

\title[Roots on the circle]{%Forcing a polynomial to have all roots on the unit circle\\ or\\
Dragging the roots of a polynomial to the unit circle}

\author[A.~Mandel]{Arnaldo Mandel} 
\thanks{This work was partially supported by Conselho Nacional de
Desenvolvimento Científico e Tecnológico – CNPq (Proc.
423833/2018-9).} 
  \address{Computer Science Department, Instituto de Matem\'atica e Estat\'\i stica, Universidade de
  S\~ao Paulo\\
  \emph{\small S\~ao Paulo, SP, Brazil 05508-970}\\
\textsc{Orcid}: {\ORCiD{\amorcid}}\amorcid}
\email{ am@ime.usp.br}

\author[S.~Robins]{Sinai Robins}
  \address{Computer Science Department, Instituto de Matem\'atica e Estat\'\i stica, Universidade de
  S\~ao Paulo\\
  \emph{\small S\~ao Paulo, SP, Brazil 05508-970}\\
\textsc{Orcid}: \ORCiD{\sinaiorcid} \sinaiorcid}
  \email{sinai.robins@gmail.com}  

\begin{document}

\glsaddallunused
\begin{abstract}
  Several conditions are known for a self-inversive polynomial
  that ascertain the location of its roots, and we present a
  framework for comparison of those conditions.  We associate a
  parametric family of polynomials \(\palpha\) to each such
  polynomial \(p\), and define \cn{p}, \il{p} to be the sharp
  threshold values of \(\alpha\) that guarantee that, for all
  larger values of the parameter, \(\palpha\) has, respectively,
  all roots in the unit circle and all roots interlacing the
  roots of unity of the same degree. Interlacing implies circle
  rootedness, hence \(\il{p}\geq\cn{p}\), and this inequality is
  often used for showing circle rootedness.  Both \cn{p} and
  \il{p} turn out to be semi-algebraic functions of the
  coefficients of \(p\), and some useful bounds are also
  presented, entailing several known results about roots in the
  circle.  The study of \il{p} leads to a rich classification of
  real self-inversive polynomials of each degree, organizing them
  into a complete polyhedral fan.  We have a close look at the
  class of polynomials for which \(\il{p}=\cn{p}\), whereas in
  general the quotient\(\frac{\il{p}}{\cn{p}}\) is shown to be
  unbounded as the degree grows.  Several examples and open
  questions are presented.
\end{abstract}
\maketitle
%\end{document}

{
  \hypersetup{linkcolor=blue}
  \tableofcontents
}

\newpage

\section{Introduction}
\label{sec:introduction}

There is considerable interest in polynomials that have all of their roots
on the real line (\emph{real rooted}\glsadd{realroot} polynomials) or on the unit
circle (\emph{circle rooted}\glsadd{circroot} polynomials).  From one perspective,
the two classes are the same: a M\"obius transformation mapping the real
line to the circle minus a point essentially maps real rooted polynomials
to circle rooted polynomials.  This well-known connection will be
explored further in section \ref{sec:circle-number}, but it is far from the
whole story.

One can envision two classes of results in this theme of real and
circle rootedness.  In the first class, specific families of
polynomials are considered, and then it is shown that some values
of the parameters imply the desired root location. For real
rooted polynomials there are many classical families of
orthogonal polynomials and several polynomials arising in
Combinatorics, either as counting functions or as generating
functions (\citet{ChSe},\citet{Br15}, \citet{SaVi},
\citet{BHVW}). For circle rooted polynomials, such families are
somewhat less common (see \citet{DR2015}, \citet{LS}, \citet{LR},
\citet{AGLS}, \citet{BMM}, \citet{LY}).

In the second class, some conditions on the coefficients implying
the root location are described.  There are some results of this
type for real rooted polynomials, as in \citet{kurtz}. For circle
rooted polynomials, however, there are plenty of sufficient
conditions of this kind on the coefficients
(\citet{LL2009,LL2007,LL2004,LL2003}, \citet{Chen},
\citet{Kw11}), and some of them have been used as tools for
results in the first class.  This paper lies in this second
class, but has a thoroughly different viewpoint on circle rooted
polynomials and a new perspective on the results in some of the
cited papers.  To present it, we discuss a little more some of
those results.

Let us for a moment forbid \(1\) as a root of a polynomial.  The most basic
necessary condition for a real degree \(n\) polynomial
\[
  p(x)=a_0+a_1x+\cdots+a_nx^n
\]
to be circle rooted is that \(a_k=a_{n-k}\) for all \(k\);
equivalently, as rational functions,
\(p(x)=x^np\left(\frac{1}{x}\right)\).  Now we reinstate \(1\) as
a kosher root, but demand that it appears with even multiplicity;
in this way, the necessary condition persists.  Those polynomials
have been called \emph{palindromic} or \emph{reciprocal}; we opt
for \emph{palindromic}.
%All articles about circle rootedness of
%real polynomials assume this from the outset.

Another motivation for this type of polynomials comes from the
problem of counting the zeros of a given complex polynomial in
the unit disk.  It refers to an appropriate generalization of
palindromic polynomial: a complex \(p(x)\) is
\emph{self-inversive} if
\(p(x)=x^n\bar{p}\left(\frac{1}x\right)\).  A symmetrization
technique developed by Cohn (see \cite[sect. 11.5]{RS}, and a
recent strengthening by \citet{LS}) essentially moves that
counting problem to that of counting the number of roots of a
self-inversive polynomial in the unit circle, and to being able
to verify whether such a polynomial is circle rooted.
\citet{vieira-symm} surveys self-inversive and palindromic
polynomials from a viewpoint that is entirely germane to this
paper, while \citet{JS} give some a priori reasons for
considering these kinds of polynomials.

Many of the methods and results to be presented here work for
self-inversive polynomials and will be shown as such.  However,
it is important to keep in mind that our focus in this paper is
real palindromic polynomials (wich are the self-inversive ones
with real coefficients), with self-inversive ones being taken
along for the ride, and only so long as the additional generality
comes at near zero cost. One can always multiply a palindromic
polynomial by a scalar so that if becomes monic - and it keeps being
palindromic; that is not true for self-inversive polynomials in
general.  So, for palindromic polynomials there is a natural
special role for the polynomials \(x^n+1\) and for the roots of
unity; forcing them upon self-inversive ones is expedient and
yields nontrivial results, but it is not exactly the ``right'' theory.
So, the discussion that follows will
concern real palindromic polynomials, but some of the main
results in \lcnamecrefs{sec:interlacing}
\labelcref{sec:interlacing,sec:interlace-number,sec:circle-number}
are proved for self-inversive polynomials.

A very simple way of showing that a degree \(n\) polynomial \(p\) is real
rooted is to present a sequence \(a_1<a_2<\cdots<a_{n-1}\) such that the
values \(p(a_i)\) alternate in sign; \(p\) is said to \emph{sign-interlace}
the sequence. Provided \(p(a_{n-1})\) has sign opposite to that of the
leading term of \(p\), the polynomial will have roots
\(x_1<x_2<\cdots<x_n\) such that \(x_1<a_1<x_2<a_2<\cdots<a_{n-1}<x_n\),
that is, the \(x_i\)'s and \(a_i\)'s \emph{interlace}.  Indeed, the main
technique for proving that a polynomial is real rooted is to make it part
of a recursive sequence, each member of it sign-interlacing the next one.
As a general technique, see \citet[1.14]{Fisk}, \citet{WY}.  For any given
polynomial, the \emph{Sturm sequence} (see \cite{BPR}) fulfills this role,
if the polynomial is real rooted.

Moving now to the complex plane, there is a geometrically
motivated notion: start with a set \(A=\{a_1,\ldots,a_n\}\) of
points in the unit circle and divide the complex plane into \(n\)
open angular sectors by half-lines from the origin, which pass
through each point in \(A\).  We say that the polynomial \(p\)
\emph{strictly angle-interlaces} \(A\) if it has one root in each
of those sectors.  These elementary concepts come together in:

\begin{rstat}{angle-inter}
  If a degree \(n\) self-inversive polynomial strictly
  angle-inter\-laces a set of size \(n\), then the polynomial is
  circle rooted.
\end{rstat}

Angle-interlacing sets in the unit circle have also been said to
be \emph{interspersed} (as in \cite{SSbook}).  Although any set
of \(n\) points in the unit circle works for the purpose of
showing that the polynomial is circle-rooted, there is one that
stands out, in the case of palindromic polynomials, namely the
set \un of \(n\)th roots of unity, and we claim that this is
because they are ``exactly'' angle interlaced by the roots of
\(x^n+1\).  Most of the preceding references prove circle
rootedness in tandem with angle-interlacing with \un.  This is
not just a coincidence; we will present a fairly simple way of
checking angle-interlacing with \un that touches on the
polynomial's Fourier transform.  

Our main motivating question derives from the idea of giving
interlacing the status of first-class target, instead of a
happenstance bonus:
\begin{quote}
\emph{Is looking for good conditions for angle-interlacing with
  \un a good strategy for nailing down circle rootedness?}
\end{quote}
We introduce numerical tools that allow us to treat this question
in quantitative terms.

% ​We propose a way to quantify this question.
% , while at the same time allowing us to consider them as
% properties of the same family of parametrized polynomials.

%It is natural to normalize palindromic polynomials to be monic,
%as we are simply considering their roots.  In other words, those
Monic palindromic polynomials look like \(x^n+1+ p(x)\), and suggest
the following idea.  The palindromic polynomial \(x^n+1\) is
trivially circle rooted and strictly angle-interlaces \un; the
latter property is open, hence all ``sufficiently close''
polynomials will also have it.  We can look at \(p(x)\) as a
perturbation of \(x^n+1\), and consider, for real \eps, the
polynomial \(x^n+1+\eps p(x)\), asking how small \(\eps\) has to
be for angle-interlacing and circle rootedness to set in.  It
turns out to be more convenient to scale the parametric
polynomial by \(\alpha=1/\eps\), and define, for any complex  polynomial
\(p\) of degree at most \(n\), the family
\begin{equation} 
    \palpha(x)= \alpha(x^n+1)+p(x),
\end{equation}
where \(\alpha\) is a positive real parameter.  Then, for a
sufficiently large \(\alpha\), \palpha angle-interlaces \un.  We
actually like the image of \(\alpha\) dragging the roots of
\palpha %kicking and screaming
until they latch on the unit
circle, and then spread out to the slots provided by the roots of
unity.

Let \( p(x)=a_rx^r+a_{r+1}x^{r+1}+\cdots+a_sx^s \) be a nonzero
polynomial, with \(a_r,a_s\neq 0\).  In order to also handle
polynomials of this form with  no constant term, we
follow \citet{Zeil} and define the \emph{darga} of \(p\) to be
\begin{equation} \label{darga}
\darga{p}=r+s.
\end{equation}
We say that a complex polynomial \(p\) of darga \(n\) is
\emph{self-inversive} if
\(x^np\left(\frac{1}{x}\right)=\bar{p}(x)\); that means
<\(p_j=\bar{p}_{n-j}\) for all coefficients \(p_j\).  For real
polynomials that reads \(p_j=p_{n-j}\), justifying the term
\emph{palindromic}.
If the darga equals the degree, we say that \(p\) is \emph{full};
otherwise, we call \(p\) a \emph{trim} polynomial.  Most of the
concepts and results in this article revolve around trim
polynomials.  If \(p\) is self-inversive, we define its
\emph{trimmed part}\glsadd{trimmed} by \(\trim p=p(x)-\bar{p}(0)x^{\darga{p}}-p(0)\).
One simple property that is good to keep in mind is that \(p\) is
palindromic and \(\w\) is in the unit circle, then \(p(\w)\) is
real; this is not generally true for self-inversive polynomials,
but holds if \(\w\in\un\).

Now we can present the two main quantities we will be exploring. We define, for a
trim self-inversive polynomial \(p\):
\begin{itemize}
    \item [the] \emph{interlace number}\\
      \(\il{p}=\inf\conj{\alpha>0}{\palpha(x)\,\text{angle interlaces \un}}\text{, and}\)
    \item [the] \emph{circle number} \\
    \(\cn{p}=\inf\conj{\alpha>0}{p_\beta(x)\,\text{is circle rooted for all}\ \beta>\alpha}.\)
\end{itemize}
  
Note the discrepancy between the definitions.  It turns out that
if \palpha angle interlaces \un, then so does \(p_\beta\) for all
\(\beta>\alpha\).  The same does not happen for the property of
being circle-rooted (\cref{ex:twointerv}), so the definition of
circle number has to be formally more complicated, to capture
what we want, foreshadowing its inherent complexity. We could
write the definition of \il{} to parallel that of \cn{}, but not
the other way around.

Those quantities allows one to rewrite and compare previously
published facts (see \cref{LL,kwon,chen}).  Many theorems about
circle-rootedness have the following form (where we identify a
trim polynomial \(\sum_{j=1}^{n-1}p_jx^j\) with the vector
\((p_1,\ldots,p_{n-1})\)):

\emph{A class \(\mathcal{F}\) of full palindromic polynomials,
  and functions \(h_n:\R^{n-1}\rightarrow\R_+\) for each \(n\)
  are specified.  It is asserted that if \(p\in\mathcal{F}\) has
  degree \(n\) and is such that its leading coefficient is real
  and at least \(h_n(\trim p)\), then \(p\) is
  circle-rooted}. Some statements come with a proviso that the
roots also angle-interlace \un (for self-inversive polynomials,
an appropriate rotation of \un, instead).

Such a fact can be restated, taking
\(\mathcal{F'}=\conj{\trim p}{p\in\mathcal{F}}\), as:

\emph{A class \(\mathcal{F'}\) of trim palindromic polynomials,
  and functions \(h_n:\R^{n-1}\rightarrow\R_+\) for each \(n\) are
  specified.  It is asserted that if \(p\in\mathcal{F'}\) has
  darga \(n\) then \(\cn{p}\leq h_n(p)\) (or the stronger
  conclusion \(\il{p}\leq h_n(p)\))}.

The existence of \il{p} follows from the topological argument about
proximity to \(x^n+1\).  \cref{angle-inter} trivially implies that \cn{p}
also exists and that we have the following simple relation between them.

\begin{rstat}{cnleil}
        \(\cn{p}\leq \il{p}\).  
\end{rstat}

Moreover, a simple argument (\cref{cnlobo}) shows that both
numbers are strictly positive.

These concepts provide a unifying framework for many of the
articles cited here; several known theorems about circle rootedness can
be reinterpreted as giving an upper bound for the interlace
number or for the circle number.

The paper proceeds as follows.  In \cref{sec:interlacing} we
present a quick background on interlacing.  This is continued in
\cref{sec:palindr-polyn}, where angle-interlacing and self-inversive
polynomials are discussed.  \cref{interlace} shows that one can
find out whether a full monic self-inversive polynomial angle interlaces
\un by computing the discrete Fourier transform of its
coefficients.

\begin{rstat}{interlace}%[\textsc{Positive Fourier coefficients}]
  Let \(p(x)\) be a full monic self-inversive polynomial of
  degree \(n\), with \(p(1)>0\).  Then \(p(x)\) strictly
  angle-interlaces \un if and only if
\begin{equation*}
  \quad p(\w)>0\quad\text{for all}\quad\w\in \un.
\end{equation*}
\end{rstat}

\Cref{sec:interlace-number} starts an account of the interlace number.  The
main result in this section is the following formulation for the interlace
number, in terms of the finite Fourier transform of the coefficients of
$p$.

\begin{rstat}{ilnumber}[\textsc{ Interlace formula}]
  If \(p\) is a trim self-inversive polynomial of darga \(n\), then
  \[
     \quad\il{p}  = \frac12\max\conj{-p(\w)}{\w\in\un}.
    \]
\end{rstat}
This result and some variations entail short proofs of some
criteria for circle rootedness in the literature that also yielded
interlacing with \un.  It turns out that, for several
parameterized families of polynomials, we can provide simple
formulas for the interlace number, based on the parameters.

For each trim palindromic polynomial \(p\), the maximum that
determines \il{p} is attained at a root of unity in the closed
upper half-plane.  Such a root we call an \emph{interlace cert},
and there may be many.  The trim palindromic polynomials of darga
\(n\) form a vector space of dimension \(\floor*{n/2}\); those
with a given interlace cert form a convex simplex cone, and the
collection of these cones comprises a complete polyhedral fan.
That Fan of Interlace Certs is depicted in detail in
\cref{sec:fan-interlace-certs}, along with a full description of
its very large group of isometries.

\Cref{sec:circle-number} discusses the circle number, which turns
out to be a more elusive concept than the interlace number.  In
this section, we use a different mapping of the circle to the
real line, namely a M\"obius transformation, making the proofs
fully algebraic.

\begin{rstat}{dbroot}[\textsc{Double-root formula}]
  For a trim self-inversive polynomial \(p\) of darga \(n\), \cn{p} is
  the largest value of \(\alpha\) for which
  \(\frac{\palpha(x)}{\gcd(p(x),x^n+1)}\) has a double root.
\end{rstat}
This implies that the circle root can be computed by first
calculating the discriminant of a polynomial, and then finding
its largest real root, as a polynomial in \(\alpha\).  This is
indeed quite effective in practice, for any given numerical
polynomial.  However, in contrast with what was achieved for the
interlace number, giving nontrivial bounds for the circle number
of polynomial families is a hard challenge.

\Cref{sec:comparing-two} is the culmination of this work, where
we compare the two numbers and see more precisely to what extent
the interlace number approximates the circle number, thus
directly addressing our main question.  On the one hand, we show
several sufficient conditions for these two numbers to coincide,
and argue that the so-called \emph{exact} polynomials form a
class worthy of further attention, in particular because for that
class the strategy underlying the main question is indeed optimal.
On the other hand, in \cref{supBE} we show that the numbers can
grow unboundedly apart.  More precisely, we define the
\emph{bounding error}
\[
  \be{p}=\frac{\il{p}}{\cn{p}}-1,
\]  
and show that, while it is bounded when restricted to polynomials
of a fixed darga, it is unbounded overall.

Finally, in the last two sections we have a more specialized look at those
numbers.  In \cref{sec:small-dargas} we investigate polynomials of small
degrees, giving a concrete feeling for how these numbers behave.  And, in
\cref{sec:some-examples} we look at some parameterized families and at other
circle rootedness results from the viewpoint raised in this article.

The research carried out here was enabled by the  use of
\sage\cite{sagemath}, which helped us in developing and testing all sorts
of wild conjectures. A few of these conjectures just turned out to be true
and are presented here as theorems.  Many results presented here (like
\cref{ilnumber,discdbroot}) lead to algorithms that are straightforward to
implement in the standard \sage environment.

In this text, ``self-inversive'' is code for
polynomials with possibly complex coefficients, while
``palindromic'' implies real.  It turns out that one could
restrict everything to polynomials over the (real) algebraic
numbers.  Indeed, \il{p} is a piecewise-linear function; \cn{p}
is far from that, but still is a semi-algebraic function of the
real and imaginary parts of the coefficients of \(p\). Some other
interesting sets of polynomials presented here are, like the
exact polynomials, also semi-algebraic.  We will not delve
deeply into this aspect, which surely deserves a better and more
knowledgeable attention, but will point out the connection here
and there.  We refer to the book by \citet{BPR} for the relevant
definitions and results.

Several concepts and some notation will be
defined as depending on a parameter \(n\).  As a rule,
whenever \(n\) is understood from the context, it will be omitted
both in terminology and notation, or added as a superscript as in
\Trim{n}.  Three frequently used symbols will always carry their
\(n\):

\begin{align*}
  \tn&=\eipi{;n}\glsadd{thetan}, && \\
  \un&=\conj{\tn^k}{0\leq k < n}\glsadd{Un}, &&\text{the set of \(n\)th roots of \(1\),}\\
   \vn&=\conj{\tn^k}{0\leq k \leq n/2}\glsadd{Vn}, &&\text{the roots with non negative imaginary part.}
\end{align*}

Also, in many statements there will
appear an unqualified entity \(p\); that will always be a
\emph{trim polynomial of darga \(n\)}. Repetition of this is too
soporific, so we only do it in the more important statements.

\section{Interlacing}
\label{sec:interlacing}

There is a vast literature on interlacing polynomials, much of which has
been collected by \citet{Fisk}, in a book which was, unfortunately,
unfinished.  We present some notions here.

Let \(A=\{a_1\leq a_2\leq\cdots\leq a_n\}\) and
\(B=\{b_1\leq b_2\leq\cdots\leq b_m\}\) be non decreasing
sequences of real numbers; they \emph{interlace} if they can be
merged into \(\{c_1\leq c_2\leq\cdots\leq c_{n+m}\}\), where the
\(c_j\)'s alternate between \(a\)'s and \(b\)'s.  In particular,
\(A\) interlaces itself. Notice that this requires that \(n\) and
\(m\) differ by at most \(1\); also, if a value appears twice in
one sequence, it must also appear in the other one. If all those
\(n+m\) numbers are distinct, we say that \(A\) and \(B\)
\emph{strictly interlace}; alternatively, two sequences strictly
interlace if they are both strictly increasing and between each
pair of successive members of each there lies exactly one of the
other.

A polynomial is said to be \emph{real rooted} if its coefficients
and all its roots are real; such a polynomial will be taken as a
proxy for the ordered sequence of its roots with multiplicity.
So, we will talk about polynomials interlacing, or a polynomial
interlacing a sequence, keeping in mind that the zeros of the
polynomial(s) are the issue.

A real polynomial \(f(x)\) is said to \emph{sign-interlace} a sequence of
real numbers \(A\) if the values \(f(a_j)\) are nonzero and alternate in
sign.

The following are easy exercises:

%\textbf{Revise this statement:}

\begin{pro}
  (see \cite[Lemma 1.9]{Fisk}) If \(f(x)\) is a degree \(n\geq 2\) real
  polynomial, and \(A\) is a set of real numbers of size \(n\), then
  \(f(x)\) has \(n\) real roots strictly interlacing \(A\) if and only if
  it sign-interlaces \(A\).
\end{pro}

\begin{pro}\label{signinter}
  A real polynomial strictly interlaces a sequence $A$ of $n$ distinct real
  numbers if and only if it sign-interlaces \(A\) and has at most \(n+1\)
  real zeros.
\end{pro}

This motivates the following definition that will be useful: a
continuous function \emph{sign-interlaces} \(A\) if its values on
\(A\) alternates in sign.  Clearly, such a function has at least
\(n-1\) zeros.  We will use this exactly once.

A sequence \(C\) is a \emph{common interlace} to sequences \(A\) and \(B\)
if it interlaces with both of them.  The following is a simplified version
of Prop. 1.35 in \cite{Fisk}, also slightly generalized so that is applies to
some polynomials with multiple roots:

\begin{pro}\label{lincomb}
  Let \(p,q\) be coprime polynomials of degree \(n\) and positive leading
  coefficient; suppose also that \(q\) has \(n\) distinct real roots. Then
  \(p_\lambda=(1-\lambda)p+\lambda q\) is real rooted for all
  \(0<\lambda\leq1\) if and only if \(p\) is real rooted, and \(p\) and
  \(q\) have a common interlace.  In this case, 
  \(p_\lambda\) has \(n\) distinct roots for \(0<\lambda<1\).
\end{pro}
\begin{proof}
  Let us notice first that if \(p\) has a nonreal root, then, by continuity
  of roots, so does \(p_\lambda\) for \(\lambda\) close to 0.  So, real
  rootedness of \(p\) is necessary for the whole setup.
  
  We consider first the simpler case where \(p\) also
  has \(n\) real roots.  Suppose that \(C=\{c_1<\ldots<c_{n-1}\}\) is a
  common interlace to \(p\) and \(q\).  Since the polynomials have positive
  leading coefficients, they are negative on \(c_{n-1}\), so they have the
  same sign on all \(c_j\).  It follows that all convex combinations of
  \(p\) and \(q\) sign-interlace \(C\) and have positive leading
  coefficient, therefore have \(n\) real roots.

  In general, let \(A=\{a_1\leq a_2\leq\cdots\leq a_n\}\) be the roots of
  \(p\), and \(B=\{b_1< b_2<\cdots< b_n\}\) be the roots of \(q\), and, for
  convenience, define \(a_{n+1}=b_{n+1}=\infty\).  The coprime hypothesis
  implies that \(A\) and \(B\) have no element in common.  If
  \(a_j=a_{j+1}\), then \(c_j=a_j\); as \(C\) interlaces \(B\), it
  follows that \(c_j<c_{j+1}\leq a_{j+2}\), implying that no root of \(p\)
  has multiplicity higher than \(2\).  Also, if some \(c_k\) is not a
  double root of \(p\), we can move it a bit so as not to coincide with
  either \(a_i\) or \(b_i\), maintaining that \(C\) still interlaces both
  polynomials; so, we may assume that unless \(c_k\) is a double root of
  \(p\), both \(p(c_k)\) and \(q(c_k)\) are nonzero.  Now, \(q\) has
  constant sign on each open interval between successive roots, and these
  signs alternate. The same happens to \(p\), except in the degenerate case
  of a double root, but in this case \(p\) has the same sign in the two
  adjacent intervals.

  Let us show that they have the same sign.  We do it by induction,
  descending from \(n-1\).  Since \(p\) and \(q\) have positive leading
  coefficients, they are both positive on the intervals \((a_n,a_{n+1})\),
  \((b_n,b_{n+1})\), respectively.  It follows from the preceding paragraph that
  for every \(k\), either \(p\) and \(q\) have the same sign in the
  intervals \((a_k,a_{k+1})\), \((b_k,b_{k+1})\), or \(b_k\) is a double
  root of \(p\).  Since \(c_k\in [a_k,a_{k+1}]\cap[b_k,b_{k+1}]\), it
  follows that, for \(0<\lambda<1\), \( p_\lambda\) sign-interlaces \(C\),
  so it has \(n\) distinct real roots.
  
  Conversely, suppose \(p_\lambda\) has \(n\) real roots for all
  \(0\leq\lambda\leq1\). We claim that for \(0<\lambda<1\), no root of
  \(p\) or \(q\) is a root \(p_\lambda\).  Indeed, if it was not the case,
  that root would be a common root of \(p\) and \(q\), contradicting the
  hypothesis that they are coprime.  Now, it is well known that the roots
  of \(p_\lambda\) depend continuously on \(\lambda\), so each root of
  \(q\) starts a path in the real line that ends on a root of \(p\). Since
  that path contains no root of \(p\) or \(q\), it must be a segment, and
  by the same argument, these segments are interiorly disjoint. Any choice
  of points separating these segments will be a common interlace.
\end{proof}

Whereas the convex combination form  is more convenient for the proof,
the following will be more convenient for our purposes:

\begin{pro}\label{alphacomb}
  Let \(p,q\) be coprime polynomials of degree \(n\) and positive
  leading coefficient; suppose also that \(q\) has \(n\) distinct
  real roots. Then \(p+\alpha q\) is real rooted for all
  \(\alpha>0\) if and only if \(p\) is real rooted, and \(p\) and
  \(q\) have a common interlace.  In this case, then \(p+\alpha q\)  has
  \(n\) distinct roots for all \(\alpha>0\).
\end{pro}
\begin{proof}
  Defining \(\lambda=\frac\alpha{1+\alpha}\) land us back in  \cref{lincomb}.
\end{proof}

The following nonconstructive existence result is immediate:  
\begin{cor}\label{notinter}
  Let \(p,q\) be coprime real polynomials of degree \(n\) and positive
  leading coefficient, each with \(n\) real roots.  If they have no common
  interlace, then there is a positive real \(\alpha\) such that
  \(p+\alpha q\) has a complex, nonreal root.
\end{cor}

We will be interested on the dynamics of the set of real roots of
\palpha, viewing \(\alpha\) as the time variable.  There is a
ready tool for that, at least for counting those roots, using the
classical subresultants.  We refer the reader to
\cite[Chapter 4]{BPR} for a full account and the definitions we
use here; they are not going to be used beyond this paragraph,
and are certainly not central to this paper. Consider a two
variable polynomial
\(p(\alpha,x)\in\R[\alpha][x]\) as a family of polynomials on
\(x\) of generic degree \(n\) parameterized by \(\alpha\). Let
\(\text{sDisc}(p)=\left(\text{sDisc}_0(p),\ldots,\text{sDisc}_{n-1}(p)\right)\)
be the sequence of subdiscriminants of \(p\), each one a
polynomial on \(\alpha\).  It follows from \cite[Thm. 4.33]{BPR}
that, given \(\alpha_0\), the number of real roots of
\(p(\alpha_0,x)\) can be gleaned from the sequence of signs of
the sequence
\((\text{sDisc}_0(p)(\alpha_0),\ldots,\text{sDisc}_{n-1}(p)(\alpha_0))\).
As the sign of a polynomial only changes around a zero of odd
multiplicity, we can observe how the number of real roots of
\(p(\alpha,x)\) changes by the following procedure, which can be
easily implemented, modulo the usual difficulties of representing
real numbers:
\begin{enumerate}
    \item Compute the sequence \(\text{sDisc}(p)\).  All terms are
  determinants of matrices which can be simply written given \(p\).
    \item Let \(\alpha_0<\alpha_1<\ldots<\alpha_m\) be the collection of
  all roots of odd multiplicity of members \(\text{sDisc}(p)\).  Let
  \(\alpha_{-1}=-\infty\), \(\alpha_{m+1}=\infty\).
    \item Within successive \(\alpha_j\), all \(p(\alpha,x)\) have the same
  number of real roots.
\end{enumerate}

\section{Palindromic polynomials}
\label{sec:palindr-polyn}

We recall some definitions.  Let
\(p(x)=p_rx^r+p_{r+1}x^{r+1}+\cdots+p_sx^s\) be a nonzero
polynomial, with \(p_r,p_s\neq 0\); the
\emph{darga}\glsadd{darga} of \(p\) is \(\darga{p}=r+s\).  We say
that a polynomial \(p\) of darga \(n\) is
\emph{self-inversive}\glsadd{self-inversive} if
\(x^n\bar{p}(\frac{1}{x})=p(x)\); equivalently, if for all \(j\),
\(p_j=\overline{p_{n-j}}\). If \(p\) is self-inversive and has
real coefficients, \(p_j=p_{n-j}\) for all \(j\), and we prefer
the term \emph{palindromic}\glsadd{palindromic}.  If the darga
equals the degree, we say that \(p\) is \emph{full}\glsadd{full},
and it is \emph{trim}\glsadd{trim} otherwise.  We denote the set
of real palindromic polynomials of darga \(n\) by
\glsadd{paln}\paln{n}, and the trim ones by
\glsadd{trimn}\Trim{n}; the corresponding sets of self-inversive
polynomials will be \glsadd{cpaln}\Cpaln{n} and
\glsadd{ctrimn}\CTrim{n}. Note that all those sets are just short
the zero polynomial of being \textbf{real} vector spaces, and it
is convenient to complete them so, with a \(0\) vector of each
darga.

\begin{pro}\label{palroots}
  Let \(p(x)\) be a full palindromic polynomial of degree \(n\). Then:
  \begin{enumerate}
      \item \label{palrootsa} If \(p(x)=q(x)r(x)\) and \(q(x)\) is palindromic, then so is
    \(r(x)\).
      \item \label{palrootsb}If \(\theta\) is a root of \(p\), so are \(\bar{\theta}\),
    \(1/\theta\), \(1/\bar{\theta}\).  These four roots are distinct, unless
    \(\theta\) lies in the circle or is real.
      \item \label{palrootsc}If \(p\) is circle rooted, the only possible real roots are
    \(1\), with even multiplicity, and \(-1\), whose multiplicity has the
    same parity as \(n\).
  \end{enumerate}
\end{pro}
\begin{proof}
  This is mostly immediate from the definition. We only comment about the
  multiplicities of \(\pm 1\).  We can write \(p(x)=(x-1)^a(x+1)^bq(x)\),
  where \(q\) has no real roots. By part (2), \(q\) is palindromic, and by
  part (1), so is \((x-1)^a(x+1)^b\); this clearly requires that \(a\) be
  even.  Also by part (2), \(q\) has even degree, so the assertion about
  \(b\) follows.
\end{proof}

For self-inversive polynomials we can say slightly less, but part \Item{sirootsb}
is the main reson for concentrating on self-inversive polynomials.
\begin{pro}\label{siroots}
  Let \(p(x)\) be a full self-inversive polynomial of degree \(n\). Then:
  \begin{enumerate}
      \item \label{sirootsa} If \(p(x)=q(x)r(x)\) and \(q(x)\) is self-inversive, then so is
    \(r(x)\).
      \item \label{sirootsb}If \(\theta\) is a root of \(p\), so is
    \(1/\bar{\theta}\).  These two roots are distinct, unless
    \(\theta\) lies in the circle.
      \item \label{sirootsc} If \(n\) is odd, -1 is a root of \(p\).
      \item \label{sirootsd} If \(\theta\in\un\), then \(p(\theta x)\) is self-inversive.
  \end{enumerate}
\end{pro}
\begin{proof}
  Parts \Item{sirootsa},\Item{sirootsb} and \Item{sirootsc} follow
  easily from the algebraic definition of self-inversive,
  while \Item{sirootsd} comes from looking at coefficients.
\end{proof}

Darga \(n\) palindromic polynomials form a subspace of the much larger
space of polynomials of degree at most \(n\), and, in this sense, the
coefficients are the coordinates in the canonical basis of the large space,
which is isomorphic to \(\R^{n+1}\).  The space \paln{n} is much
smaller, and it will be useful to use  the
\emph{\palsym-basis}\glsadd{pbasis}
\begin{equation}
\conj{\pal{n,j}(x)}{j=0,1,\ldots,\floor*{n/2}},
\end{equation}
 where
\begin{equation}  \label{ibasis_polynomials}
\pal{n,j}(x) :=x^j+x^{n-j}. 
\end{equation}
Notice that
\(\pal{2n,n}(x)=2x^n\).  
To summarize, if \(p(x)=\sum_{j=0}^{n}p_jx^j\), then we define its
\emph{\palsym-representation}\glsadd{prep} to be
\begin{equation} \label{sigma-coefficients}
p(x)=\sum_{j=0}^{\floor*{n/2}}\sig{p}_j\pal{n,j}(x),
\end{equation}
where each \emph{\palsym-coefficient} \(\sig{p}_j\) equals
\(a_j\), except if \(n=2j\), in which case \(\sig{p}_j=p_j/2\).
The space \Trim{n} of trim polynomials is spanned by
\(\pal{n,1},\ldots,\pal{n,\lfloor n/2 \rfloor}\).  The parity of
\(n\) will be of relevance in many results yet to appear; as
above, the middle coefficient, for even \(n\), often deserves
special attention.  So, to simplify some statements, we will
adopt the convention
\begin{itemize}
    \item []
  \emph{any condition involving the middle coefficient
    \(p_{n/2}\) implicitly carries the proviso ``if \(n\) is
    even'', and is vacuously true for odd \(n\).}
\end{itemize}

For instance, we can simply write \(\sig{p}_{n/2}=\frac12p_{n/2}\).

Any complex polynomial can be written as \(p(x)=p^R(x)+ip^I(x)\),
where \(p^R\) and \(p^I\) are real polynomials.  If \(p\) is
self-inversive, then \(p^R\) is palindromic, while
\(p^I(x)=-x^n\bar{p}^I(1/x)\), that is, \(p^I\) is
\emph{anti-palindromic}.  A natural basis for anti-palindromic
polynomials is
\begin{equation}
        \conj{\hat\palsym_{n,j}(x)}{0\leq j < n/2},  
\end{equation}
where
\begin{equation}  \label{basis_polynomials}
\hat\palsym_{n,j}(x) :=x^j-x^{n-j}. 
\end{equation}
We extend the \palsym-representation to self-inversive polynomials, writing
\begin{equation} \label{sigma-inv-coefficients}
p(x)=\sum_{j=0}^{\floor*{n/2}}\sig{p}_j\pal{n,j}(x) + i\!\!\!\!\!\sum_{j=0}^{\floor*{(n-1)/2}}\hat{p}_j\hat\palsym_{n,j}(x).
\end{equation}

The notion of angle-interlacing will allow us to transfer some of
the results from \cref{sec:interlacing} to \Cpaln{n}.
Given sequences of nonzero complex numbers, with nondecreasing
arguments in the interval \([0,2\pi)\), \(A=\{a_1, a_2,\cdots,\)
\( a_n\}\) and \(B=\{b_1, b_2,\cdots, b_n\}\) of same length, we
say that they \emph{angle-interlace} if they can be merged into
\(\{c_1, c_2,\cdots, c_{n+m}\}\), where the \(c_j\)'s alternate
between \(a\)'s and \(b\)'s and arguments are still
nondecreasing; if all arguments are different, the sequences
\emph{strictly
  angle-interlace}\glsadd{strictlyangleinterlaces}. Geometrically,
in the strict case, we think of the rays emanating from the
origin through the points in \(A\) and dividing the complex plane
into \(n\) sectors, with one element of \(B\) in each sector.  In
the context of angle-interlacing, a polynomial will be considered
a proxy for the adequate sequence of its roots with multiplicity,
so we may speak about polynomials angle-interlacing.

We will use the exponential map \(\exp(x)=\eipi{x}\), from the real
interval \([0,1)\) to the unit circle \(\ucirc=\conj{z\in\C}{|z|=1}\).
Note that for every \(x\), \(\exp(-x)\) is the conjugate of \(\exp(x)\).
For the remaining of this section, let
\(\In=\{0,\frac{1}{n},\frac{2}{n},\ldots,\frac{n-1}{n},1\}\), so \(\un=\exp(\In)\).

\begin{Example}\label{ex:xnplus1}
  For any \(n\geq 2\), \(x^n+1\) angle-interlaces \un.  Indeed, its roots
  are 
  \conj{\exp\left(\frac{2j+1}{2n}\right)}{j=0,1,\ldots,n-1}.
  % , whose arguments are\\
%   \conj{\frac{2j+1}{2n}2\pi}{j=0,1,\ldots,n-1}.
\end{Example}

Interlace and circle roots come together in:
\begin{pro}\label{angle-inter}
  If a full self-inversive polynomial strictly angle-interlaces a
  set of \(n\) points, then it is circle rooted.
\end{pro}
\begin{proof}
  If \(\w\) is a root not in the circle, then, by \cref{siroots}
  so is \(1/\bar{\w}\), which has the same argument.  So, the polynomial
  cannot strictly angle-interlace anything.
\end{proof}

Recalling the definitions of interlace and circle number, we have the
following consequence.
\begin{pro}\label{cnleil}
        \(\cn{p}\leq \il{p}\).  
\end{pro}

Among the sets of \(n\) complex numbers, it is fair to say that the \(n\)th
roots of unity stand out, for many reasons.  Here they play a special role
from the outset, as shown in the following two simple facts:

\begin{pro}
  If \(p(x)\) is self-inversive, of darga \(n\), then for every
  \(\w\in\un\), \(p(\w)\in\R\).
\end{pro}
\begin{proof}
  Indeed, as
  \(p(\w)=\w^n\bar{p}(1/\w)=\bar{p}(\bar{\w})\),
  \(p(\w)\in\R\).
\end{proof}

It is worthwhile to register that most work can be cut by half in the real case:

\begin{pro}
  \label{wnvn}
  If \(p(x)\) is palindromic of darga \(n\), then
  \[\conj{p(\w)}{\w\in\un}=\conj{p(\w)}{\w\in\vn}.\]
\end{pro}
\begin{proof}
  If \(\w\in\un\backslash\vn\), then
  \(\w^{-1}\in\vn\) and \(p(\w^{-1})=p(\w)\).
\end{proof}

The following is an outgrowth of an idea used in \citet{DR2015}.

\begin{teo}\label{interlace}
  Let \(p(x)\) be a full self-inversive polynomial of degree \(n\).  Then,
  \(p(x)\) strictly angle-interlaces \un if and only if all \(p(\w)\), \(\w\in \un\),
  have the same sign (and it is enough to consider \(\w\in \vn\) if \(p\) is palindromic).
\end{teo}
\begin{proof}
  %Multiplying by \(-1\), if necessary, we assume that
  %\(p(1)\geq 0\).
  Looking at \(p(x)\) as a function, we can
  rewrite self-inversiveness as
  \(x^{-n/2}p(x)=x^{n/2}\bar{p}\left(\frac{1}{x}\right)\).
  Hence, for every real \(t\),
  \[
    \exp\left(-\frac{nt}{2}\right)p(\exp{t})=\exp\left(\frac{nt}{2}\right)\bar{p}(\exp(-t)).
  \]
  As the two sides are conjugate,  the function
  \(f(t)=\exp\left(-\frac{nt}{2}\right)p(\exp(t))\) is real valued.
  Moreover, for \(0\leq t<1\), \(f(t)\) is  zero if
  and only if \(t\) is the argument of a zero of \(p\) in the unit
  circle, so \(f(t)\) has at most \(n\) zeros in the unit interval.

  In particular, we have, for \(0\leq j\leq n\),
  \[
        p\left(\exp\left(\frac{j}{n}\right)\right)=(-1)^jf\left(\frac{j}{n}\right).
  \]
  This shows that \(p\left(\exp\left(j/n\right)\right)\) is real
  (as expected), and \(f(x)\) sign-interlaces \In if and only if
  %\(p\left(\exp\left(j/n\right)\right)>0\) for all \(j\), as we
  %forced \(p(\exp(0))\geq 0\).  This is the same as saying that
  all \(p(\w)\), with \(\w\in \un\), have the same sign.

  So, if the latter condition holds, as \(f(x)\) has at most \(n\) zeros in
  the interval, by \cref{signinter}, \(f(x)\) interlaces \In, hence
  \(p(x)\) angle-interlaces \un.

  On the other hand, if \(p(x)\) strictly angle-interlaces \un, \(f(x)\)
  strictly interlaces \In, so, \(p(\w)\) has one fixed sign for every \(\w\in \un\).
\end{proof}

\begin{cor}\label{cosinter}
  Let \(p(x)\) be a full palindromic polynomial of degree \(n\),
  with \palsym-coefficients \(\sig{p}_0,\sig{p}_1,\ldots\), and
  suppose that \(p(1)>0\).  Then \(p\) angle-interlaces \un if
  and only if %for every \(0\leq j\leq \floor*{\frac{n}{2}}\),
  \[
    \text{for every}\ 0\leq j\leq \floor*{n/2},
    \quad \sum_{k=0}^{\floor*{\frac{n}{2}}} \sig{p}_k\cos\frac{2\pi j k}{n}\geq 0.
    \]
  \end{cor}
  \begin{proof}
    Just notice that \(\pal{n,k}\left(\eipi{jk;n}\right)=2\cos\frac{2\pi j k}{n}\cdot\)
  \end{proof}

  A similar statement for self-inversive polynomials, derived
  from the \palsym-representation, will find its way in
  \cref{siilnumber}, and is left for the reader.
 
\section{The interlace number}
\label{sec:interlace-number}

Recall that for a trim polynomial \(p\) of darga \(n\) and
\(\alpha\in\R\), we define
\[
  \palpha(x)\glsadd{palpha}=\alpha \pal{0}+p= \alpha(x^n+1)+p(x).
\]  
The coefficients of \(p\) will generally be \(p_0,p_1,\ldots\);
either the presence or absence of a variable or the context will
hint on how to interpret the subscript.  The
\emph{interlace number}\glsadd{ilnum} of a trim self-inversive
polynomial \(p\) is
\[
    \il{p}=\inf\conj{\alpha>0}{\palpha(x)\ \text{strictly angle-interlaces \un}}.
\]

\begin{pro}\label{ilscaling}
  Let \(p\) be a trim self-inversive polynomial. Then
  \begin{enumerate}
      \item \label{ilscalinga} (\textsc{Linear Scaling}) For every real \(\lambda>0\),
    \(\il{\lambda\,p}=\lambda\,\il{p}\).
      \item \label{ilscalingb} (\textsc{Exponent Scaling}) For every positive integer \(r\),
    \(\il{p(x^r)}\!=\!\il{p(x)}\).
      \item \label{ilscalingc} If \(\theta\in\un\),
    \(\il{p(\theta x)}=\il{p(x)}\).  In particular, if \(p\) has even darga,
    then \(\il{p(-x)}=\il{p(x)}\).
  \end{enumerate}
\end{pro}
\begin{proof}
  Part (a) is immediate from the definition.  For part (b),
  although a direct argument is not hard, we use 
  \cref{ilnumber}. Let \(n=\darga{p}\) and let
  \(q(x)=p(x^r)\). Note that a complex \(\w\in U_{rn}\) if and
  only if \(\w^r\in \un\), and in that case,
  \(\conj{q(\w)}{\w\in U_{rn}}=\conj{p(\w)}{\w\in\un}\).  It
  follows from \cref{ilnumber}(a) that
  \(\il{q}=\il{p}\). For part (c), notice that \(\un=\theta\un\)
  for \(\theta\in\un\), so \(p_\alpha(\theta x)\) interlaces \un
  if and only if so does \(p_\alpha(x)\); the result follows from
  the fact that \((p(\theta x))_\alpha=p_\alpha(\theta x)\).
\end{proof}

The linear scaling  suggests that maybe the interlace number should be
replaced by a normalized version, dividing it by another function of \(p\)
that scales linearly.  Indeed, \cref{LL} suggests the use of
\(\frac12||p||_1\) as normalizing function, which would restrict the
interlace number to the interval \((0,\!1]\).  At this point, a
normalization looks more like a nuisance than an advantage (introducing
annoying denominators when dealing with integer polynomials, for instance),
so we will keep the interlace number as is.

The interlace number can be described precisely:
\begin{teo}\label{ilnumber}\textsc{[Interlace formula]}
  If \(p\) is a trim palindromic polynomial of darga \(n\), then 
  \begin{enumerate}
    %\text{a)}\quad\il{p}  &= -\frac12\min\conj{p(\w)}{\w\in\vn}\\
    %\text{b)}\quad\il{p} &= -\frac12\min\conj{p\left(e^\frac{2\pi ij}{n}\right)}{j=0,1,\ldots,\floor*{\frac{n}{2}}}\\
      \item \label{ilnw} \(\quad\il{p} = \frac12\max\conj{-p\left(\w\right)}{\w\in\vn}\). \\
      \item \label{ilnc}
    \(\quad\il{p} =
    \max\conj{\sum_{k=1}^{\floor*{n/2}}
      -\cos\frac{2\pi j
        k}{n}\cdot\sig{p}_k}{j=0,1,\ldots,\floor*{\frac{n}{2}}}\ccomma\)
    where the \(\sig{p}_j\) are the \palsym-coefficients of
    \(p\).
  \end{enumerate}
\end{teo}
\begin{proof}
    Let \(\mu=\min\conj{p(\w)}{\w\in\vn}\), and let
  \(\alpha=-\frac12\mu\). Then, for any \(\w\in\un\),
  \(\palpha(\w)= \alpha(\w^n+1)+p(\w)=-\mu+p(\w)\geq 0\).
  In particular, if \(\w_0\in\vn\) is such that \(p(\w_0)=\mu\),
  then \(\palpha(\w_0)=0\), which shows that \palpha does not strictly
  angle-interlace \un, hence \(\il{p}\geq\alpha\).  To show equality,
  consider any \(\beta >\alpha\). Then, for \(\w\in\vn\)
  \(p_\beta(\w)= (\beta-\alpha)\pal{n,0}(\w)+\palpha(\w)\geq
  2(\beta-\alpha)>0\), so, by \cref{interlace}, \(p_\beta\) strictly
  angle-interlaces \un.

  The cosine formula follows in the same way as Corollary \ref{cosinter}.
\end{proof}

For self-inversive polynomials, we do not have the convenience of
checking just half the roots.  The preceding argument also yields:

\begin{teo}\label{siilnumber}\textsc{[Interlace formula]}
  If \(p\) is a trim self-inversive of darga \(n\), then 
 \begin{enumerate}
      \item \label{iilnw} \(\quad\il{p} = \frac12\max\conj{-p\left(\w\right)}{\w\in\un}\)\\
      \item \label{iilnc} \(\quad\il{p} = \max_{0\leq j<n}-\left(\sum_{k=1}^{\floor*{n/2}} \cos\frac{2\pi j k}{n}\cdot\sig{p}_k+\sum_{k=1}^{\floor*{(n-1)/2}} \sin\frac{2\pi j k}{n}\cdot\hat{p}_k\right).\)
  \end{enumerate}
\end{teo}

As we will quote \cref{ilnumber} and \cref{siilnumber}
extensively in what follows, we will just remember them as the
\Ifor.\glsadd{Ifor}

A \(\w\in\un\) which yields the maximum in the \Ifor will be
called an \emph{interlace cert}\glsadd{icert}.  For a palindromic polynomial,
we will restrict the term interlace certs to refer only to those
in \vn.  We will see several examples of polynomials for which
one can precisely pinpoint an interlace cert (all certs, in some
cases).  Moreover, the next section shows a geometrical
classification of polynomials by their set of interlace
certs.  %From the proof of the \Ifor we extract:

The ineffable ugliness of the formula in
\cref{siilnumber}.\Item{iilnc} is one good reason for us to
forego presenting things in full for self-inversive polynomials.
So, as mentioned in the Introduction, we will stick mostly to
palindromic polynomials and invite the reader to restate the
simpler results ahead for self-inversive ones - usually,
substituting \un for \vn works.

\begin{pro}\label{interlacecert}
  Given a palindromic \(p\), let \(\alpha\) be a positive real
  and \(\theta\in\vn\).  Then, \(\alpha=\il{p}\) and \(\theta\)
  is an interlace cert if and only if
  \begin{enumerate}
      \item \(\palpha(\omega)\geq 0\) for all \(\omega\in\vn\), and
    
      \item \(\palpha(\theta)=0\).
  \end{enumerate}
\end{pro}
\begin{proof}
  Clearly, we have that
  \(\palpha(\omega)=2(\alpha+\frac12p(\omega))\).  If
  \(\alpha=\il{p}\) and \(\theta\) is an interlace cert, then (1)
  is immediate from the \Ifor, and (2) just restates the
  definition of interlace cert.  For the converse, just observe
  that (1) implies, via the \Ifor, that \(\alpha\geq\il{p}\), and
  on the other hand, (2) implies that \(\alpha\leq\il{p}\).
\end{proof}
% Notice that the \Ifor expresses the interlace number as the
% maximum of a finite set.  Moreover, provided that \(p\) has
% integer \palsym-coefficients, \(\il{p}\) is an algebraic integer;
% polynomials for which it is integer are explored in
% \cref{sec:interl-rati-polyn}.  The following companion to
% \cref{interlacecert} is also useful:

% \begin{pro}\label{ilcert}
%   If \(p\) is a trim palindromic polynomial of darga \(n\), then
%   \begin{equation*}
%     \il{p}= \max\conj{\alpha\in\R}{\palpha\ \text{has a root in \vn}}.
%   \end{equation*}
% \end{pro}
% \begin{proof}
%   If \(\omega\in\vn\), \(\palpha(\omega)=2\alpha+p(w)\), so
%   \(\palpha(\omega)=0\) if and only if
%   \(\alpha=-\frac12p(\omega)\).  The result now follows from the
%   \Ifor.
% \end{proof}

% This is also a direct consequence of the \Ifor:
% \begin{cor}
%   If \(p\) has integer \palsym-coefficients, then \(\il{p}\) is an algebraic integer. 
% \end{cor}

Substituting the definition of the interlace number, the \Ifor
can be restated as a minimax expression:

\begin{cor}  \label{minmaxil}
  If \(p\) is a trim palindromic polynomial of darga \(n\), then
  \[
    \min\conj{\beta}{\palpha\ \text{angle-interlaces \un for all}\ \alpha\geq\beta}=
    \max\conj{-\lilhalf p\left(\w\right)}{\w\in\vn}\cdot
  \]  
\end{cor}

The discrete Fourier transform of degree \(n\) of the sequence of
coefficients of \(p\) is the sequence
\(p\left(\tn^j\right)_{j=0,\ldots,n-1}\), just a listing of
\(\conj{p\left(\w\right)}{\w\in\un}\).  So, the \Ifor says that
for any self-inversive polynomial,\il{p} is the largest
coefficient of the discrete Fourier transform of degree \darga{p}
of \(-p\).  One interesting consequence of the \Ifor is that,
provided that \(p\) has integer \palsym-coefficients, \(\il{p}\)
is an algebraic integer; polynomials for which it is a
\emph{rational} integer, as well as some more relations with the
Fourier transform are explored in \cref{sec:interl-rati-polyn}.

The \Ifor gives several lower bounds for the interlace number of
a polynomial as linear combinations of its coefficients. Those
combinations use multipliers which are usually irrational and may be
unwieldy in an abstract setting.  \Cref{rambound} gives lower
bounds with integer multipliers; we postpone its presentation, as
it is not needed in the sequel and it requires the introduction of
some specific tools.

% All elements of \vn can be interlace certs:

% \begin{pro}\label{uniqint}
%   For every \(n\geq 3\) and any \(\theta\in\vn\), there exists a
%   polynomial whose only interlace cert is \(\theta\).
% \end{pro}
% \begin{proof}
%   Let \(p(x)=-\sum_{j=1}^{n-1}(\theta^j+\bar\theta^j)x^j\).
%   Then,
%   \(-p(\theta)=\sum_{j=1}^{n-1}\theta^{2j}+n-1=\frac{\theta^2-\theta^{2n}}{1-\theta^2}+n-1=n-2\),
%   the last equality following from \(\theta^n=1\).  On the other
%   hand, if \(\w\in\vn\backslash\{\theta\}\), both
%   \(\theta\w,\bar\theta\w\in\un\backslash\{1\}\), so
%   \(-p(\w)=\sum_{j=1}^{n-1}(\theta\w)^j+\sum_{j=1}^{n-1}(\bar\theta\w)^j=-2\),
%   so \(\il{p}=\frac{n-2}2\), and \(\theta\) is the only interlace
%   cert.
% \end{proof}

% Some simple consequences of the Theorem:

% \begin{cor}
%   \label{minusx}
%   If \(p\) has even darga, then \(\il{p(-x)}=\il{p(x)}\).
% \end{cor}
% \begin{proof}
%   In this case, \(-\un=\un\), and the result follows from
%   \cref{wnvn} and \cref{ilnumber}.
% \end{proof}

The literature provides a considerable list of statements of conditions on
the coefficients of a full palindromic polynomial that imply it interlaces
\un. In many cases, they can be rewritten as a bound for the leading
coefficient, in terms of the other ones. Then it is routine to reinterpret
the statement as a bound for the interlace number for the trimmed
polynomial.  We present a couple of these.

\begin{teo}\label{LL}
  \emph{(\citet{LL2004})}
  If \(p\) is self-inversive, then
    \[
        \il{p}\leq \frac12\sum_{j=1}^{n-1}|p_j|=\sum_{j=1}^{\floor*{n/2}}|\sig{p}_j|,
      \]
      where the last equality is for palindromic \(p\). 
\end{teo}
\begin{proof}
  As \(|-p(\w)|\leq\sum_{j=1}^{n-1}|p_j|\) for every
  \(\w\in\un\), the result follows from the \Ifor.
\end{proof}

The same authors extended that result
\cite{LL2007},\cite{LL2009}, which were further generalized by
\citet{Kw11}, in the real case.  In what follows, the
\emph{median} of a sequence of real numbers is the element that
would be in position \(\left\lfloor\frac{n}2\right\rfloor\) if
the elements were ordered as \(a_1,\ldots,a_n\).

\begin{teo}\label{kwon}
  \emph{(\citet{Kw11})} Suppose that \(p\in\Trim{n}\), and denote by
  \(m(p)\) the median of the sequence of coefficients of \(p\).
  Then, if \(p(1)\geq 0\),
      \begin{equation}
        \label{eq:kwon}
        \il{p}\leq \frac12\left(m(p)+\sum_{j=1}^{n-1}\left|p_j-m(p)\right|\right).
      \end{equation}
\end{teo}

The proof of this and of the following corollary
will follow shortly.  Notice that the condition \(p(1)\geq 0\)
cannot be entirely dispensed with, as witnessed by \cref{pofone}.
The following special case, while more cumbersome, occasionally
simplifies evaluating the upper bound (see
\cref{increasing-upper} for an application):

\begin{cor}\label{kwon-simple}
  Denote by \(m(p)\) the median of the sequence of coefficients
  of \(p\in\Trim{n}\).  Let \(M=\conj{j}{1\leq j< \frac{n}{2}, p_j<m(p)}\).
  Then, provided \(p(1)\geq 0\), and
  \(p_{n/2}\geq m(p)\), then
      \begin{equation}
        \label{eq:kwon-simple}
        \il{p}\leq \frac12p(1)-2\sum_{j\in M}p_j-\left(\floor*{\frac{n-1}{2}}-2|M|\right)m(p).
      \end{equation}
\end{cor}

We will provide a simple proof of \cref{kwon}.  On the way, we
show the following formalization (in a sense) of Joyner's
``metatheorem'' \cite{J2013}:
\begin{lem}\label{shifta}
  Suppose that \(1\) is not an interlace cert
  of \(p(x)\in\Trim{n}\).  Then, for any real \(a\),
  \[\il{p(x)+a(x+x^2+\cdots+x^{n-1})}\geq\il{p}+\frac{a}{2},\]
  with equality if \(a\geq -\frac{2\,\il{p}+p(1)}{n-1}\cdot\)
\end{lem}
\begin{proof}
  Let \(r(x)=x+x^2+\cdots+x^{n-1}\) and \(q(x)=p(x)+a\,r(x)\).  For every
  \(\gamma\in\vn, \gamma\neq 1\), \(r(\gamma)=-1\), so
  \(q(\gamma)=p(\gamma)-a\); on the other hand, \(q(1)=p(1)+(n-2)a\).  By
  hypothesis, \(\il{p}=-\frac12p(\w)\) for some
  \(\w\in\vn, \w\neq 1\), and it follows that
  \(\il{q}=\max \left\{-\frac12q(\w),-\frac12q(1)\right\}\).  Hence,
  \(\il{q}\geq -\frac12q(\w)=
  -\frac12\left(p\left(\w\right)-a\right)= \il{p}+\frac{a}{2}\).  The
  last condition on \(a\) is sufficient to guarantee that
  \(q(\w)\leq q(1)\), whence the equality.
\end{proof}

\begin{proof}[Proof of \cref{kwon}]
  For any real \(a\), let \(q(x)=p(x)-a(x+x^2+\cdots+x^{n-1})\).  By
  \cref{shifta}, \(\il{q}\geq\il{p}-\frac{a}2\), hence
  \(\il{p}\leq \frac{a}2+\il{q}\), which, by \cref{LL}, is bounded above by
  \(f(a)=\frac12\left(a+\sum_{j=1}^{n-1}|p_j-a|\right)\).  It follows that
  \(\il{p}\leq \inf_{a\in\R}f(a)\).  Since \(f\) is a piecewise linear
  function, it actually attains a minimum, at a breakpoint. That is, the
  point of minimum is \(a=p_j\), for some \(j\), and it is an easy exercise
  to show that \(a=m(p)\).
\end{proof}

\begin{proof}[Proof of Corollary \ref{kwon-simple}]
  We introduce some additional notation. Let \(\mu=m(p)\),
  \(P=\conj{j}{1\leq j <\frac{n}{2}, p_j>\mu}\), \(s_-=\sum_{j\in M}p_j\),
  \(s_+=\sum_{j\in P}p_j\) and let \(e(n)\) be \(1\) if \(n\) is even,
  \(0\) otherwise.  Just from those definitions, we note:
  \begin{equation}
    \label{eq:p-of-one}
    p(1)=2s_+ +2s_- + e(n)p_{n/2} +(n-1-e(n)-2|P|-2|M|)\mu.
  \end{equation}
  We can rewrite \eqref{eq:kwon} as
  
    \(\displaystyle\il{p}\leq \frac12\left(\mu +2\sum_{j\in P}(p_j-\mu)+
             2\sum_{j\in M}(\mu-p_j)+e(n)|p_{n/2}-\mu|\right)\)
  
  \begin{align*}
        \phantom{\il{p}} & = s_+ -s_- +\frac12e(n)p_{n/2}+
             (|M|-|P|)\mu &&\text{with \(p_{n/2}\geq\mu\)}\\
            & = \frac12p(1)-2s_- -\left(\frac{n-e(n)-1}{2}-2|M|\right)\mu
              &&\text{substitute \(s_+\) from \eqref{eq:p-of-one}}\\
     & = \frac12p(1)-2s_- -\left(\floor*{\frac{n-1}{2}}-2|M|\right)\mu,
  \end{align*}
  whence the result follows.
\end{proof}

If one wants to consider a structured or parameterized family of polynomials
to obtain a general estimate for the interlace number, the exact formulas
may be hard to apply.  The bounds derived from Lakatos, Lozonczi and Kwon's
results are useful for this purpose.  However, they have a feature that is,
at the same time, elegant and a weakness: they only depend on the multiset
of coefficients of the given polynomial.  In some cases, that yields a weak
bound.

\begin{Example}\label{ex:ilpal}
  Let us consider the interlace number for the basis polynomial
  \(\pal{n,k}=x^k+x^{n-k}\). Theorems \ref{LL} and \ref{kwon}
  coincide in giving the bound 1 for all cases.  Now we compute
  the precise value. Let \(d=\gcd(n,k)\); by
  \cref{ilscaling}\Item{ilscalingb},
  \(\il{\pal{n,k}}=\il{\pal{n/d,k/d}}\), so we need only consider
  the case \(\gcd(n,k)=1\).  By \cref{ilnumber}\Item{ilnc}, we
  need to find \(j\) such that \(\cos\frac{2\pi j k}{n}\) is
  minimal.  If \(n\) is even (and \(k\) odd), \(j=\frac{n}2\)
  works, yielding \(\il{\pal{n,k}}=1\), so the bound is tight;
  moreover, \(-1\) is an interlace cert of \pal{n,k}.  If \(n\)
  is odd, we choose \(j\) such that
  \(jk\equiv \pm\lfloor n/2\rfloor\mod n\).  It follows that
  \(\il{\pal{n,k}}=-\cos\frac{2\pi\lfloor
    n/2\rfloor}{n}=\cos\frac{\pi}{n}\). So, the bound is best
  possible for large \(n\), but not attained. The smallest
  interlace number for those polynomials is attained with
  \(n=3\), so we have that for all positive \(k\),
  \(\il{\pal{3k,k}}=\frac12\).
\end{Example}

\begin{Example}\label{ex:geom}
  For \(n\geq 2\) we define the \emph{geometric polynomial}\glsadd{geompol}
  \(\geom_n(x)=x+x^2+\cdots+x^{n-1}\).  It will occur quite
  often later, either as part of a construction or as a member of
  some class, so we give it some attention.  The argument of
  \cref{shifta} works with \(p=0\), so we get
  \(\il{\text{ge}_n(x)}=\frac12\cdot\).
\end{Example}

\begin{Example}\label{ex:d6}
  Let \(n=6\), and consider the polynomial
  \(q(x)=172\pal{1} + 100\pal{2} + 198\pal{3}\).  Here all linear forms in
  \cref{ilnumber}\Item{ilnc} have coefficients \(\pm\frac12,\pm 1\), and we
  easily compute \(\il{q}=171\).  On the other hand,
  \(p(x)=100\pal{1} + 172\pal{2} + 198\pal{3}\) has the same multiset, but
  \(\il{p}=135\).  The bound given by \cref{kwon} is 171, exact for
  \(q\).  It gets worse: as per \cref{ilscaling},
  \(\il{p(x^2)}=\il{p(x)}\), but for \(p(x^2)\) the the median coefficient
  is 0, and the bound of \cref{kwon} coincides with that of \cref{LL}, which
  is 371 -- more than twice the actual value.
\end{Example}

\section{The Fan of Interlace Certs}
\label{sec:fan-interlace-certs}

The \Ifor \Item{ilnc} for real palindromic polynomials leads to a
geometric classification of those polynomials into a complete
simplicial fan.  In order to describe it, we present in the first
part of this section some general facts about the normal fan and
the polar of a simplex.  These are bespoke versions of general
constructions in the theory of convex polyhedra;
\cite{BG,ewald,ziegler} are useful references.  The second part
is the specific study of the fan implicitly described in
\cref{ilnumber}; the focus is on palindromic polynomials.  The
denouement in the third part explains why the whole space of
self-inversive polynomials is less interesting from the viewpoint
of this section.

\subsection{Fans and their automorphisms}
\label{sec:fans-their-autom}

It will be convenient to distinguish between a \(d\)-dimensional
real vector space \(V\) and its dual \st{V}; the usual treatment
fixes a basis of \(V\) and identifies \(V\) and \st{V} using the
dual basis. It will be convenient to start with a simplex in dual
space: let \(L=\{\ell_0,\ell_1,\ldots,\ell_d\}\) be dual points
(that is, linear functionals on \(V\)) whose convex hull
\(\hat{L}\) is a \(d\)-simplex, with \(0\) in its interior.  That
means that \(\{\ell_1,\dots,\ell_d\}\) are linearly independent
and there exists a convex linear combination
\(\lambda_0\ell_0+\lambda_1\ell_1+\ldots+\lambda_d\ell_d=0\) with
all positive coefficients; this clearly implies that any
\(d\)-subset of \(L\) is linearly independent and the
\((d+1)\)-tuple \((\lambda_0,\lambda_1,\ldots,\lambda_d)\) is
unique.

The \emph{polar} of \(\hat{L}\) is
\(\st{\hat{L}}=\conj{x\in
  V}{\ell(x)\leq 1,\ \text{for all}\ \ell \in \hat{L}}\), and it is a
simplex with half-space description \conj{x\in
  V}{\ell_i(x)\leq 1, i=0,1,\ldots,d}.  Its vertices are
\(r_0,\ldots,r_d\), where \(r_i\) is the unique solution of the linear
system \(\{\ell_j(x)=1, \text{all}\ j\neq i\}\).  Another view of the polar
is that \conj{\ell\in \st{V}}{\!\ell(r_i)\leq 1, i=0,\ldots,d} is the
half-space description of \(\hat{L}\), where \(r_i\) is used to describe
the unique facet not containing \(\ell_i\).

% \textbf{I don't know if this is useful, leave it here for now:}

% \begin{pro}
%   Suppose that \(L=\{\ell_0,\ell_1,\ldots,\ell_d\}\) are the vertices of a
%   \(d\)-simplex, and the positive reals
%   \(\lambda_0,\lambda_1,\ldots,\lambda_d\) satisfy
%   \(\sum_i\lambda_i\ell_i=0\).  Then, \(\sum_i\lambda_ir_i=0\).
% \end{pro}
% \begin{proof}
%   For each \(i\),
%     \(\ell_i\left(\sum_j\lambda_jr_j\right) = \sum_j\lambda_j\ell_i(r_j)
%                               = \sum_{j\neq i}\lambda_j +\lambda_i\ell_i(v_i) 
%                                = \sum_{j\neq i}\lambda_j\ell_j(v_i) +\lambda_i\ell_i(v_i)
%                               = \sum_j\lambda_j\ell_j(v_i)  
%                               = 0.\)
%   % \textbf{maybe display mode here is too much}
%   % \begin{align*}
%   %   \ell_i\left(\sum_j\lambda_jr_j\right) &= \sum_j\lambda_j\ell_i(r_j) \\
%   %                             &= \sum_{j\neq i}\lambda_j +\lambda_i\ell_i(v_i) \\
%   %                              &= \sum_{j\neq i}\lambda_j\ell_j(v_i) +\lambda_i\ell_i(v_i) \\
%   %                             &= \sum_j\lambda_j\ell_j(v_i)  \\
%   %                             &= 0.
%   % \end{align*}
%   As the \(\ell_i\) span \st{V}, the result follows.
% \end{proof}

A\, \emph{fan} (in \(V\)) is a collection \(\mathcal{F}\) of
polyhedral cones, such that any face of a cone in \(\mathcal{F}\)
is also in \(\mathcal{F}\), and the intersection of any two cones
in \(\mathcal{F}\) is a face of both. A fan is \emph{complete} if
the union of its cones is \(V\), \emph{pointed} if \(0\) is in it,
and \emph{simplicial} if each of its members is a cone over a
simplex.  A \emph{simplex fan} is the collection of cones over
the proper faces of a full-dimensional simplex containing the
origin in its interior, the \emph{face-fan} of the simplex.

To present the normal fan, we momentarily switch the viewpoint and consider
each point of \(V\) as a linear functional on \st{V}, in the natural way.
The \emph{normal cone} of a face \(F\) of \(\hat{L}\) is the set of linear
functionals whose maximum over \(\hat{L}\) is attained at \(F\).  The
collection \(\mathcal{L}\) of normal cones of the faces of \(\hat{L}\) is
the \emph{normal fan} of \(\hat{L}\).  This is face fan
of the polar simplex \st{\hat{L}}.  Switching back to the original setting,
the normal cone of a face \(F\) has the half-space description
\conj{x\in
  V}{\ell_i(x)-\ell_j(x)\geq 0,\ \text{for all}\ \ell_i\in f, j\neq i} (the
inequalities with both sides in \(F\) are actually equalities).  In
particular, the normal cone of a vertex \(\ell_i\) is
\(\left\{(\ell_i-\ell_j)(x)\geq 0, \text{all}\ j\neq i\right\}\); these are
the maximal cones of the fan.

The \emph{level function} of \(\mathcal{L}\) is defined on \(V\)
by \(f_L(x)\!=\!\max\{\ell_i(x)\,|\,i\!=\!0,\ldots,d\}\).  One easily
verifies that it equals \(\ell_i\) on the normal cone of vertex
\(\ell_i\), and it is continuous over \(V\).  Moreover,
\(\st{\hat{L}}=\conj{x\in V}{f_L(x)\leq 1}\), just from the
definition.  For every \(i\), the set
\conj{x\in \st{\hat{L}}}{\ell_i(x)=1} is a facet of \st{\hat{L}},
and the cone over it is the normal cone of \(\ell_i\).

We will be interested in the symmetries of the whole structure,
the simplex, the polar and the normal fan.  
It is natural to look at linear maps, and we simply observe:

\begin{pro}
  Let \(L\) be the set of vertices of the simplex \(\hat{L}\in \st{V}\),
  and let \(T\in GL(V)\).  Then, \(\hat{L}\) is invariant under \(T^t\) if
  and only if \(\st{\hat{L}}\) is invariant under \(T\) (and, a fortiori,
  so is the normal fan \(\mathcal{L}\)).
\end{pro}

Let \(\Delta_d\) be a simplex in a \(d\)-dimensional space with
vertices \(v_0,v_1,\ldots,v_d\), and \(0\) in its interior; it
will stand both for \(\hat{L}\) and \st{\hat{L}}.

The 1-skeleton of \(\Delta_d\) is isomorphic (as a graph) to the
complete graph \(K_{d+1}\) on vertices \(\{0,1,\dots,d\}\).  Any
automorphism \(T\) of \(\Delta_d\) will induce an automorphism of
the 1-skeleton, hence it will permute the vertices, inducing a
permutation \(\pi\in S_{d+1}\) on the indices.  So, we are led to
defining an \emph{automorphism} of \(\Delta_d\) as a permutation
\(\pi\) on the indices such that there exists \(T\in GL(V)\) such
that for all \(i\), \(T(v_i)=v_{\pi(i)}\).  This is a minor
variation on the definition of linear automorphism in
\cite{BSPRS}; so is the following description of automorphisms
using colored graphs.  We permute indices, as \(T\) will effect
the same permutation on the simplex, its polar, and the
maximal cones of the normal fan.  Actually, given \(\pi\),
there is only one corresponding \(T\), as the vertices span the space.
Indeed, if \(\pi\) is any permutation, define the linear map
\(T^\pi\) on a basis by \(T^\pi(v_i)=v_{\pi(i)}\),
\(i=1,\ldots,d\).  Then, \(\pi\) is an automorphism if and only
if \(T^\pi(v_0)=v_{\pi(0)}\).  The group of automorphisms of
\(\Delta_d\) will be denoted \(\Aut(\Delta_d)\).

Further, if \(\sum_j\lambda_jv_j=0\) denotes the only convex combination of
vertices expressing \(0\), an application of \(T\) yields
\(\sum_j\lambda_jv_{\pi(j)}=0\), and the uniqueness implies
\(\lambda_{\pi(j)}=\lambda_j\).

\begin{pro}\label{simplexauto}
  Let \(L=\{v_0,v_1,\dots,v_d\}\) be the set of vertices of a
  simplex, and suppose that \(\sum_j\lambda_j v_j=0\), where the
  \(\lambda_j>0\).  Color vertex \(j\) of \(K_{d+1}\) with color 
  \(\lambda_j\).  Then, a permutation is an automorphism of \(\Delta_d\) if and
  only if it preserves colors (that is, it is an automorphism of the colored graph).
\end{pro}
\begin{proof}
  The comments preceding the statement of the theorem imply that any
  automorphism preserves colors.  Let us show that preserving
  colors is enough.

  Let \(\pi\) be a color preserving permutation, and let us show that
  \(T^\pi(v_0)=v_{\pi(0)}\), which will  prove that
  \(\pi\in\Aut(\hat{L})\).

  For every color \(\lambda\), let \(L_\lambda\) be the set of
  vertices with index colored \(\lambda\).  This is invariant under
  \(T^\pi\), so, the vector \(\sum_{j\in L_\lambda}v_j\) is fixed
  by \(T^\pi\).  The definition of \(\Lambda\) implies that
  \[
    \sum_{\lambda\in\Lambda}\lambda\cdot\sum_{j\in L_\lambda}v_j=0,
  \]  
  and applying \(T^\pi\) it follows that \(\sum_{j\in L_{\lambda_0}}v_j\) is
  also fixed.  Since \(T^\pi(v_j)=v_{\pi(j)}\) for
  \(j\in L_{\lambda_0}\backslash\{0\}\), it follows that
  \(T^\pi(v_0)=v_{\pi(0)}\), as required.
\end{proof}

Endow now \(V\) with an inner product \ip{}{}, and
identify \(V\) and \st{V} via the isomorphism \(f\) such that
\(f(v)(x)=\ip{v}{x}\).  Say that an automorphism of a
simplex is an \emph{isometry} if the corresponding linear map is
an isometry.  Since \(L\) spans \st{V}, an automorphism is an
isometry if and only if it preserves all pairwise inner products
of vertices.

The \emph{isometry graph } of \(\Delta_d\) is the complete graph
on the vertices of \(\Delta_d\) with vertices and edges colored
as follows: color vertex \(v_j\) with \ip{v_j}{v_j}, and edge
\(v_i\!-\!v_j\) with color \ip{v_i}{v_j}.

\begin{pro}\label{colorsimplex}
  An automorphism of \(\Delta_d\) is an isometry if and only if
  it is an automorphism of its isometry graph.
\end{pro}
\begin{proof}
  Let \(\pi\) be an automorphism of \(\Delta_d\).  If it is an
  isometry, then \(T^\pi\) will preserve all inner products, hence
  \(\pi\) will be an automorphism of the graph.  Conversely, if
  \(\pi\) is an automorphism of the graph, \(T^\pi\) preserves
  all inner products, and, as the vertices of \(\Delta_d\) span
  the space, \(T^\pi\) is an isometry.
\end{proof}

\begin{Example}
  Let \(e_1,\ldots,e_d\) be the canonical basis for \(\R^d\), let
  \(v_i=e_1+\cdots+e_i\), , \(i=1,\ldots,d\), and
  \(v_0=-(v_1+\cdots+v_d)\). Those \(v_j\) span a simplex, and
  \(v_0+\cdots+v_d=0\), hence its group of automorphisms is
  \(S_{d+1}\).  However, no nontrivial automorphism is an
  isometry, as the vertices are at distinct distances from the
  origin (so, each vertex of \(K_{d+1}\) is a different color,
  never mind the edges).
\end{Example}

\subsection{The FOIC}
\label{sec:foic}

We will fix \(n\) here throughout; several entities will be
introduced depending on \(n\), and use \cref{eq:Mp} as a paradigm
for extending the notation to show that dependency, when needed.
Let us interpret the \Ifor(b) in light of the preceding material.
Denote \(V=\R^{\floor*{n/2}}\), and for
\(j=0,1,\ldots,\floor*{n/2}\), define the functional
\begin{equation}
    \label{eq:Mp}
      I_j^{(n)}(x) = I_j(x)=\sum_{k=1}^{\floor*{n/2}} -\cos\frac{2\pi j k}{n}\cdot x_k,
\end{equation}
and let \(I=\{I_0,I_1,\ldots,I_{\floor*{n/2}}\}\).  The \Ifor
says that the interlace number of
\(p(x)=\sum_{k=1}^{\floor*{n/2}}\sig{p}_k\pal{n,k}\) is
\(\il{p}=\max_j I_j(\sig{p}_1,\ldots,\sig{p}_{\floor*{n/2}})\).
\begin{Example}\label{ex:ifunc6}
  We will use the case \(n=6\) to illustrate our notation, here
  and later in this section; as
  \((\cos 2\pi j/n)_{j=0,\ldots,5}=(1,1/2,-1/2,-1,-1/2,1/2)\),
  all numbers involved are rational, and the examples look
  uncomplicated.  Of course, for higher values of \(n\), the
  cosines are messier algebraic numbers.

  In this case, the four functionals in \cref{eq:Mp} are:
  \[
  \begin{array}{lrrr}
    I_0(x) =  &-x_1&- x_2&-x_3 \\
    I_1(x) =  &-\lilhalf x_1&+\lilhalf x_2&+x_3\\
    I_2(x) =  &\lilhalf x_1&+\lilhalf x_2&-x_3\\
    I_3(x) =  & x_1&- x_2&+x_3 .
  \end{array}
  \]
\end{Example}

On the space of trim
darga \(n\) polynomials we define the complex valued maps
\(J_j(p)=-\frac12 p(\tn^j)\); then, if \(p(x)\in \Trim{n}\),
\begin{equation}
  \label{eq:IJ}
  I_j(\sig{p})= J_j(p).
\end{equation}
That means that \(I_j\) and \(J_j\) describe the same linear
functional on \Trim{n}, on different bases.  So we do away with
\(J_j\) and write just \(I_j(p)\), using whatever basis is convenient. 

Expressing \(I_j\) in terms of the canonical basis for polynomials,
the submatrix whose rows are \(I_1,\ldots,I_{\floor*{n/2}}\),
restricted to columns \(1,\ldots,\floor*{n/2}\) is a row Vandermonde
matrix, and this shows that the corresponding functionals are linearly
independent.

In what follows ahead, we will have to argue separately according to the
parity of \(n\).  Some case analysis will be finessed by defining:
\begin{equation}
  \label{eq:mdelta}
  m=\floor*{\frac{n-1}2}\quad\text{and}\quad \delta(n)=
  \begin{cases}
    -1 & \text{if \(n\) is even,}\\
    \hfill 0 & \text{if \(n\) is odd.}
  \end{cases}
\end{equation}

\begin{pro}\label{Isimplex}
  The functionals in \(I\) are the vertices of a simplex, containing the
  origin.  Moreover,
  \[
  \begin{array}{lll}
    I_0+2\sum_{j=1}^m I_j+I_{\frac{n}2}&=0 & \text{if \(n\) is even,}\\
    I_0+2\sum_{j=1}^m I_j&=0 &   \text{if \(n\) is odd.}
  \end{array}
  \]  
\end{pro}
\begin{proof}
  Since \(I_1,I_2,\ldots\) are linearly independent, it is enough to show
  the two linear relations.  It follows from the properties of roots of
  unity that \(\sum_{j=0}^{n-1}J_j=0\). However, for \(j=1,\ldots,m\),
  \(J_{n-j}=J_j\) on \Trim{n}, hence the previous sum becomes
  \(J_0+2\sum_{j=1}^m J_j=0\) if \(n\) is odd, and
  \(J_0+2\sum_{j=1}^m J_j+J_{\frac{n}2}=0\) if \(n\) is even. The result follows from \cref{eq:IJ}.
\end{proof}

We call \(\hat{I}^{(n)}\) the \emph{interlace simplex}; its
normal is the \emph{fan of interlace certs}, or simply
the \emph{FOIC}, denoted by \(\CC^{(n)}\); these superscripts,
as usual, are omitted.  Both \(\CC^{(2m)}\) and \(\CC^{(2m+1)}\)
live in \(\R^m\), but there are notable geometric differences,
which can be seen in \cref{symgp} (and exemplified in
\cref{sec:small-dargas}).  In view of the \Ifor, the interlace
number, viewed as a function on \Trim{n} %(on the \palsym-basis)
is the level function of the interlace fan.  We will denote by
\(C_j\) the normal cone of \(I_j\).  So,
\begin{center}
  \(C_j\) \glsadd{coneC} is the set of trim polynomials that have \(\tn^j\) as an interlace cert.
\end{center}

As with \(p_{n/2}\), writing \(C_{n/2}\), or \(I_{n/2}\)
implicitly carries the proviso ``in case \(n\) is even''.

The simplex structure trivially implies:

\begin{pro} \label{FacesOfTheFan}
  For every proper
  \(A\subset\vn\), the polynomials whose set of interlace certs
  is precisely \(A\) comprise a nonempty face of the fan \CC.
\end{pro}
Also, it follows that \(C_j\) has the half-space description
\begin{equation}
      \label{eq:Cj}
      I_j(p)-I_r(p)\geq 0,\qquad 0\leq r\leq\floor*{n/2}, r\neq j.
\end{equation}

\newcommand{\hp}{\sig{p}}
\begin{Example}\label{ex:foic6}
  Going back to \(n=6\), we recall \cref{ex:ifunc6}, and compute
  the descriptions of all four \(C_j^{(6)}\).  As all
  inequalities are homogeneous, we present each conveniently
  scaled by a positive rational, so that all coefficients are
  integers.

  \(C_0\)  is given by the three inequalities:
  \[
  \begin{array}{lr@{}r@{}r}
   I_0(\hp)-I_1(\hp)= &\hp_1&-3\hp_2&-4\hp_3\geq 0,\\
  I_0(\hp)-I_2(\hp)=&-3\hp_1&-3\hp_2&\geq 0,\\
  I_0(\hp)-I_3(\hp)=&-\hp_1& &-\hp_3\geq 0 .
  \end{array}
\]
Similarly, we obtain:
  \[
  \begin{array}{lr@{}r@{}r@{\quad}r@{}r@{}r@{\quad}r@{}r@{}rrr}
    C_1:&\hp_1&+3\hp_2&+4\hp_3\geq 0, &-\hp_1& &+2\hp_3\geq 0, &\hp_1&+\phantom{3}\hp_2&\geq 0.\\
    C_2:&\hp_1&+\phantom{3}\hp_2&\geq 0, &\hp_1& &-2\hp_3\geq 0, &-\hp_1&+3\hp_2&-4\hp_3\geq 0.\\
    C_3:&\hp_1& &+\hp_3\geq 0, &\hp_1&-\hp_2& \geq 0, &\hp_1&-3\hp_2&+4\hp_3\geq 0.
  \end{array}
  \]
\end{Example}
All coefficients of the linear forms \(I_j(p)\) are real
algebraic numbers, as all involved cosines are zeros of
\(T_n(x)-1\), where \(T_n\) is the \(n\)th Chebyshev polynomial
of the first kind.  Moreover, the interlace number is linear in
each of the domains \(C_j\).  This implies:

\begin{pro}\label{ilsemialg}
  The map \(p\mapsto\il{p}\) is a piecewise-linear,
  semi-algebraic function on \Trim{n}.
\end{pro}

The semi-algebraic character of \il{} is, at this point, a mere
curiosity.  The discussion preceding \cref{cnsemialg} should
clarify this.

We can also exhibit the vertices of \st{\hat{I}}, thus explicitly
describing the rays corresponding to polynomials with maximum
number of interlace certs.

\begin{pro}\label{inter-rays}
  The vertices of \st{\hat{I}} are
  \(p^{(j)}(x)=\sum_{k=1}^{n-1}(\tn^{jk}+\tn^{-jk})x^k\), with
   \palsym-representation
  \(\sig{p}_k^{(j)}=4\cos\frac{2\pi jk}n, 1\leq k\leq m\), and
  \(\sig{p}_{n/2}^{(j)}=2(-1)^j\).
\end{pro}
\begin{proof}
  Notice that if \(r\neq j\) and both lie in
  \([0,\floor*{n/2}]\), then \(\tn^{j+r}\) and
  \(\tn^{j-r}\) are \(n\)th roots of 1, and neither is 1, so
  they are zeros of \(\sum_{k=0}^{n-1}x^k\).  We compute
  \begin{align*}
    I_r((p^{(j)}) &= -\lilhalf p^{(j)}(\tn^r)\\
                  &= -\lilhalf\left(\sum_{k=1}^{n-1}\tn^{(j+r)k}+\sum_{k=1}^{n-1}\tn^{(r-j)k}\right)\\
                  &= -\lilhalf(-1-1)\\
                  &= 1,
  \end{align*}
  which shows that \(p^{(j)}\) is indeed the solution defining
  the facet of \(\hat{I}\) opposite \(I_j\).
\end{proof}

\begin{Example}\label{ex:geom-ray}
  As a special case, we have \(p^{(0)}(x)=\geom_n(x)\). Indeed,
  for any \(\w\in\vn\backslash\{1\}\), \(\geom_n(\w)=-1\), so all
  those \w are interlace certs.
\end{Example}

A somewhat surprising conclusion from this is that
\(p^{(j)}=-2I_j\), if we identify \(V\) and \st{V} using the
canonical basis, hence \(\st{\hat{I}}=-2\hat{I}\).

Polynomials that have a real interlace cert are quite special
with respect to the circle number, as we will see in
\cref{pofone}.  So, it seems interesting to dedicate further
study to \(C_0\) and \(C_{n/2}\).  It would be
nice to have easily checkable sufficient conditions implying
\(p\) is in either of those. For instance:
    
\begin{pro}\label{inc0}
  If \(p(x)\) has nonpositive coefficients, then \(p\in C_0\).
  If \(n=\darga{p}\) is even, and \(p(-x)\) has nonpositive
  coefficients (in particular, if \(p\) is
  an odd polynomial with nonnegative coefficients), then \(p\in C_{n/2}\).
\end{pro}
\begin{proof}
  Suppose that all coefficients of \(p\) are nonpositive. Write
  \(p(x)=\sum_j a_j x^j\).  For every \(w\in \ucirc\),
  \[
    -p(w)= \sum_j (-a_j)w^j\leq \sum_j (-a_j)|w^j|=-p(1),
  \]
  hence \(1\) is an interlace cert.  The \(p(-x)\) case is similar.
\end{proof}

Preparing for the next result and for \cref{ex:botta}, we note
the following fact that is easily proved by induction:
\begin{pro}\label{sinnx}
  For all real \(x\) and integer \(n\geq 0\),
  \(|\!\sin nx|\leq n\,|\!\sin x|\), with equality only if either
  both sides are \(0\) or \(n=1\).
\end{pro}

In \cref{sec:binomial-polynomials} we will see a ``natural''
family satisfying the following conditions:
\begin{pro}\label{sumj2}
  If
  \begin{equation*}
    p_1\geq \sum_{k=2}^{\floor*{n/2}} k^2|\sig{p}_k|
  \end{equation*}
  then \(p\in C_{\floor*{n/2}}\).
\end{pro}
\begin{proof}
  Consider the real function
  \(f(t)=-\sum_{k=1}^{\floor*{n/2}} \sig{p}_k\cos k t\); we will
  show that \(f(t)\) is nondecreasing in the interval
  \([0,\pi]\), whence the result will follow immediately from the
  \Ifor.  We take the derivative and show it is nonnegative in
  the interior of that interval, then use \cref{sinnx}. So,
  \begin{align*}
    f'(t) &= \sum_{k=1}^{\floor*{n/2}} k \sig{p}_k\sin k t\\
          &=\left(p_1+\sum_{k=2}^{\floor*{n/2}} k \sig{p}_k\frac{\sin k t}{\sin t}\right)\sin t\\
          &\geq\left(p_1-\sum_{k=2}^{\floor*{n/2}} k |\sig{p}_k|\frac{|\sin k t|}{|\sin t|}\right)\sin t\\
          &\geq\left(p_1-\sum_{k=2}^{\floor*{n/2}} k^2 |\sig{p}_k|\right)\sin t\\
          &\geq\sin t \geq 0
  \end{align*}
  for \(0\leq t\leq\pi\).
\end{proof}

It is clear that the basic idea in the proof above should work
under less restrictive hypotheses on the coefficients.

Another way in which \(C_0\) and \(C_{n/2}\) differ from the
other roots is reflected in the symmetries of \CC.  Recall, from
\cref{sec:fans-their-autom}, that \(\Aut(\CC)\) is the group of
permutations of indices that can be realized as linear
transformations of the whole space.

Immediately
from \cref{Isimplex,simplexauto} we have:

\begin{pro}
  The automorphism group of \CC is isomorphic to \(S_1\times S_m\)
  if \(n\) is odd, \(S_2\times S_m\) if \(n\) is even.  In both
  cases, the cones \(C_j\) fall in two orbits: one is
  \(\{C_1,\ldots,C_m\}\), the other is \(\{C_0\}\) or
  \(\{C_0,C_{n/2}\}\), according with \(n\) being odd or even.
\end{pro}

Let us endow \(V\) and \(V^*\) with the usual inner product, and
consider the isometries of the fan, as defined in
\cref{sec:fans-their-autom}.  
We describe the isometry group of the fan.

Before that, we prepare a little tool.  Recall that
\(\delta\) was defined in \labelcref{eq:mdelta}.

\begin{lem}\label{Srn}
  Define, for any integer \(r\), \(S(r,n)=\sum_{k=1}^m(\tn^{rk}+\tn^{-rk})\).  Then,
  \[
    S(r,n)=
    \begin{cases}
      -1-\delta(n)^r & n\nmid r,\\
      2m & n|r.
    \end{cases}
  \]  
\end{lem}
\begin{proof}
  If \(n|r\), all terms are \(1\) and the result is clear.
  Otherwise, \(\w=\tn^r\) is a zero of
  \(f(x)=\sum_{k=0}^{n-1}x^k\).  But
  \(f(\w)=1+S(r,n)+\delta(n)^r\), where, in the \(n\) even
  case we substitute \(\delta(n)\) for \(\w^{n/2}\).  The result
  follows.
\end{proof}

Recall that the isometry graph of a simplex is the complete graph
on the vertices of the simplex, with each vertex colored with its
norm squared and each edge colored with the inner product of its
extremities.
\begin{pro}\label{ilcolors}
  The isometry graph of the interlace simplex has two vertex
  colors, one for \(I_0\) and \(I_{n/2}\), the second for all
  other vertices.  If \(n\) is odd, all edges have the same
  color; if \(n\) is even there are two edge colors, the color of
  an edge depending only on whether its vertices have indices of
  the same parity or not.
\end{pro}
\begin{proof}
  Let us define vectors \(H_j\in \R^{m+1}\),  \(0\leq k\leq \floor*{n/2}\), by
  \(H_{j,k}=\tn^{jk}+\tn^{-jk}\) if \(1\leq k\leq m\),
  \(H_{j,m+1}=2\delta(n)^j\).  Thus, \(H_j=-2I_j\),
  except that, in the \(n\) odd case , we isometrically identify
  \(\R^m\) with \(\R^m\times\{0\}\).  A simple computation shows that
  \[
    \ip{H_r}{H_s} = S(r+s,n)+S(r-s,n)+4\delta(n)^{r+s}.
  \]  
  If \(r\neq s\), \cref{Srn} implies that
  \[
    4\ip{I_r}{I_s}=\ip{H_r}{H_s} =-1-\delta(n)^{r+s}-1-\delta(n)^{r-s}+4\delta(n)^{r+s}=-2+2\delta(n)^{r+s}.
  \]

  For each \(r\),
  \begin{align*}
    4\ip{I_r}{I_r}=\ip{H_r}{H_r} &= S(2r,n)+S(0,n)+4\delta(n)^{2r}\\
                           &=S(2r,n)+2m+4\delta(n)^2\\
                           &=
    \begin{cases}
      4m+4\delta(n)^2 & \text{if \(r=0\) or \(r=n/2\)}\\
      2m-1+ 3\delta(n)^2& \text{otherwise}.
    \end{cases}
  \end{align*}

  Summarizing these formulas:
  \begin{itemize}
      \item If \(n\) is odd, \(\norm{I_0}^2=\frac{n-1}2\) and
    for \(1\leq r\leq m\), \(\norm{I_r}^2=\frac{n-2}4\). Further, if
    \(r\neq s\), \(\ip{I_r}{I_s}\!=\!-\frac12\).
      \item If \(n\) is even,
    \(\norm{I_0}^2=\norm{I_{n/2}}^2=\frac{n}2\)
    and for \(1\leq r\leq m\), \(\norm{I_r}^2=\frac{n}4\).
    Further if \(r\neq s\), \(\ip{I_r}{I_s}=0\) if \(r\) and
    \(s\) have the same parity, \(=-1\) if they have opposite
    parity.

  \end{itemize}
  From this, all the statements about coloring follow.
\end{proof}

Now we can describe:

\begin{teo}\label{symgp}
  The group of isometries of \CC is isomorphic to
  \begin{center}
    \begin{tabular}{l@{\hspace{20pt}}l}
          \(S_m\times S_1\)  & if \(n\) is odd \\
      \(S_{\floor*{m/2}}\times S_{\ceil*{m/2}}\times S_2\) & if \(n\equiv 0 \pmod 4\),\\
      \(S_1\times S_{m/2}\wr S_2\)    & if \(n\equiv 2 \pmod 4\),  
    \end{tabular}
  \end{center}
  where \(\wr\) stands for the wreath product.
\end{teo}
\begin{proof}
  If \(n\) is odd, the isometry graph described in
  \cref{ilcolors} has just the vertex \(0\) singled out, no
  colors distinguishing otherwise betweeen vertices or between
  edges.  So, a permutation is an automorphism if and only if it
  fixes \(0\).
  
  Now we move to \(n\) even.  Denote \(d=\floor*{n/2}\), and
  \(O\), \(E\) the sets of odd and even members of
  \(\{1,2,\ldots,m\}\); note that \(|E|\!=\!\floor*{m/2}\),
  \(|O|\!=\!\ceil*{m/2}\). Let \(S_E,S_O\) be the full group of
  permutations of each set, also acting as automorphisms of
  \(K_{d+1}\) leaving the other vertices fixed.  Those
  permutations are clearly automorphisms of the isometry graph.
  As \(E\cap O=\emptyset\), the subgroup of \(\Aut{K_{d+1}}\)
  generated by their union is the product \(S_E\times S_O\).

  If \(n\equiv 0 \pmod 4\), \(\frac{n}2\) is even.  That implies
  that for every vertex \(j\) of \(K_{d+1}\), the edges joining
  \(j\) to \(0\) and to \(\frac{n}2\) have the same color,
  describing whether \(j\in E\) or \(j\in O\); so the
  transposition \((0\; \frac{n}2)\) is an automorphism of the
  isometry graph.  Putting everything together, we have
  \(\Aut{K_{d+1}}= S_E\times S_O\times S_{\{0,\,n/2\}}\).

  The case \(n\equiv 2 \pmod 4\) is more interesting.  Here,
  \(m\) is even and \(|E|=|O|=m/2\).  The involution
  \(\tau(j)= d-j\) of \(K_{d+1}\) is an automorphism of the
  colored graph. It exchanges \(0\) with \(n/2\), \(E\) with
  \(O\).  Consider any \(\pi\in\Aut{K_{d+1}}\).  If \(\pi(0)=0\),
  then \(\pi\) fixes \(n/2\) and leaves \(E\) and \(O\)
  invariant, so
  \(\pi\in S_E\times S_O\times S_{\{0\}}\times S_{\{n/2\}}\).  If
  \(\pi(0)=n/2\), then \(\pi\tau\) is an automorphism that fixes
  \(0\).  This shows that
  \(\pi\in S_E\times S_{\{0\}} \wr \langle\tau\rangle\).  On the
  other hand, any permutation in
  \(S_E\times S_{\{0\}} \wr \langle\tau\rangle\) is easily seen to
  be an automorphism of \(K_{d+1}\).
\end{proof}

\subsection{The self-inversive FOIC}
\label{sec:self-inversive-foic}

In the same way that the FOIC was inspired by the
\Ifor\Item{ilnc}, \cref{siilnumber}\Item{iilnc} leads to a fan
classifying self-inversive polynomials into a simplex fan in
\(\R^{n-1}\), with faces indexed by subsets of \un.  However,
some of the interesting structure uncovered in the previous
section fades away, muddled by excessive symmetry. We just state
what happens.

\begin{pro}
  If \(\w\in\un\), the map \(p(x)\mapsto p(\w x)\) is a unitary
  transformation of \CTrim{n} preserving interlace number;
  further, if \(\theta\) is an interlace cert of \(p(x)\),
  \(\w^{-1}\theta\) is an interlace cert of \(p(\w x)\).
\end{pro}

\begin{pro}\label{self-inv-foic}
  Identify \CTrim{n} with \(\R^{n-1}\) using the
  \palsym-representation.  For each \(j=0,1,\ldots,n-1\), let
  \(CC_j\) denote the set of self-inversive polynomials of darga
  \(n\) with interlace cert \(\tn^j\). Then the cyclic group \un
  acts isometrically and transitively on the family \(CC_j\).
\end{pro}

Which means that for self-inversive polynomials,
``you are in a maze of simplicial little cones,  all alike''\footnote{For those missing it, this is a reference to one of the oldest Adventure games --- see \url{https://en.wikipedia.org/wiki/Colossal_Cave_Adventure\#Memorable_words_and_phrases}.}.

\section{The circle number}
\label{sec:circle-number}

We recall that the \emph{\gls{circnum}} of a self-inversive
polynomial \(p\) is
\[\cn{p}=\inf\conj{\beta>0}{\palpha\
  \text{is circle rooted for all}\ \alpha\geq\beta}.\]

In fact, as the following results show, \(\inf\) can be
rightfully substituted by \(\min\) in the definition above.

We start by guaranteeing that the circle number is positive.

\begin{pro}\label{cnlobo}
  \(\cn{p}\geq\max_k \frac{|p_k|}{\binom{n}{k}}\cdot\)
\end{pro}
\begin{proof}
  Let \(\alpha>0\) be such that \palpha is circle rooted.  Then,
  the monic polynomial
  \(\palpha/\alpha=x^n+1+\frac1\alpha \sum_{k=1}^{n-1}p_kx^k\)
  has as coefficients the elementary symmetric polynomials of the
  roots.  The \(k\)'th symmetric polynomial has \(\binom{n}{k}\)
  terms, and since all roots have norm one, it is bounded in
  absolute value, by the number of terms.  That is,
  \(\left|\frac{p_k}{\alpha}\right|\leq \binom{n}{k}\).
\end{proof}

The following is a companion to \cref{ilscaling}; the same
discussion about normalization applies:

\begin{pro}\label{cnscaling}
  Let \(p\) be a trim self-inversive polynomial. Then
  \begin{enumerate}
      \item \label{cnscalinga} For every real \(\lambda>0\), \(\cn{\lambda\,p}=\lambda\,\cn{p}\).
      \item \label{cnscalingb} For every positive integer \(r\), \(\cn{p(x^r)}=\cn{p(x)}\).
      \item \label{cnscalingc} If \(\theta\in\un\),
    \(\cn{p(\theta x)}=\cn{p(x)}\).  In particular, if \(p\) has even darga,
    then \(\cn{p(-x)}=\cn{p(x)}\).
  \end{enumerate}
\end{pro}
\begin{proof}
  Linear scaling is quite obvious.  For exponent scaling, let
  \(q(x)=p(x^k)\), and note that \(\darga{q)}=k\,\darga{p}\).
  For any \(\alpha\), the roots of \(q_\alpha\) are the \(k\)-th
  roots of zeros of \(p_\alpha\).  So, \(q_\alpha\) is circle
  rooted if and only if \(p_\alpha\) is.  Part \Item{cnscalingc}:
  for any \(\alpha\in\R\), \(\palpha(x)\) is circle rooted if and
  only if so is \(\palpha(\theta x)\).
\end{proof}

A marked difference between interlace number and circle number is
that the former concerned interlacing a specific set of points,
while the latter, as we will see, concerns finding a common
interlace between two polynomials.  As before, we will translate
between angle interlace and interlace in the real line, but
instead of the exponential map, we use a M\"obius transformation,
sometimes called the \emph{Cayley map}.  This is well explained
in an expository text by \citet{Conrad} (and is a minor variation
of a technique presented in \citet{Marden} and hinted in
\cite[Sect. 11.5]{RS}); moreover the recent article by
\citet{vieira19} uses the Cayley map as a tool for the purpose of
counting roots in the circle.  Some properties of that map we
present here also appear in that article, in particular ``root
correspondence'', but we have kept our narrowly focused
presentation.  Given \(\w\in\T\), define, for \(p\in\C[x]\)
\[
    \St{p}(x)= (x+i)^{\darga{p}}p\left(\frac{\w x-\w i}{x+i}\right).
  \]

  Note that this map is multiplicative, \(\St{pq}=\St{p}\St{q}\).
  The idea is that the M\"obius transformation
  \(\mu(z)=\frac{\w z-\w i}{z+i}\) is a homeomorphism from the
  real line to the unit circle minus the point \(\w\). This
  implies a dictionary between things happening on the circle and
  things on the line (see \cite[Thm. 3]{vieira19} for a similar
  statement):

  \begin{pro}[\textsc{Root correspondence}]\label{stretch}
    The map \(x\mapsto\mu(x)\) is a multiplicity preserving
    bijection between the real roots of \St{p} and the roots of
    \(p\) in \(\ucirc\backslash\{\w\}\).  Further, two
    \(k\)-subsets of \R interlace if and only if their images by
    \(\mu\) angle-interlace.
  \end{pro}

  Usually, \St{p} will have non-real coefficients, but not in the
  cases we are interested in:

  \begin{pro}\label{realpol}
    Suppose that \(p\) is self-inversive of darga \(n\) and
    \(\w\in\un\). Then,
    \begin{enumerate}
        \item \label{realpola} \(\St{p}\in\R[x]\).
        \item \label{realpolb} Provided \(p(\w)\neq 0\), \St{p} has degree \(n\).
        \item \label{realpolc} \(p\) is circle rooted if and only if \St{p} is real
      rooted.
    \end{enumerate}
  \end{pro}
  \begin{proof}
    For the first part, we use the definition of \(S\) and of self-inversive:
    \phantom{xxxxxx}\(\St{p}(x) =(x+i)^np\left(\frac{\w (x-i)}{x+i}\right)
    =(x+i)^n\left(\frac{\w (x-i)}{x+i}\right)^n \bar{p}\left(\frac{x+i}{\w (x-i)}\right)\)\\
    \phantom{xxxxxx}\(\phantom{\St{p}(x)} =(x-i)^n\bar{p}\left(\overline{\left(\frac{\w (x-i)}{x+i}\right)}\right)=\overline{\St{p}(x)}.
    \)

    The degree follows from the observation that the coefficient
    of \(x^n\) in \St{p} is precisely \(p(w)\).

    The last part follows from the earlier ones and root correspondence.
  \end{proof}

  \begin{lem}\label{goodstr}
    Let \(p\) be a trim self-inversive polynomial of darga \(n\).
    Then, there exists \(\w\in\un\) such that \St{\palpha}
    has degree \(n\) for every \(\alpha>0\).
  \end{lem}
  \begin{proof}
    Let \w be an interlace cert of \(-p\); hence
    \(p(\w)>0\).  It follows that for any \(\alpha>0\),
    \(\palpha(\w)=p(\w)+2\alpha>0\), so, by \cref{realpol},
    \St{\palpha} has degree \(n\).
  \end{proof}

  This gives a further characterization of the circle number and
  is a prelude to the proof of \cref{dbroot}:

  \begin{pro}
    Let \(p\) be a trim self-inversive polynomial of darga \(n\).
    Then
    \begin{align*}
      \cn{p}=\min\{\, \beta\mid\, &\palpha\ \text{and}\ x^n+1\ \text{have a common angle interlace}\\
      &\text{for every}\       \alpha\geq\beta\}.
    \end{align*}
  \end{pro}
  \begin{proof}
    We will assume that \(p\) and \(x^n+1\) are coprime;  the
    case where there exists a nontrivial common factor can be
    handled as in the proof of \cref{dbroot}.  Choose
    \(\w\) as in \cref{goodstr}.  Let \(\beta\) be such that
    \(\palpha(x)\) is circle rooted for every
    \(\alpha\geq\beta\); equivalently,
    \(p_\beta(x)+\alpha(x^n+1)\) is circle rooted for every
    \(\alpha\geq 0\).

    By root correspondence, that is equivalent to say that for
    every \(\alpha\geq 0\) \(\St{p_\beta}(x)+\alpha\St{x^n+1}\)
    is real rooted .  This, in turn, happens, by
    \cref{alphacomb}, if and only if \(\St{p_\beta}(x)\) and
    \(\St{x^n+1}\) have a common interlace.

    If \(A\) is a common interlace for these polynomials, then
    \(\mu(A)\) is a common angle interlace for their
    pre-images. Reciprocally, if \(A\) is a common angle
    interlace of \(p_\beta(x)\) and \(x^n+1\) not containing
    \(\w\), \(\mu^{-1}(A)\) is a common interlace for the mapped
    polynomials.  The assumption that \(\w\not\in A\) is
    harmless; if it fails for a given \(A\), we substitute in it
    \(\w\) by some point nearby in \ucirc, and as \(\w\) is not a
    root of \(p\) or of \(x^n+1\), the new \(A\) is still a
    common angle interlace.
  \end{proof}

  An aside for some R\&R. The definition of interlace number
  \il{p} can be rephrased as ``the least \(\beta\) such that
  \(x^n-1\) angle interlaces both \(\palpha(x)\) and \(x^n+1\),
  for every \(\alpha\geq\beta\)''. This, combined with the
  preceding result, yields a very complicated and long proof of
  \cref{cnleil}.
  
  We are ready for the main result about the circle number:

  \begin{teo}\label{dbroot}
    For a  trim self-inversive polynomial \(p\) of darga n, \cn{p} is the
    largest value of \(\alpha\) for which
    \(\frac{\palpha(x)}{\gcd(p(x),x^n+1)}\) has a double root.
  \end{teo}
    \begin{proof}
    Choose \(\w\in\un\) as in Lemma \ref{goodstr}.  Let
    \(g=\gcd(p,x^n+1)\), \(P=\St{\frac{p}{g}}\),
    \(X=\St{\frac{x^n+1}{g}}\), \(G(x)=\St{g}\),
    \(Q(\alpha,x)=\St{\frac{\palpha}{g}}\), so that
    \(\St{\palpha}=G(x)Q(\alpha,x)\).  Since \(\w\in\un\),
    \(\w^n+1\neq 0\), so \(G\) and \(Q\) have the same degree
    in \(x\); further, as \(X\) is circle rooted, \(G\) is real
    rooted, so \St{\palpha} is real rooted if and only if so is
    \(Q(\alpha,x)\).  Notice also that, as all roots of \(X\) are
    distinct, the same happens to the roots of \(G\).

      Let \(\gamma=\cn{p}\); then, for all \(\alpha>\gamma\),
      \(Q(\alpha,x)=Q(\gamma,x)+(\alpha-\gamma)X(x)\) is real
      rooted, hence so is \(Q(\gamma,x)\), by 
      \cref{alphacomb}.  We claim that \(Q(\gamma,x)\) has a
      double root.  Otherwise, both \(Q(\gamma,x)\) and \(X\)
      have a common sign interlace, and that would also sign
      interlace \(Q(\gamma-\eps,x)\) for any sufficiently
      small \(\eps>0\), implying that
      \(Q(\gamma-\eps,x)\) is real rooted.  That
      contradicts the choice of \(\gamma\).

      Finally, also from \cref{alphacomb}, \(Q(\alpha,x)\) has
      \(n\) distinct roots for every \(\alpha>\gamma\), and this
      shows that \(\gamma\) is largest value of \(\alpha\) such
      that \(Q(\alpha,x)\) has a double root.  That is,
      \(\gamma\) is the largest real root of
      \(\Disc(Q(\alpha,x))\).  The result now follows from root
      correspondence.
  \end{proof}

  Our definition of interlacing has carefully allowed for the
  possibility that a polynomial interlaces itself, and that
  entails the following useful result:
  \begin{cor}\label{cnub}
    Let \(p(x)\) be a degree \(n\) full self-inversive polynomial
    with real constant term.  If \(p\) angle interlaces
    \(x^n+1\), then \(\cn{\trim p}\leq p(0)\), with equality if
    and only if \(p\) has a double root.
  \end{cor}

    The classical discriminant \(\Disc(p)\) (see
    \cite{BPR},\cite{Cohen}) of a polynomial \(p\) can be
    expressed as the determinant of the Sylvester matrix, whose
    entries are integer multiples of the coefficients of the
    polynomial.  It is well known that a polynomial has a null
    discriminant if and only if it has a double root.  So, we
    have the following corollary of \cref{dbroot}, whose
    importance is pointing to an algorithm for
    computing the circle number.

   \begin{teo}\label{discdbroot}
     For a  trim self-inversive polynomial \(p\) of darga n, \cn{p} equals
     the largest real root of
     \(\Disc\left(\frac{\palpha(x)}{\gcd(p(x),x^n+1)}\right)\cdot\)
   \end{teo}

   Or, in short (provided \(\gcd(p(x), x^n + 1)=1\)), we have an analogue of \cref{minmaxil}:
  \[
      \min\conj{\beta\phantom{\big|}}{\palpha\ \text{is circle rooted for all}\
      \alpha\geq\beta}=\max\conj{\beta\phantom{\big|}}{\Disc(\palpha)(\beta)=0}.
  \]

   In this statement, the polynomial has to be thought of as
   being in \(\R[\alpha][x]\), so the discriminant is a
   polynomial in \(\alpha\), of degree
   \(2\left(\darga{p}-1-k\right)\), where \(k=\darga{\gcd(p(x),x^n+1)}\).

   In what follows we will mention semi-algebraic sets and
   functions a few times; that will be only in recognizing that
   some relevant sets and functions have this property, and will
   not be used later.  We point the reader to \citet{BPR} for an
   in-depth treatment of the subject, and for filling the gaps in
   our proof sketches.  As an aid in understanding the
   statements, we recall that a set in \(\R^n\) is
   \emph{semi-algebraic} if it can be expressed as a finite
   boolean combination of solution sets of polynomial
   inequalities, and a function \(\R^n\rightarrow\R^m\) is
   semi-algebraic if its graph is a semi-algebraic set.

   %From the entries in the Sylvester matrix,
   \begin{cor}\label{cnsemialg}
     The map \(p\mapsto\cn{p}\) is a semi-algebraic function on \Trim{n}.
   \end{cor}
   \begin{proof}
     \textsc{(Sketch)} For any given \(m\), identify \(\R^m\) with the set of monic
     polynomials of degree \(m\).  The set \(S\) of polynomials
     with at least one real root is semi-algebraic, and the graph
     of the map \(D\mapsto\)\emph{the largest real root of \(p\)}
     can be described, in the notation of \cite{BPR}, as
     \[
     \conj{(f,\alpha)}{f(\alpha)=0 \wedge \forall x (f(x)=0 \Rightarrow x\leq\alpha)},
     \]  
     hence it is semi-algebraic, by quantifier elimination.  The
     map \(p\mapsto\Disc(\palpha)\) on \Trim{n} is also
     semi-algebraic, and we are almost done, by
     \cref{discdbroot}, but for a little correction. It remains
     to notice that, for each factor \(q\in\R[x]\) of \(x^n+1\), its
     multiples in \Trim{n} form a semi-algebraic set; an
     inclusion-exclusion argument then shows that the set
     \conj{p\in\Trim{n}}{\gcd(p,x^n+1)=q} is semi-algebraic, and
     these sets partition \Trim{n}.  Now we can actually apply the
     expression in \cref{discdbroot} to each block of the
     partition.
   \end{proof}

   A similar result holds for \CTrim{n}, viewed as a real vector space.

   The image of palindromic polynomials by the Cayley map \(S_1\)
   is special, and yields a considerable speedup in the algorithm
   for computing the circle number.  The following rewrites
   \cite[T. 4]{vieira19}, where an \emph{even} (\emph{odd})
   polynomial is a sum of terms of \emph{even} (\emph{odd})
   degree:

   \begin{pro}\label{evenodd}
     If \(p\) is a palindromic polynomial, then \(S_1(p)(x)\) is
     an even polynomial if \darga{p} is even, an odd polynomial if
     \darga{p} is odd.
   \end{pro}
   \begin{proof}
     Denote \(n=\darga{p}\), and let
     \(q(x)=S_1(p)(x)=(x+i)^np\left(\frac{x-i}{x+i}\right)\). We
     compute
     \(q(-x)=(-1)^n(x-i)^np\left(\frac{x-i}{x+i}\right)=
     (-1)^n(x-i)^np\left(\frac{x+i}{x-i}\right)=(-1)^nq(x).\)
   \end{proof}

   % \begin{pro}\label{rootseven}
   %   If \(p(x)\) is an even polynomial, then, for every nonzero
   %   root \(\beta\) of \(H(p)\), of multiplicity \(k\), then
   %   \(\pm\sqrt{\beta}\) are roots of \(p\) with the same
   %   multiplicity, and this covers all nonzero roots of \(p\).
   % \end{pro}

   In preparation for the next theorem:

   \begin{pro}\label{heckedisc}
     Let \(F\) be a field, and let \(f(x)=\sum_{j=0}^nf_jx^j\in F[x]\) have degree
     \(n\).  Then,
     \(\Disc(p(x^2))=4^nf_n^3f_0\Disc(p)^2\).
   \end{pro}
   \begin{proof}
     Let \(r_1,\ldots,r_n\) be the roots of \(f\) in the
     algebraic closure of \(F\). Then, by definition,
     \(\Disc(f)=f_n^{2n-2}\prod_{i\neq j}(r_i-r_j)\).  Let the
     two square roots of \(r_j\) be \(s_j,t_j\); then,
     \(s_j^2=t_j^2=r_j\) and \(t_j=-s_j\), and the roots of
     \(f(x^2)\) are all the \(s_j,t_j\).  It follows that
     \begin{align*}
       \Disc(f(x^2))&= f_n^{4n-2}\prod_{i\neq j}(s_i\!-\!s_j)(t_i\!-\!t_j)
                      \prod_{i\neq j}(s_i\!-\!t_j)(t_i\!-\!s_j)\prod_i(s_i\!-\!t_i)(t_i\!-\!s_i)\\
                    &= f_n^{4n-2}\prod_{i\neq j}(s_i\!-\!s_j)^2\prod_{i\neq j}(s_i\!+\!s_j)^2
                      \prod_i(-4s_i^2)\\
                    &=  f_n^{4n-2}\prod_{i\neq j}(r_i\!-\!r_j)^2.4^n\prod_i (-r_i) \\
       & = 4^nf_n^2\Disc(f)^2f_0/f_n.
     \end{align*}
   \end{proof}
   
   Let us define on \(\R[x]\) the linear operator
   \(H:\sum_{j\geq0}a_jx^j\mapsto\sum_{j\geq0}a_{2j}x^j\) (this
   is an instance of a Hecke operator, as in \cite{GR}).

 \begin{teo}\label{heckecirc}
     Suppose that \(p\) is trim palindromic of darga \(n\), and
     let \(q(x)=\frac{\palpha(x)}{\gcd(p(x),x^n+1)}\). Define
     \begin{itemize}
         \item [] \(r_1=\frac{p(-1)}{2}\).
         \item [] \(r_2 = \frac{p(-1)}{2}\) if \(n\) is even,  \(=\frac{p'(-1)}{n}\) if \(n\) is odd,
         \item [] \(r_3=\) the largest real root of \(\Disc\left(H\left(S_1\left(q\right)\right)\right)\) (if there is any).
     \end{itemize}
     Then, \[\cn{p}=\max\{r_1,r_2,r_3\}.\]
   \end{teo}
   
   \begin{proof}
     First notice that \darga{q} is even, so that
     \(s(x)=S_1\left(q\right)\) is an even polynomial, by
     \cref{evenodd}.  Indeed, if \(n\) is even, all real
     irreducible factors of \(x^n+1\) have even degree, so \(q\)
     is the quotient of two polynomials of even degree. If \(n\)
     is odd, the only odd degree factor of \(x^n+1\) is \(x+1\),
     which is a factor of any palindromic polynomial of odd
     degree, \palpha in particular; so, \(q\) is the quotient of
     two odd degree polynomials.

     Denote \(h=H(s(x))\); as \(s\) is an even polynomial,
     \(s(x)=h(x^2)\).  From \cref{discdbroot}, we know that
     \cn{p} is the largest real root of \(\Disc(s)\), and from
     \cref{heckedisc}, the roots of \(\Disc(s)\) consist of the
     roots of \(\Disc(h)\) and the roots of the leading and
     constant coefficients of \(h\) (of \(s\) as well).  As in
     the proof of \cref{realpol}, the leading coefficient is
     \(q(1)\), which is a constant multiple of
     \(\palpha(1)=2\alpha+p(1)\); for \(\alpha=-p(1)/2\) root
     correspondence breaks down, but in this case, 1 is a root of
     \palpha, hence a double root, so this makes \(-p(1)/2\) one
     of the candidates for \cn{p}.  
     % We claim
     % that the latter roots are \(-p(1)/2\) and either
     % \(-\frac{p(-1)}{2}\) or \(-p'(-1)/n\), according with the
     % parity of \(n\).  
     For the constant term, we evaluate \(s(0)\), and we have to account for
     the parity of \(n\).  If \(n\) is even, then we immediately
     get that \(s(0)\) is a constant multiple of
     \(\palpha(-1)=2\alpha+p(-1)\), hence \(-p(-1)/2\) is a
     candidate.  If \(n\) is odd, \(\palpha(x)=(x+1)t(x)q(x)\)
     for some divisor \(t(x)\) of \(x^n+1\). So,
     \(\palpha'(-1)=t(-1)q(-1)\), and \(t(-1)\) is a nonzero constant.
     Since \(\palpha'(x)= n\alpha x^{n-1}+p'(x)\),
     \(\alpha=-p'(-1)/n\) is the only root of \(s(0)=q(-1)\).
   \end{proof}

   While \cref{heckedisc} looks more cumbersome than
   \cref{discdbroot}, it is a significant simplification,
   computationally, analogous to the ``cutting the work by half''
   in \cref{wnvn}.  Computing \(H\left(S_1\left(q\right)\right)\)
   takes negligible time; the bulk of the computation lies in
   finding the discriminant of a polynomial with polynomial
   coefficients.  To simplify the analysis, suppose \(n\) is even
   and coprime with \(x^n+1\).  To compute \(\Disc(\palpha)\)
   entails evaluating a \((2n-1)\times(2n-1)\) determinant, and
   the result is a polynomial in \(\alpha\) of degree \(2n-2\).
   On the other hand, \(H\) divides the degree by 2, hence to
   compute \(\Disc(H\left(S_1\left(q\right)\right))\) requires
   only a \((n-1)\times(n-1)\) determinant, resulting in a
   polynomial of degree \(n-2\).  As the computation of of an
   \(N\times N\) determinant takes roughly \(N^3\) arithmetic
   operations, and here these operations are on polynomials in
   \(\R[\alpha]\) of degree near \(N\), which take time
   superlinear on \(N\), one can expect a speedup of at least 16
   in general.  We have leisurely implemented those in \sage with
   integer polynomials and observed even greater speedups.
   % Probably they result from the handling of ever bigger
   % numerical coefficients, besides the fact that polynomial
   % multiplication actually takes more than linear time.

   In parallel to the interlace cert, we call each double root of
   \(p_{\cn{p}}(x)\) %with nonnegative imaginary part
   a \emph{circle cert}\glsadd{ccert} of \(p\) (for palindromic
   polynomials, we only consider the double roots with
   nonnegative imaginary part).  In contrast with an interlace
   cert, which can be completely specified by a pair \((n,j)\) of
   integers (as \(\theta_n^j\)), a circle cert is almost
   unrestricted, and it is a rare case in which it can be
   expressed in a more illuminating way than just the definition.

    \begin{pro}\label{circcert}
      If \(\w\) is a circle cert of \(p\), then
      \(\w\in\ucirc\) and it is a root of
      \[
          R(p)=nx^{n-1}p(x)-(x^n+1)p'(x).
        \]
    \end{pro}
    \begin{proof}
      Let \(\alpha=\cn{p}\); then, by continuity of roots,
      \palpha is circle rooted. and that implies
      \(\w\in\ucirc\).  Furthermore, \(\w\) is a root of
      both \(\palpha(x)\) and \(p'_\alpha(x)\); eliminating
      \(\alpha \) between them yields \(R(p)\).
    \end{proof}
\vspace{2ex}

It is usually hard to bound the circle number, but here is an
occasionally useful result.  It is implicit in \cref{heckecirc},
but here is a short direct proof.

\begin{pro}\label{cnone}
  For all \(p\), \(\cn{p}\geq-\frac12 p(1)\). If \(n\) is
  even, \(\cn{p}\geq-\frac12 p(-1)\).  If \(n\) is odd, then
  \(\cn{p}\geq -\frac{p'(-1)}n=\frac1n\sum_{j=1}^{\floor*{n/2}}(-1)^{j-1}(n-2j)p_j\).
\end{pro}
\begin{proof}
  For the first two assertions, let \(c=\pm1\),
  \(\alpha=-\frac12 p(c)\).  Then, \(\palpha(c)=0\), and
  \cref{palroots} implies that \(c\) is a double root, hence
  \(\cn{p}\geq\alpha\).  For the third, there is nothing to do
  unless \(p'(-1)<0\).  In this case, let
  \(\alpha=-\frac{p'(-1)}n\).  Clearly \(p'_\alpha(-1)=0\), hence
  \(-1\) is a double root of \palpha.
\end{proof}

\begin{Example}\label{ex:geom-circ}
  Let us compute the circle number of the geometric polynomials.
  If \(n\) is even, \cref{pofone} gives us
  \(\cn{\geom_n}\geq\frac12\); on the other hand, \cref{ex:geom}
  yields \(\il{\geom_n}=\frac12\), so this is \cn{\geom_n}.  If
  \(n\) is odd, \cref{cnone} yields
  \(\cn{\geom_n}\geq\frac{n-1}{2n}\).  This bound is exact; we
  will prove directly that \(-1\) is the only circle cert, by
  computing \(R(\geom_n)\), as in \cref{circcert}.  As
  \(\geom_n(x)=\frac{x^n-x}{x-1}\),
  \(R(\geom_n)(x)=\frac{n(x^{n-1}-x^{n+1})+x^{2n}-1}{(x-1)^2}=
  \frac{x^n}{(x-1)^2}(x^n-x^{-n}-n(x-x^{-1}))\).  If
  \(x=\text{e}^{it}\in\ucirc\) is a root of \(R(\geom_n)(x)\),
  then \(\sin nt-n\sin t=0\), and by \cref{sinnx} this implies
  \(\sin t=0\).  So we must have \(x=1\) or \(x=-1\); but \(1\)
  cannot be a circle cert of a polynomial with nonnegative
  coefficients, so \(-1\) is the only circle cert.  The two
  parity cases can be coalesced in the expression
  \(\cn{\geom_n}=\frac{\floor*{n/2}}{n}\cdot\)
\end{Example}

\begin{Example}\label{ex:botta}
  \citet{BMM} studied for which real values of \(\lambda\) the
  polynomial \(f_\lambda(x)=1+\lambda(x+\cdots+x^{n-1})+x^n\),
  \(n\geq2\), is circle rooted.  We show how our general results
  here can be leveraged to obtain part of theirs.  In our terms,
  we want to know for which values of \(\alpha\)=\(1/\lambda\)\;
  \(\palpha\) is circle rooted, both for \(p=\geom_n\) and for
  \(p=-\geom_n\).  We have just determined \cn{\geom_n}; on the
  other hand \cref{pofone} immediately yields
  \(\cn{-\geom_n}=\frac{n-1}2\).  So, in terms of the original
  question, \(f_\lambda(x)\) is circle rooted provided that
  \(\frac2{n-1}\leq\lambda\leq\frac{n}{\floor*{n/2}}\).  In
  \cite{BMM} it is also shown that no \(\lambda\) outside this
  interval will do.
\end{Example}

In some cases, a full determination of the values of
\(\alpha\) for which \palpha is circle rooted can be made (see
\cref{ex:botta}).  One source of difficulty is that this set is a
collection of disjoint intervals including
{\small\([\cn{p}\!,\!\infty)\)}, as can be inferred from the
discussion after \cref{notinter}, using root correspondence.  We
have not been able to produce polynomials for which the number of
intervals grows unboundedly (neither did we try hard), but at
least we can present examples with more than one interval:

\begin{Example}\label{ex:twointerv}
  Take \(0<b<a\), and consider \(f_{a,b}(x)=ax^2+2bx+a\); it has
  two roots in the circle, with negative real part.  Let \(q\) be
  a product of \(k\geq 3\) of those (with possibly different
  parameters), and \(p=\trim q\).  If \(\alpha=q(0)\),
  \(\palpha=q\) so \palpha is circle rooted.  However, as all
  roots of \(q\) have negative real part, there is no common
  angle interlace between \(q\) and \(x^{2k}+1\); so, from
  \cref{notinter} and root correspondence, there is a larger
  value of \(\alpha\) for which \palpha is not circle rooted.
  And then it becomes circle rooted again at \(\alpha=\cn{p}\).
  %This construction also works with \(k=2\) provided we require
  %that \(\sqrt{2}b>a\).
\end{Example}

In \cref{sec:binomial-polynomials} we present other such examples.

The circle number is just one way of approaching circle rooted
polynomials.  \citet{PS} have taken a very different route from
ours, describing the set of monic self-inversive circle-rooted
polynomials, topologically and geometrically; that properly
contains the set of self-inversive polynomials with circle number
at most 1.  Our focus on \palpha can be seen as the study of the
maximal radial segments from the origin within that set (which is
not star-shaped, as they noted and \cref{ex:twointerv} shows);
that is, \(p/\cn{p}\) is the furthest point from the origin in
the corresponding segment of the ray through \(p\).

\section{Comparing the two}
\label{sec:comparing-two}

In this section, we fall back into working with palindromic
polynomials, only.  As before, some concepts and results here can
be judiciously extended to self-inversive ones.

The interlace number was created as a way of relating some known
theorems to each other, being also a convenient upper bound for
the circle number.  It is natural to wonder how good of an upper
bound it really is and that was our main driving
question.  In many cases, the bound is actually tight, and we
will say that \(p\) is \emph{exact}\glsadd{exact} if
\[\il{p}=\cn{p}.\]
Here is a criterion that can be combined with \cref{inc0} to
provide explicit examples of exact polynomials.

\begin{pro}\label{pofone}
  If either \(1\) or \(-1\) is an interlace cert of \(p\), then
  \(p\) is exact.  That is, \(C_0\) and \(C_{n/2}\) contain
  only exact polynomials.
\end{pro}
\begin{proof}
  \cref{cnone} implies \(\cn{p}\geq\il{p}\), hence equal.
\end{proof}

A special case of this is:

\begin{pro}\label{negpol}
  If all coefficients of \(p\) are nonpositive, then \(p\) is exact.
\end{pro}
\begin{proof}
  The interlace number is at least
  \(-\frac12 p(1)\), but this attains the upper bound given by
  \cref{LL}, hence \(1\) is an interlace cert.
\end{proof}

This is in stark contrast with nonnegative polynomials, which can
be as far from exact as one can think of, as we show later in
this section.

As the interlace number and the circle number are both
computable, exactness can be decided, provided equality can be
accessed from those computations.  Actually, there is a better
way to decide exactness.

\begin{teo}\label{cneqil}
  A trim palindromic \(p\) is exact if and only
  if \(p_{\il{p}}(x)\) has a double root.
\end{teo}
\begin{proof}
  Just combine  \cref{cnleil,dbroot}.
\end{proof}

So, for instance, if \(p\) has integer coefficients, almost all the
computations necessary to decide whether \(p\) is exact can be
carried out symbolically over a cyclotomic field.

\Cref{cneqil} immediately leads to a semi-algebraic
characterization of exact polynomials: for each
cone \(C_j\), \(0<j<n/2\), of the FOIC, the exact polynomials lie
in a single algebraic surface.  We refer to the description of
the FOIC, in particular to \(I_j\) and \(C_j\), as in
\cref{eq:Mp,eq:Cj}:

\begin{teo}\label{exact}
  If \(p\in C_j\), then it is exact if and only if
  \[\Disc\left(I_j(p)(x^n+1)+p\right)=0.\]
\end{teo}

One wonders, where is that double root of \cref{cneqil}?  Any
interlace cert is a natural candidate, since it is a root of
\(p_{\il{p}}(x)\), and there are ways of detecting whether that is
the case:

\begin{teo}\label{twocerts}
  Suppose that \(p\in C_j\).  The following are equivalent:
  \begin{enumerate}
      \item \label{twocerts1}\(\w=\theta_n^j\) is a circle cert for \(p\) (hence \(p\) is exact).
      \item \label{twocerts2}\(\w p'(\w)=\frac{n}2p(\w)\).
      \item \label{twocerts3}\(\sum_{k=1}^{\floor*{n/2}}\sig{p}_k(n-2k)\sin \frac{2\pi jk}n=0.\)
  \end{enumerate}
\end{teo}
\begin{proof}
  Let \(\alpha=\il{p}\); since \(p\in C_j\), \w is an interlace
  cert of \(p\), hence \(\palpha(\w)=0\).  We refer to
  \cref{circcert}: since \(\w^n=1\), the condition
  \(\w p'(\w)=\frac{n}2p(\w)\) is equivalent to \(R(p)(\w)=0\),
  which, in turn, is equivalent to \(p'_\alpha(\w)=0\).  That is,
  the condition is equivalent to \(\w\) being a double root of
  \palpha.  This shows that \Item{twocerts1} and \Item{twocerts2}
  are equivalent.

  Part \Item{twocerts2} can be restated as saying that \w is a
  zero of \(q(x)=2xp'(x)-np(x)\).  The
  \palsym-representation quickly yields
  \(q(x)=\sum_{k=1}^{\floor*{n/2}}\sig{p}_k(n-2k)(x^{n-k}-x^k)\).
  As \(\w\in\un\), \(\w^{n-k}-\w^k=-2\sin \frac{2\pi jk}n\),
  hence \(q(\w)=0\) if and only if \Item{twocerts3} holds.
\end{proof}

% \begin{pro}
%   Let \(\w\) be an interlace cert of \(p\).  Then \(\w\)
%   is a circle cert for \(p\) if and only if
%   \(\w p'(\w)=\frac{n}2p(\w)\).
% \end{pro}
% \begin{proof}
%   Let \(\alpha=\il{p}\); hence \(\palpha(\w)=0\).  We refer
%   to \cref{circcert}: since \(\w^n=1\), the condition
%   \(\w p'(\w)=\frac{n}2p(\w)\) is equivalent to
%   \(R(p)(\w)=0\), which, in turn, is equivalent to
%   \(p'_\alpha(\w)=0\).  That is, the condition is equivalent
%   to \(\w\) being a double root of \palpha.
% \end{proof}

% Note that this yields alternative proofs of \cref{interlaceroot},
% using the \(\palsym\)-represen\-tation.

All polynomials in \(C_0\) and \(C_{n/2}\) exact. Apart from
these, there is only one of darga 5, shown next, and it lies in
\(C_1\).  It was found in an entirely ad hoc manner (see
\cref{sec:darga-5}).  For darga 6, \cref{exact} enabled a
systematic search using \sage and \texttt{QepCAD}
\cite{Brown2002QEPCADBA}; we describe some of the results in
\cref{sec:darga-6}.

\begin{Example}  \label{exapol}
  Let \(p(x)=(1-\sqf)(x^4+x)+6(x^3+x^2)\).  It is well known
  that \(\cos \frac{2\pi}{5}=\frac{\sqf-1}{4}\), and
  \(\cos \frac{4\pi}{5}=\frac{-\sqf-1}{4}\), so, from
  the \Ifor, \(\il{p}=3+\sqf\), and the only interlace cert
  is \(\theta_5\).  It follows that
  \(p_{\il{p}}(x)=(3+\sqf)(x^5+1)+p(x)=(3+\sqf)  (x + 1) 
  (x^2 + \frac{1-\sqf}2x + 1)^2\).
  %(thanks \sage for the  factorization).
  So, the two roots of the trinomial are
  double roots, and by \cref{cneqil} \(p\) is exact.  Moreover,
  one of the double roots is the interlace cert.
\end{Example}

Here is a family of exact polynomials in which nether 1 or \(-1\)
is an interlace cert.

\begin{Example} \label{exapols} Let us fix an odd number
  \(n\geq 5\); we will present a family on exact darga \(2n\)
  polynomials which are neither in \(C_0\) nor in \(C_n\).
  Choose \(0<a<\frac9{n^2}\), let
  \(f(x)=\left((1-x)^2+ax\right)\left(\frac{1-x^n}{1-x\phantom{n}}\right)^2\),
  and let \(p(x)=\trim f\).  Every \(\w\in\vn\),
  \(\w\neq 1\) is a double root of \(f(x)\), which is monic;
  hence, \(\cn{p}\geq 1\).  We will show that \(\il{p}=1\) using
  \cref{interlacecert}, that is, by showing that \(f\) is
  nonnegative at the remaining \(2n^{\text{th}}\)-roots of unity.
  Clearly, \(f(1)=an^2>0\), as \(a>0\).  It remains to consider
  now \(\w\in V_{2n}\backslash\vn\); as \(\w^{2n}-1=0\) and
  \(\w^n-1\neq0\), we know that \(\w^n=-1\).  It follows that
  \(f(\w)=4\left(1+\frac{a\w}{(1-\w)^2}\right)=4\left(1+\frac{a}{2(\Rep\w-1)}\right)\).
  So, proving that \(f(\w)\geq 0\) is equivalent to showing that
  \(a\leq 2(1-\Rep\w)\).  We want this for all \w; the minimum on
  the right hand side is attained for \(\w=\theta_{2n}=\text{e}^{\frac{\pi i}n}\), so, it
  is \(2(1-\cos\frac{\pi}n)\). This, by choice, is \(>a\) if
  \(n=3\); for \(n\geq 5\), from the MacLaurin series, that
  expression is bounded below by
  \(\frac{\pi^2}{n^2}-\frac{\pi^4}{12n^4}\geq\frac9{n^2}>a\).
  So, \(p\) is indeed exact, and it belongs to
  \(C_2\cap C_4\cap\ldots\cap C_{n-1}\) but not to \(C_0\) nor
  to \(C_n\).  From this we also get for free the exact polynomial
  \(p(-x))\), which belongs
  to \(C_1\cap C_3\cap\ldots\cap C_{n-2}\).
\end{Example}

Given the examples we have seen so far, one would think that
\textsl{ a trim polynomial \(p\) is exact only if some interlace
  cert is a double root of \(p_{\il{p}}(x)\).  }  However, this
cannot be true, as \cref{exact} implies that the set of exact
polynomials in some \(C_j\)'s is not contained in a hyperplane.

The main question driving this paper is  how bad the interlace
number can be as a bound for the circle number.  As both numbers
scale linearly, their quotient is invariant under linear scaling,
and they can be thus compared without questions of
normalization. Any monotonic function of their quotient could
naturally be used to quantify the quality of the approximation;
the jury is still out on what is the ``best'' function.

We chose to define the \emph{bounding error}\glsadd{berror} of a trim
palindromic \(p\) by:
\[
    \be{p}= \frac{\il{p}}{\cn{p}}-1.
\]

Since \(0<\cn{p}\leq\il{p}\), \(\be{p}\) is nonnegative; the
\(-1\) term is added so that \(\be{p}=0\) if and only if \(p\) is
exact.
  % A large bounding error means that the polynomial is far from
  % exact.

It follows immediately from \cref{ilscaling,cnscaling} that, for
every real \(\lambda>0\) and positive integer \(k\), that
\(\be{\lambda p(x^k)}=\be{p}\), and for even darga,
\(\be{-p}=\be{p}\).

\begin{pro}\label{bebound}
  For a nonzero \(p\in\Trim{n}\),
  \[
    \be{p}\leq \frac{n-1}2\binom{n}{\floor*{n/2}}-1.
  \]  
\end{pro}
\begin{proof}
  Let \(a\) be the largest absolute value of a coefficient of
  \(p\).  From \cref{LL} we get \(\il{p}\leq \frac{n-1}2 a\);
  \cref{cnlobo} yields
  \(\cn{p}\geq a\big/\binom{n}{\floor*{n/2}}\). The result
  follows now from the definition of \be{}.
\end{proof}

\begin{Example}\label{ex:cgpal}
  Let us go back to \cref{ex:ilpal}, and compute \be{\pal{n,k}}
  where possible; as noted there, it is enough to consider the
  cases where \(n\) and \(k\) are coprime.  If \(n\) is even (and
  \(k\) odd), \cref{pofone} applies, and \(\be{\pal{n,k}}=0\).
  If \(n\) is odd, we have not been able to get a simple
  expression for the circle number, but \cref{cnone} yields
  \(\cn{\pal{n,k}}\geq 1-\frac{2k}n\).  Since
  \(\il{\pal{n,k}}<1\), this shows that the
  bounding error  approaches \(0\) as \(n\) increases.
  % Computational evidence suggests that this is true uniformly in \(k\).
\end{Example}

\begin{Example}\label{ex:geom-be}
  Consider now the geometric polynomials, where the work has
  already been done in \cref{ex:geom,ex:geom-circ}.  So, for
  \(n\) even, \(\geom_n\) is exact, whereas for odd \(n\), we get
  \(\be{\geom_n}=\frac1{n-1}\ccomma\) so \(\geom_n\) is never
  exact, but gets close.
\end{Example}

How bad can \be{} be\footnote{
  % At least, we didn't inflict the reader with hackneyed faux
  % Shakespeare. \textbf{OR}\\
  We studiously spared the reader a dose of hackneyed faux Shakespeare.}?  Let us define:
\[
  \text{BE}(n)=\sup\conj{\be{p}}{\darga{p}=n}.
\]

It follows from \cref{bebound} that \(\text{BE}(n)\) is finite.
Also, as we will see in \cref{sec:small-dargas}, the restriction
of \be{} to \Trim{n} has some discontinuities, and the \(\sup\)
cannot be substituted by \(\max\); indeed, \(\text{BE}(4)\) and
\(\text{BE}(5)\) are not attained by any polynomial.  That
section determines the value of \(\text{BE}(n)\) for \(n\leq5\).

% \textbf{Yaai!  \cref{cnlobo} has been lying around for a long
% time, looking a little sily; suddenly, it has shown itself
% useful.}

% In \cref{sec:small-degrees} we determine the value of \(\text{BE}(n)\) for \(n\leq5\).

Notice that exponent scaling implies
\(\text{BE}(n)\leq \text{BE}(m)\) if \(n|m\).  Is there an
interesting expression or tight estimate for \(\text{BE}(n)\)?  We do
not know one, neither do we know if \(\text{BE}(n)\)
is increasing, although of course (wink) it is.  Still, it is
natural to ask what is the order of growth of
\(\text{BE}(n)\). There is some wiggle room: \cref{bebound} shows
an exponential upper bound, while the proof of \cref{supBE}
yields an \(\w(n)\) %\(\w(\!\!\sqrt{n})\)
lower bound (and another, more
elaborate, shows \(\Omega(n)\). %\(\Omega(\!\sqrt{n})\)\,).
Most likely, the
very rough upper bound is wrong, and the true growth is close to
the proven lower bound.

\begin{teo}\label{supBE}
  \(\limsup_{n\rightarrow\infty}\text{BE}(n)=\infty\).
\end{teo}
\begin{proof}
  This is a direct consequence of \cref{limbeP}, which implicitly
  gives increasing lower bounds for \(\text{BE}(4n)\).
\end{proof}

We will prove the theorem by producing polynomials with fixed
circle number and providing estimates for their interlace number.

Let \(n=4m\), \(\theta=\frac{\pi}{2n}\) and
\(\w=\text{e}^{i\theta}\); notice that \(\w^{2n}=-1\). For
\(0<k<n\), let
\(F_k=(x-\w^{2k})(x-\w^{-2k})=x^2-2x\cos{2k\theta}+1\).
Define \(Q_n(x)=\prod_{0\leq j\leq m-1}F_{4j+1}\) and
\(P_n(x)=\trim Q_n^2(x)\).
\begin{pro}\label{limbeP}
  \(\be{P_n(x)}=\Omega(n^{1-\varepsilon})\), for any \(\varepsilon>0\).
\end{pro}

\begin{proof}
  For each \(k\equiv 1\pmod 4\), \(\w^{2k}\) is a double root of
  \(P_n+x^n+1=Q_n^2\).  The roots of \(x^n+1\) are
  \(\w^2,\w^4,\w^6, \ldots,\w^{2n}\); those of \(Q_n^2\) are
  \(\w^2,\w^6,\ldots,\w^{2n-2}\), all of them double.  It
  follows that \(P_n+x^n+1\) angle interlaces \(x^n+1\); hence,
  by \cref{cnub}, \(\cn{P_n}= 1\).

  The result now is equivalent to showing that
  \(\il{P_n(x)}=\Omega(n^{1-\varepsilon})\) for any
  \(\varepsilon>0\).  If \(\zeta\in\vn\), then by the \Ifor
  \(\il{P_n}\geq -\frac12P_n(\zeta)=
  -\frac12(Q_n(x)^2-x^n-1)(\zeta)=-\frac12Q_n(\zeta)^2+1\).

  In what follows, we will take \(\zeta=\w^4\).

  Substituting, \(F_k(\w^4)=\w^4(\w^4+\w^{-4}-2\cos{2k\theta})=
  2\w^4(\cos4\theta-\cos{2k\theta})=\\
  4\w^4\sin(k-2)\theta\cdot\sin(k+2)\theta\).  Hence
  \begin{align*}
    Q_n(\w^4)&=4^m\w^{4m}\prod_{0\leq j\leq m-1}\sin(4j+1)\theta\prod_{0\leq j\leq m-1}\sin(4j+5)\theta\\
             &=2^{2m}i\prod_{0\leq j\leq m-1}\sin(4j+1)\theta
               \prod_{0\leq j\leq m-2}\sin(4j+5)\theta\cdot\sin(4m-1)\theta,
  \end{align*}
  where in the last line we used the fact that
  \(\sin(4m+1)\theta=\sin(\pi/2+\theta)=\sin(\pi/2-\theta)=\sin(4m-1)\theta\).
  We now introduce the missing odd multiples of \(\theta\). Defining
  \[
    D_n=\prod_{0\leq j\leq m-2}\frac{\sin(4j+5)\theta}{\sin(4j+3)\theta},
  \]
  we can write
  \[
    Q_n(\w^4)=2^{2m}i\prod_{1\leq j\leq 2m}\sin(2j-1)\theta\cdot D_n.
  \]

  As a sanity check, observe that \(Q_n(\w^4)^2\) is a
  negative real number, so it will indeed give us a bound for
  \il{P_n}.
  % At this point we suspend the proof in order to bound the two
  % products in the last expression.

  Let \(S(n)=\prod_{k=1}^{n-1}\sin\frac{k\pi}{n}\); it is well
  known and easily proved that \(S(n) =\frac{n}{2^{n-1}}\cdot\)  It
  follows that
  \[
    \prod_{\substack{1\leq k\leq 2n-1 \\ k\ \text{odd}}} \sin k\theta=
    S(2n)/S(n)= 1/2^{n-1}=2/2^{4m}.
  \]
  But the product on the left is
  \(\left(\prod_{1\leq j\leq 2m}\sin(2j-1)\theta\right)^2\), so,
  substituting the product in the last expression for
  \(Q_n(\w^4)\), we obtain:
  \[
    Q_n(\w^4)=\sqrt{2}i D_n,
  \]
  whence
  \[
    \il{P_n}\geq 2D_n^2+1.
  \]

  %It is enough now to show \cref{limDn}, and this will
  The result will now follow from  \cref{Dnomega}.

\end{proof}

We will show that \(D_n\) grows a tad slower than \(\sqrt{n}\).
A more elaborate proof (and a lot more generality) in
\cite{mandel} implies that actually \(D_n\) is asymptotic to
\(c\sqrt{n}\) for some constant \(c\).

As a technical ingredient, we will need the following well known
asymptotics for the Gamma function:
\begin{lem}\label{prod}
  For \(0<\alpha<1\), and positive integer \(n\),
  \[
    \prod_{j=1}^n\left(1+\frac{\alpha}{j}\right)\sim  \frac{n^\alpha}{\Gamma(\alpha)}.
  \]
\end{lem}

\begin{lem}\label{Dnomega}
  \(D_n=\Omega\left(n^{\frac12-\eps}\right)\), for any \(\eps>0\).
\end{lem}

\begin{proof}
  In what follows, we take care that all arguments to the sine
  function to be in the interval \([0,\pi/2]\), where the
  function is strictly increasing.

  For a real parameter \(a\), let
  \(f(x)=\frac{\sin(a+x)}{\sin(a-x)}\); using the sine addition
  formula it follows that
  \(f'(x)=\frac{\sin(2a)}{\sin^2(a-x)}\). Let us verify that
  the first-order term of the MacLaurin series gives a lower
  bound for \(f(x)\) when \(0\leq x\leq a\), that is,
  \[
    f(x)\geq 1+\frac{\sin(2a)}{\sin^2(a)}x=1+\frac2{\tan a}x.
  \]
  Indeed, this holds for \(x=0\), and the difference
  \(f(x)- 1-\frac{\sin(2a)}{\sin^2(a)}x\) has positive derivative in the interval.

  Now we fix \(\eps>0\). Let \(0<h<\frac\pi2\) be such that
  \(\frac{\tan h}h=1+2\eps\), and let
  \(M=\floor*{\frac{nh}{2\pi}}\). For \(0\leq j\leq M-1\), let
  \(a=\frac{2(j+1)\pi}{n}\), \(x=\theta\).
  
  Then,
  \begin{align*}
    f(\theta)&=\frac{\sin(4j+5)\theta}{\sin(4j+3)\theta}\geq
               1+\frac2{\tan\frac{2(j+1)\pi}{n} }\frac{\pi}{2n}\\
             &= 1+\left(\frac{\tan 2(j+1)\pi/n}{2(j+1)\pi/n}\right)^{-1}\!\!\!\cdot\frac1{2(j+1)}\\
             &\geq 1+\frac1{2(1+2\eps)(j+1)}\ccomma
  \end{align*}
  where we used that \(\frac{\tan x}{x}\) is increasing, so
  \(\frac{\tan
    2(j+1)\pi/n}{2(j+1)\pi/n}\leq\frac{\tan h}h\cdot\)

  Finally, using Lemma \ref{prod},
  \[
    D_n\geq \prod_{1\leq j\leq M}\left(1+\frac1{2(1+2\eps)j}\right)
    = \Omega\left(M^{\frac1{2(1+2\eps)}}\right)
    = \Omega\left(M^{\frac12-\eps}\right),
  \]
  from which we conclude the desired result.
\end{proof}
      
%       We also bound \(D_n\) from above, showing that
%       \cref{sinprod} is in the right ballpark.
      
% \begin{lem}
%    \(D_n=O\left(\sqrt{n}\right)\).
% \end{lem}
% \begin{proof}
% Given reals \(0<\beta<\alpha\), consider
% \(g(t)=\frac{\sin\alpha t}{\sin\beta t}\).  Elementary
% algebraic manipulation shows that \(g'(t)<0\) for
% \(0\leq t<\frac\pi2\), so the maximum of \(g(t)\) in this
% interval is attained at \(0\), hence \(g(t)\leq \frac\alpha\beta\).

% Consider the product
% \begin{align*}
%   E_k(t)&=\prod_{0\leq j\leq k}\frac{\sin(4j+5)t}{\sin(4j+3)t}\\
%         &=\frac{\sin5t}{\sin3t}\prod_{1\leq j\leq k}\frac{\sin(4j+5)t}{\sin(4j+3)t}\\
%         &\leq \frac53\prod_{1\leq j\leq k}\frac{4j+5}{4j+3}\\
%         &\leq\frac53\prod_{1\leq j\leq k}\left(1+\frac1{2j}\right)\\
%         &= O\left(\sqrt{n}\right),
% \end{align*}
% by Lemma \ref{prod}. It follows that
% \(D_n=E_{n/4}(\theta)<\frac56e^{\frac\gamma2}\sqrt{n}=(1.112+\eps)\sqrt{n}\).
% \end{proof}

One easily sees from the construction that \(P_n(-x)\) has
nonnegative coefficients.  Since those polynomials have even
darga, this change does not modify the bounding error.
Incorporating to this the result in \cite{mandel} we have

\begin{pro}
  \(\text{BE}(n)=\Omega(n)\) even if we restrict its definition
  to nonnegative polynomials.
\end{pro}
      
A couple of comments about the proof of \cref{limbeP}: it looks
like  (experimentally, we have not proved it) that the \(\zeta\) we chose
is indeed an interlace cert; that being true, \il{P_n} grows
linearly in \(n\), while the upper bounds in \cref{LL} and the
like are exponential.

\section{Small dargas}
\label{sec:small-dargas}

Here we look at palindromic polynomials of darga up to 6.  A
detailed analysis (including pretty pictures) is possible, due to
the low dimension of the spaces involved; for self-inversive, the
dimensions double, and everything becomes unmanageable rather quickly.

\subsection{Darga 2}
\label{sec:darga-2}

Here \(p\) is just a monomial of degree 1, and, up to scaling, it
is either \(2x\) or \(-2x\), both exact by \cref{pofone}, hence
\(\text{BE}(2)=0\). It is, however, instructive to appreciate the
dynamics of \palpha.  We just do it for the \(-2\) case. Since
\(\palpha(x)=\alpha x^2- 2 x+\alpha\), its roots are
\((1\pm\sqrt{1-\alpha^2})/\alpha\). As \(\alpha>0\) increases,
one can see, while \(\alpha<1\), two real roots, straddling 1;
for \(\alpha=1\) the roots have converged to a double root, and,
for \(\alpha>1\) we have two roots in the circle,
\(\frac{1}{\alpha}\pm i\sqrt{1-\frac{1}{\alpha^2}}\), approaching
\(\pm i\) as \(\alpha\) increases.

\subsection{Darga 3}
\label{sec:darga-3}

Up to scaling, \(p(x)=\pm 2(x^2+x)\).  The minus sign case is
covered by \cref{pofone}, the polynomial is exact, and \(1\) is
the only interlace cert.  For the plus sign, we work directly.

To compute \il{p} using the \Ifor, we need to evaluate \(p\) at
the cubic roots of unity, namely \(1,\w, \w^2\), where
\(\w+ \w^2=-1\).  It is immediate that \(p(1)=4\),
\(p(\w)=p(\w^2)=-2\), whence, \(\il{p}=1\).  To compute \cn{p},
we factor \(\palpha(x)=(x+1)f(x)\), where
\(f(x)=\alpha x^2-(\alpha-2)x+\alpha\), and seek \(\alpha\) such
that \(f(x)\) has a double root.
\(\Disc f(x) = (\alpha-2)^2-4\alpha^2\), so a double root
requires \(\alpha-2=\pm 2\alpha\). The only positive solution to
that is \(\alpha=\frac23\), so that \(\cn{p}=\frac23\), witnessed
by the circle cert \(-1\), and  showing that
\(\text{BE}(3)=\frac12\cdot\)

\subsection{Darga 4}
\label{sec:darga-4}

Here things start to look interesting.  The
\palsym-representation is
\(p(x)=b\pal{1}+c\pal{2}\).

The FOIC here will be defined based on the functionals
\(I_0(p)=-b-c, I_1(p)=c, I_2(p)=b-c\), and the three cones are given by:
\[
\begin{array}{ll@{\,}ll@{\,}l}
  C_0: &-b -2c&\geq 0,& -b&\geq 0, \\
  C_1: & \phantom{-}b+2c&\geq 0, &-b+2c&\geq 0,\\
  C_2: & \phantom{-}b&\geq 0, &\phantom{-}b-2c&\geq 0.
\end{array}
\]

\newlength{\imh}

% \begin{minipage}[t]{\textwidth}

\begin{figure}[htb]
  \centering
  \settoheight{\imh}{\includegraphics[width=0.595\textwidth,keepaspectratio=true]{darga4}}
    \begin{subfigure}[h]{0.595\textwidth}
    \settoheight{\imh}{\includegraphics[width=\textwidth,keepaspectratio=true]{darga4}}
    \includegraphics[width=\textwidth,keepaspectratio=true]{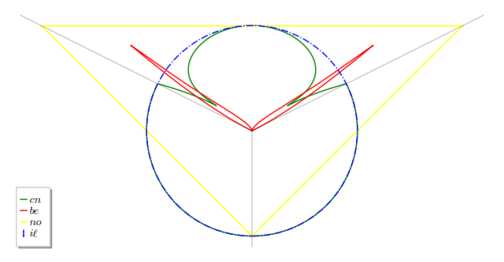}
     % triplot(4)
    \caption{All numbers, darga 4    \label{fig:darga4}}
  \end{subfigure}
  \begin{subfigure}[h]{0.395\textwidth}
    \vbox to \imh{
      \includegraphics[height=\imh]{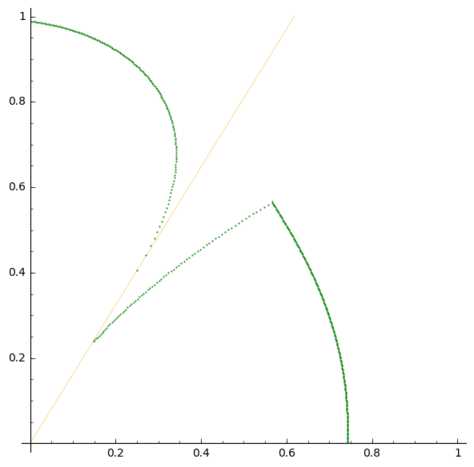}
      \vfill}
    % scatter_polar(lambda t: cnoveril(pf45(n,t)),0,pi/2,500,col=cor['cn'])+
    % line([(0,0),((sqrt(5)-1)/2,1)], color=cor['de'],thickness=.3)
    % I wonder why
    \caption{Darga 4, \(\cn{}/\il{}\)    \label{fig:cnoveril4}}
  \end{subfigure}\\
  \begin{tabular}[h]{p{\textwidth}}
    \tiny
    Axes: \(b\) --- horizontal, \(c\) --- vertical.
    Black rays are facets of the interlace fan.
    The polar simplex  \st{\hat{I}} is yellow; \il{}, \cn{}
    and \be{} are computed on \st{\hat{I}} simplex, on each ray.  As
    \(\il{}=1\) on \st{\hat{I}}, there is a blue circle, while the
    green curve also represents \(\cn{}/\il{}\) on each ray.
    The figure on the right is a piece of the green curve.
  \end{tabular}
\caption{Darga 4 polynomials}
% \label{fig:subfigureexample}
\end{figure}

By \cref{pofone}, in \(C_0\) and \(C_2\) the polynomials are all
exact.  We concentrate on  \(C_1\).

Here, the inequality \(c>0\) is valid
if \(p\neq 0\), and we will scale \(p\) so that \(c=1\), that is
\(p(x)=bx^3+2x^2+bx\); as a bonus, \(\il{p}=I_1(p)=1\).  Also,
the defining inequalities of \(C_1\) can be abridged as
\(|b|\leq 2\).

In order to compute \cn{p} and \be{p}, we use \cref{heckecirc}.
Up to a constant factor,
\(H\left(S_1\left(\palpha\right)\right)=(\alpha + b + c)*x^2 + (-6*\alpha + 2*c)*x + \alpha - b + c\), its discriminant (ditto) is \(8*\alpha^2 - 8*\alpha + b^2\), and \cn{p} is the largest of
\[
  r_1=-b-1,\quad r_2=b-1,\quad r_3={\scriptstyle \frac{1+\sqrt{1-\frac12b^2}}2}.
\]
More simply, \(\cn{p}=\max\{|b|-1,r_3\}\), where \(r_3\) is only considered for \(|b|\leq\sqrt2/2\).  

Since \(|b|\leq 2\), there are two separate cases:
\begin{enumerate}
    \item \(|b|\leq\sqrt{2}\).  
  Since \(|b|-1\leq\sqrt{2}-1<\frac12<r_3\), \(\cn{p}=r_3\).  It follows that
  \(\be{p}=\frac{2}{\scriptstyle\left(1+\sqrt{1-\frac{\text{\phantom{i}}b^2}{2}}\right)}-1\leq
  1\), and the bound is attained for \(b=\pm\sqrt{2}\).  Also,
  in this interval there is a single exact polynomial,
  specified by \(b=0\), namely \(p(x)=2x^2\).

    \item \(\sqrt{2}<|b|\leq 2\). In this case, there is no
  \(r_3\) to consider, so \(\cn{p}=|b|-1\), and
  \(\be{p}=\frac1{|b|-1}-1<\sqrt{2}\).  As
  \(|b|\rightarrow\sqrt2\), \(\cn{p}\rightarrow \sqrt2-1\) and
  \(\be{p}\rightarrow\sqrt2\), but not attained.
\end{enumerate}

The previous paragraphs show that \cn{} and \be{}, are
discontinuous at the rays \(b=\pm\sqrt{2}c\), and this is illustrated in
\cref{fig:cnoveril4}.

It follows that \(\text{BE}(4)=\sqrt{2}\).

\subsection{Darga 5}
\label{sec:darga-5}

The general form is \(p(x)=b\pal{5,1}+c\pal{5,2}\).

It is well known that \(\cos \frac{2\pi}{5}=\frac{\sqf-1}{4}\),
and \(\cos \frac{4\pi}{5}=\frac{-\sqf-1}{4}\).  So,  from
\cref{eq:Cj}, defining
\(\gamma=\frac{3+\sqf}2\), we find that  the interlace fan comprises the
cones:
\newcommand{\fm}{\phantom{-}}
\[
  \begin{array}{l@{\,}r@{\,}r@{\;}l}
    C_0\!\!:&-b-\gamma c\geq 0,& -\gamma b-c \geq 0,&\text{on which}\ \il{p}=-(b+c).\\
    C_1\!\!:&b+\gamma c\geq 0,& -b+c \geq 0, &\text{on which}\ \il{p}=\left((b+c)+\sqf(c-b)\right)/4.\\
    C_2\!\!:&\gamma b+ c\geq 0,& b-c \geq 0, &\text{on which}\ \il{p}=\left((b+c)+\sqf(b-c)\right)/4.
  \end{array}
\]  
In order to compute \cn{p} and \be{p}, we use \cref{heckecirc}.
Up to a constant factor,
\(H\left(S_1\left(\frac\palpha{x+1}\right)\right)=(\alpha + b + c)x^2
+ (-10\alpha - 2b + 2c)x + 5\alpha - 3b + c\), its discriminant (ditto)
is \(5\alpha^2 -(4c-2b)\alpha + b^2\), and \cn{p} is the largest of
\[
  r_1=-b-c,\quad r_2=\frac{3b-c}5,\quad r_3=\frac15\left(2c-b+2\sqrt{c^2-bc-b^2}\right),
\]
where \(r_3\) can only contribute to the circle number if \(c^2-bc-b^2\geq 0\).

We analyze each cone separately:

\noindent \(C_0:\) Here all polynomials are exact, so there is
nothing interesting to be said.

\noindent \(C_2:\) Here, \(b> 0\) is a valid inequality for all
nonzero polynomials, so we scale, making \(b=1\).  Hence, we have
\(p(x)=\pal{5,1}+c\pal{5,2}\), \(-\frac{3+\sqf}2\leq c \leq 1\),
and \(\il{p}=\frac14\left((1-\sqf)c+1+\sqf\right)\).  Also,
\(r_3\) exists either for \(c\geq (1+\sqf)/2\), which is out of
\(C_2\), or for \(c\leq (1-\sqf)/2\), on which \(r_3<0\).  So, it
turns out that \(\cn{p}=\max\{r_1,r_2\}\). Each of them can be
the largest, so we have:
  \[
    \cn{p}=
    \begin{cases}
      {\scriptstyle -1-c} & \text{if}\quad\;\, c\leq-2 \\
      \phantom{.}\frac{3-c}5 & \text{if}\ -2\leq \phantom{-}c,
    \end{cases}
  \]
  and it happens to be continuous.
  Computing \be{p} in each case, we find that its maximum in
  \(C_2\) is attained at \(c=-2\) and its value is
  \(\frac{3\sqf-5}4\cdot\)

\noindent \(C_1:\) Here, \(c> 0\) is a valid inequality for all
nonzero polynomials, so we scale, making \(=1\).  Hence, we have
\(p(x)=b\pal{5,1}+\pal{5,2}\), \(-\frac{3+\sqf}2\leq b \leq 1\),
and \(\il{p}=\frac14((1-\sqf)b+1+\sqf)\).  There exists an
\(r_3\) in the interval
\(\left[-\frac{3+\sqf}{2},-\frac{1+\sqf}2\right]\), and this
splits the interval \( [-({3+\sqf})/{2},1]\) into three parts, as
follows:
    \begin{itemize}
        \item []
      \hspace{-4ex}\(-\frac{3+\sqf}{2}\leq b <-\frac{1+\sqf}2\):
      \(r_1=-b-1\) is the largest root; so, as
      \(b\rightarrow-\frac{1+\sqf}2\),
      \(\cn{p}\rightarrow \frac{\sqf-1}2\) and \(\be{}\rightarrow\sqf+1\); further, \be{} is
      increasing and bounded above by the value shown.
        \item []
      \hspace{-4ex}\(-\frac{1+\sqf}2\leq b \leq \frac{\sqf-1}2\):
      in this interval \(r_3\) exists, and it is the largest
      root.  Note that for \(b=-\frac{1+\sqf}2\),
      \(\cn{p}=\frac{5+\sqf}{10}\) and \(\be{p}=\frac{1+\sqf}4\);
      for \(b=\frac{\sqf-1}2\), \(\cn{p}=\frac{5-\sqf}{10}\) and
      \(\be{p}=\sqf-1\), which is maximum in the interval.

      Also, in this interval, there exists a single exact
      polynomial (just find \(b\) such that \(\il{p}=r_3\)),
      specified by \(b=\frac{1-\sqf}6\) (see \cref{exapol}).
        \item [] \hspace{-4ex}\( \frac{\sqf-1}2<b\leq 1\): here
      \(r_2=\frac{3b-1}5\) is largest; as
      \(b\rightarrow \frac{\sqf-1}2\),
      \(\cn{p}\rightarrow \frac{3\sqf-5}{10}\) and
      \(\be{}\rightarrow\frac{3+\sqf}{2}\), which is an upper
      bound for \be{}, but not attained.
    \end{itemize}

% Nothing interesting happens on \(C_0\), where we know that all
% polynomials are exact.  It is tempting to analyze \(C_1\) and
% \(C_2\) separately (in particular because \(c\geq 0\) and
% \(b\geq 0\) are valid inequalities for each, respectively).
% However, it is more expedite to split cases according to the sign of \(c\).

\newlength{\imw}
\setlength{\imw}{0.55\textwidth}
  \begin{figure}[htb]
    \centering
       \settoheight{\imh}{\includegraphics[width=\imw,keepaspectratio=true]{darga5}}
    \begin{subfigure}[h]{\imw}
      \settoheight{\imh}{\includegraphics[width=\textwidth,keepaspectratio=true]{darga5}}
      \includegraphics[width=\textwidth,keepaspectratio=true]{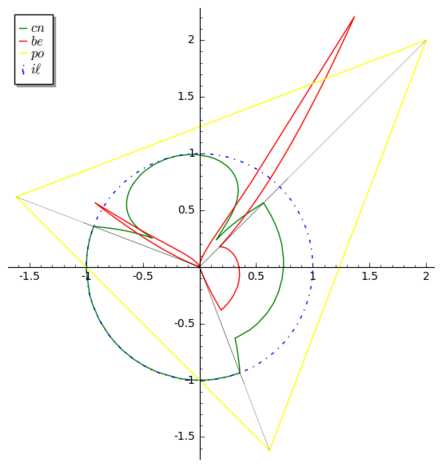}
      % triplot(5)
      \caption{All numbers, darga 5}
      \label{fig:darga5}
    \end{subfigure}
    \begin{subfigure}[h]{0.44\textwidth}
      \vbox to \imh{
        \vfill
        \includegraphics[width=\textwidth]{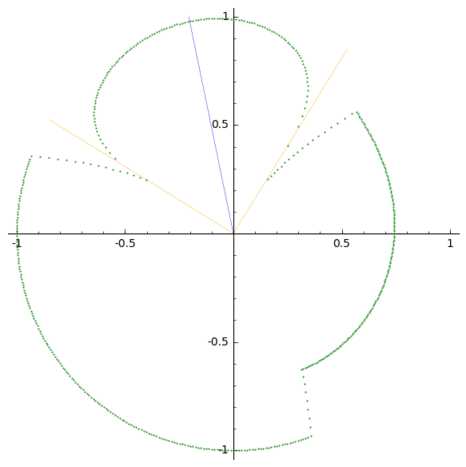}
        \vfill}
      %cna=scatter_polar(lambda t: cnoveril(pf45(5, t)),0,pi*2,500,col=cor['cn'],aspect_ratio=1)
      %cna+discontinuities(5)+exact5()+ifan(5)
      \caption{Darga 5, \(\cn{}/\il{}\)}
      \label{fig:cnoveril5}
    \end{subfigure}
    \caption{Darga 5 polynomials}
    %\label{fig:subfigureexample}
  \end{figure}
So, we note that \cn{} and \be{} have two rays of
discontinuity, spanned by \(\frac{-1\pm\sqf}2\pal{5,1}+\pal{5,2}\).
\Cref{fig:darga5} follows the same conventions as \Cref{fig:darga4}.
\Cref{fig:cnoveril5} shows \(\cn{}/\il{}\), along with the fan
rays, discontinuity rays and the exceptional exact polynomial.

The whole analysis also implies that
\(\text{BE}(5)=\frac{3+\sqf}2\cdot\)

  \subsection{Darga 6}
\label{sec:darga-6}

The FOIC corresponding to darga 6 has been explicitly described
in \cref{ex:foic6}, and it happens to be a rational
cone. Unfortunately, a general study of the circle number seems
to be too complicated to be worth it.  We just recall our
recurring example \(p(x)=50x^5 + 86x^4 + 99x^3 + 86x^2 + 50x\),
which lies in \(C_1\) and has bounding error exactly \(4\). It
would be really nice if it was maximum, but numerical
experimentation led to \(23p-x^3\), whose bounding error yields:
\(\text{BE}(6)\geq 4.00645161290323\).  We also used \sage and
\texttt{QepCAD} \cite{Brown2002QEPCADBA} to find out about the
exact polynomials in cones \(C_1\) and \(C_2\), using the
characterization in \cref{exact}.  The result was two components
on each cone, with rather complicated descriptions.
  
\section{Some interesting families}
\label{sec:some-examples}

Here we present some families of polynomials for which something
nontrivial can be said about interlace or circle numbers.  The
first three sections have a number-theoretic flavor, the other
ones are more analytic.

\subsection{Gcd polynomials}
\label{sec:gcd-polynomials}

This family of polynomials was the starting point for the research being presented here.
\citet{DR2015} studied the following family of polynomials
\[
  \text{gcd}_n^{(k)}(x)=\sum_{j=0}^n\gcd(n,j)^kx^j.
\]
At that time, the notions of interlace and circle numbers had not
yet come up; the main result there can be stated as
\[
        \il{\text{tgcd}_n^{(k)}(x)}\leq n^k,    
\]
where \(\text{tgcd}_n^{(k)}(x)=\trim{\text{gcd}_n^{(k)}(x)}\),
but one can extract a more precise result from what is proved
there.

Jordan's totient function is
\[
  J_k(n)=n^k\!\!\!\!\prod_{\substack{p|n\\ p\
      \text{prime}}}\left(1-\frac1{p^k}\right),
\]
and, in particular, \(J_1(n)=\varphi(n)\).

It is stated as Prop.1 of \cite{DR2015} that for every \(m\)
\[
  n^k+\text{tgcd}_n^{(k)}\left(\eipi{m;n}\right)=\sum_{d|\gcd(m,n)}dJ_k\left(\frac{n}{d}\right).
\]

It is easy to see that the sum on the right is minimized if
\(\gcd(m,n)=1\), where its value is simply \(J_k(n)\).  It
follows that
\begin{equation}
  \label{eq:ilgcd}
    \il{\text{tgcd}_n^{(k)}(x)}=\frac{n^k-J_k(n)}2 \ccomma
  \end{equation}
  and the interlace certs are precisely the elements of \vn that
  generate the group \un.  Notice that the interlace number is an
  integer if \(n\) is even.  Moreover, this value is generally
  much smaller than the upper bounds provided in
  \cref{sec:interlace-number}.  Also, if \(n\) is prime,
  \(\text{tgcd}_n^{(k)}(x)=\geom_n(x)\), and indeed
  \cref{eq:ilgcd} yields the value \(\frac12\) for the interlace
  number, as expected.  The next result ties in with the
  following section.

\begin{pro}
  Let \(a(x)=\sum_{k=0}^ma_kx^k\) be a polynomial with nonnegative real coefficients. Then,
  \[
    \il{\sum_{j=0}^na(\gcd(n,j))x^j} = \frac12\left(a(n)-\sum_{k=0}^ma_kJ_k(n)\right).
  \]  
\end{pro}
\begin{proof}
  First, note that \(J_0(n)=0\) if \(n>1\), and \cref{eq:ilgcd}
  holds true if \(k=0\).  Furthermore, \(\eipi{;n}\) is an
  interlace cert for all \(\text{tgcd}_n^{(k)}\), so all these
  polynomials are in the cone \(C_1\) of the interlace fan.  So
  are all nonnegative linear combinations of those polynomials,
  and, as the interlace number is linear on \(C_1\), the result
  follows from \cref{eq:ilgcd}.
\end{proof}
% Another curious class:  with \(n\) odd, consider the polynomial\\
% \(\gcd_n(x)/(x+1)\).  This also seems to satisfy \(\il{}=\cn{}\),
% and the proof is again at \(-1\).

We have no idea on how to express \cn{\text{tgcd}_n(x)}, but experiments
suggest that \be{\text{tgcd}_n(x)} is always positive, and goes to 0 with
\(n\).

\subsection{Interlace rational polynomials}
\label{sec:interl-rati-polyn}

We will use, exclusively in this section, the notation \((r,k)\)
to stand for \(\gcd(r,k)\); \(n\) will be fixed throughout.
Recall that for every \(1\leq r<n\), \(\tn^r\) is
a primitive \(\frac{n}d\)th root of unity, where \(d=(r,n)\), and
its conjugates over \Q are \conj{\tn^{kd}}{(k,n/d)=1}. We
refer to those sets as \emph{conjugate classes} of roots of
unity; these classes can be indexed by the divisors of \(n\), the
\(d\)-class being \conj{\tn^{k}}{(k,n)=d}.

The \emph{coprime-support polynomials} are the stars of this section and  are defined as:
\[
  R_n(x)=\sum_{\substack{0<k<n\\(k,n)=1}}x^k.
\]  
Clearly, \(R_n(1)=\varphi(n)\).  More generally, for every \(j\),
\(c_n(j)=R_n(\tn^j)\) is called a \emph{Ramanujan sum}
\cite{SchSp}.  This is known to be integer, and indeed
\begin{equation}
  \label{eq:vonserneck}
  c_n(j) = \mu\left(\frac{n}{(n,j)}\right)\frac{\varphi(n)}{\varphi\left(\frac{n}{(n,j)}\right)}\cdot
\end{equation}

% \noindent\textbf{-----------------------------------------------------------------------------------------\newline
%   This proposition, along with the necessary definitions above, deserves to be moved to \cref{sec:interlace-number}, next to the upper bounds}

The following lower bound for the interlace number turns out to
be racional if the polynomial has racional coefficients.
\begin{pro}\label{rambound}
  Let \(p\) be a darga \(n\)  trim palindromic polynomial.  Then,
  \begin{equation}
    \label{eq:rambound}
    \il{p}\geq \max_{d|n}\frac{-1}{\varphi(d)}\sum_{j=1}^{\floor{n/2}}\sig{p}_jc_d(j).
  \end{equation}
\end{pro}
\begin{proof}
  We know, by the \Ifor, that
  \(\il{p}\geq -\frac12 p(\w)\) for every \(\w\in\un\).
  For each \(d\), consider the class of primitive \(d\)-roots of
  unity; averaging those inequalities over the class gives
  \(-\frac{1}{2\varphi(d)}\sum_{j=1}^{n-1}p_jc_d(j)\).  Noticing
  that \(c_d(n-j)=c_d(j)\), the last step (to the
  \palsym-representation) follows.
\end{proof}

%\noindent\textbf{----------------------------------------------------------------------------------------------}

A trim palindromic polynomial with rational coefficients will be
said to be \emph{interlace rational} provided its interlace
number is rational.  If such a polynomial has integer
\palsym-coefficients, \(\il{p}\) is an algebraic integer, hence,
an integer.  Clearly, any rational polynomial in \(C_0\) or
\(C_{n/2}\) is interlace rational.  More generally:

\begin{pro}\label{icorbit}
  If \(p\) is interlace rational of darga \(n\), then the set of
  interlace certs of \(p\) is a union of conjugate classes.
  Furthermore,
  \begin{equation}
    \label{eq:icorbit}
    \il{p}= \max_{d|n}\frac{-1}{\varphi(d)}\sum_{j=1}^{\floor{n/2}}\sig{p}_jc_d(j).
  \end{equation}
\end{pro}
\begin{proof}
  \cref{orbit} implies that if \(\w\) is an interlace cert,
  so are all its conjugates. \cref{rambound} tells us that the
  right hand side is a lower bound. Now, choose \(d|n\) such that
  \(\il{p}=-\frac12p(\tn^{n/d})\) and apply \cref{eq:pthetad}
  to see that it is attained.
\end{proof}

\begin{lem}\label{orbit}
  Let \(p\) be a polynomial of with rational coefficients.  If for some \(r\),
  \(p(\tn^r)\in\Q\), then \(p(\tn^{kd})=p(\tn^r)\)
  for all \(k\) such that \((k,n/d)=1\), where \(d=(n,r)\).  Moreover,
  \begin{equation}
    \label{eq:pthetad}
    p(\tn^d)=\frac1{\varphi\left(n/d\right)}\sum_jp_jc_{n/d}(j).
  \end{equation}
\end{lem}
\begin{proof}
  Since \(\tn^r\) is a root of \(p(x)-p(\tn^r)\), so are all
  its conjugates, hence \(p(\w)=p(\tn^d)\) for all
  \(\w\) in the conjugate class of \(\tn^d\).  Averaging
  this equation over the conjugate class, we get \cref{eq:pthetad}.
\end{proof}

Coprime-support polynomials afford a simple formula for the interlace number:

\begin{pro}\label{ilRn}
  Let \(q\) be the smallest prime factor of \(n\). Then
  \(\il{R_n}=\frac{\varphi(n)}{2(q-1)}=\frac{R_n(1)}{2(q-1)}\),
  and its interlace certs are the primitive \(q\)-roots of unity.
\end{pro}
\begin{proof}
  From the \Ifor,
  \(\il{R_n}\!=\!-\frac12\min_jc_n(j)\!=\!-\frac12\min_{j|n}c_n(j)\).
  The last equality follows from \cref{eq:vonserneck}, which
  implies \(c_n(j)=c_n((n,j))\), so the values of \(c_n\) are
  attained at divisors of \(n\).  If \(j\) is such that
  \(c_n(j)\) is minimum, one must have \(\mu(n/j)=-1\), so \(n/j\)
  is a product of an odd number of distinct primes; also,
  \(\varphi(n/j)\) must be minimum. Those two conditions imply
  that \(n/j=q\), and the result follows from the fact that
  \(\varphi(q)=q-1\), and that the interlace certs are the proper
  powers of \(\tn^j\).
\end{proof}

Here, again, \(n\) prime get us back to the geometric polynomial,
with the expected result.

Let \(p\) be a trim darga \(n\) polynomial.  Trivially,
\[
  p(x)=\sum_{d|n}\sum_{\substack{d\leq j<n\\(j,n)=d}}p_jx^j
  =\sum_{d|n}\sum_{\substack{1\leq j<n/d\\(j,n/d)=1}}p_{jd}x^{jd}.
\]

%There is something TODO with the representation above.
An arithmetic function \(f:\N\rightarrow\C\) is said to be
\emph{\(n\)-even} (\cite{SchSp},\cite{LTPH}) if for all \(j\),
\(f(j)=f((j,n))\).  We attach the same name to polynomials \(p\)
of degree \(\leq n\) such that for all \(j\), \(p_j=p_{(j,n)}\).
Then, we can write
\begin{align}
  p(x) &= \sum_{d|n}\:p_d\!\!\!\!\sum_{\substack{d\leq j<n\\(j,n)=d}}x^j \\
    &=\sum_{d|n}\:p_d\!\!\!\!\!\sum_{\substack{1\leq j<n/d\\(j,n/d)=1}}x^{jd}\\
    &=\sum_{d|n}p_dR_{n/d}(x^d).\label{Rpol}
\end{align}
So, for every \(k\),
\begin{equation}
  \label{eq:ramanu}
  p(\tn^k)= \sum_{d|n}p_dc_{n/d}(k),
\end{equation}
which is the main result in \cite{Sch}.

It follows from \cref{ilnumber} and \cref{eq:ramanu} that
rational \(n\)-even polynomials are interlace rational.  The
converse is not true. On one hand, \(n\)-even polynomials lie in
a rational linear subspace of dimension equal to the number of
divisors of \(n\); on the other hand, \(C_0\) and \(C_{n/2}\) are
\(\floor*{n/2}\)-dimensional cones whose rational points are all
interlace rational.  However, those cones are quite special, and
\cref{icorbit} restricts the faces of the FOIC that contain
interlace rational polynomials.  We show next how to construct
some of those polynomials in thinner faces.

\begin{Example}\label{notfourat}
  Let \(n\) be an odd number and \(q\) be its smallest prime
  factor.  Let \(g\) be any rational darga \(n/q\) trim
  palindromic polynomial such that \(q(1)=0\).  Consider now
  \(p(x)=R_n(x)+\lambda g(x^q)\), where \(\lambda\) is a positive
  rational.  \cref{ilRn} tells us that any primitive \(q\)-root
  of unity \w is an interlace cert for \(R_n(x)\).  The choice of
  \(g\) implies that \(p(\w)=R_n(\w)\), and with \(\lambda\)
  small enough we can guarantee that \(p(\w)=\min_rp(\tn^r)\),
  so it follows that \(\il{p}=\il{R_n}\), with the same interlace
  certs.  For concreteness, choose an odd prime \(q\) and an
  integer \(n\) such that \(q\) is the smallest prime factor of
  \(qn\) and let
  \(p(x)=R_{qn}(x)+\pal{n,1}(x^q)-\pal{n,2}(x^q)\).  One can show
  (exercise for the reader) that, provided
  \(\varphi(n)-\varphi(q)\geq 4\), \(p(x)\) is interlace rational;
  also, the same construction, but violating the inequality,
  yields polynomials which are not interlace rational.
\end{Example}

% Removed
% This is cor 2 of \cite{LTPH}
% The \emph{Fourier Transform of order \(n\)} is the
% linear map \(p\mapsto\hat{p}\) defined on the vector space of
% complex polynomials of degree \(<n\) given by
% \(\hat{p}_j=p(\tn^j)\). The special behavior of \(n\)-even
% polynomials under the Fourier transform has been thoroughly
% studied, but the following simple observation seems to have
% fallen through the cracks.

% \begin{pro}
%   The Fourier transform of a polynomial \(p\) is a rational
%   polynomial if and only if \(p\) is an \(n\)-even rational polynomial.
% \end{pro}
% \begin{proof}
%   From \cref{eq:ramanu} it follows immediately that the Fourier
%   transform of an \(n\)-even rational is rational.  For the
%   converse, consider the polynomial \(\hat{p}\).  By hypothesis
%   it is a rational polynomial, and its Fourier transform
%   \(\hat{\hat{p}}=np\) is also rational.  So we can apply
%   \cref{orbit} to \(\hat{p}\), and conclude that \(p\) has the
%   desired form.
% \end{proof}

A further consequence of \cref{eq:ramanu} is that if the
coefficients of \(p\) are integers, so are the coefficients of
\(\hat{p}\), hence \(2\il{p}\) is an integer.

\begin{Example}
  The geometric  polynomial is is trivially \(n\)-even, and
  \(\geom_n(x)=\sum_{d|n}R_{n/d}(x^d)\).  It follows that
  \[ R_n(x)=\sum_{d|n}\mu(d)\geom_{n/d}(x^d).
  \]  
\end{Example}

\begin{Example}
  The trimmed gcd polynomials are \(n\)-even, and can be simply written as
  \[
    \text{tgcd}_n^{(k)}(x)=\sum_{d|n}d^kR_{n/d}(x^d).
  \]
\end{Example}

\begin{Example}
  Suppose \(n=q^m\), with \(q\) prime, and that \(p\) is
  \(n\)-even; write \(a_j=p_{q^j}\).  Then from \cref{Rpol}
  \[
    p(x)= \sum_{j=1}^ma_jR_{q^{m-j}}(x^{q^j}).
  \]  

  From \cref{ilRn}, \(\il{R_n}=\frac12 q^{m-1}\), and its
  interlace certs are the \(q\)-th roots of unity.  Exponent
  scaling implies that
  \(\il{R_{q^{m-j}}(x^{q^j})}\!=\frac12 q^{m-j-1}\), and,
  further, the \(q\)-th roots of unity are also interlace certs
  of all those polynomials.  Hence the polynomials are all in the
  face of the interlace fan corresponding to all those roots, so
  that, provided the \(a_j\) are nonnegative,
    \[
      \il{p} =\frac12\sum_{j=1}^ma_jq^{m-j-1}.
    \]  
\end{Example}

\begin{Example}
  This is a more general form of the previous example. Start with
  a positive integer \(n\), and denote by \(q\) its
  smallest prime divisor. Let reals \(a_d\geq 0\) be given for
  every \(d\) such that \(q|d\) and \(d|n\).  Then, if
  \(p(x)=\sum_{q|d|n}a_dR_{\frac{n}d}(x^d)\), its interlace
  number is
  \(\il{p}=\frac1{2(q-1)}p(1)\).
\end{Example}

Given the special role of the coprime-support polynomials, one are
tempted to ask about their circle number.  Half of the task is
easy: if \(n\) is even, \(R_n(x)\) is an odd polynomial, so
\(-1\) is an interlace cert and \(p\) is exact, with
\(\cn{p}=\il{p}=\frac{\varphi(n)}2\cdot\) The other half we
cannot do, and is left to the reader.  Computations up to
\(n=40\) suggest that \be{R_n} is positive, but vanishingly small.

% \textbf{-----------------------------------------------------------------------------------------\newline
%   This is food for thought -- I haven't checked anything }
% \cref{eq:icorbit} can be seen as a description of the FOIC
% restricted to \(n\)-even polynomials.  This is a much smaller
% subspace than \Trim{n}, its dimension is the number of divisors
% of \(n\) (maybe even have to subtract 1 from that).  So, for this
% new FOIC to be interesting, \(n\) must be highly composite
% (whatever that means :-).  On the other hand, this fan is normal
% to a very specific integer polytope, so there must be some
% interesting facts to be found.

% We can define on degree \(n\) polynomials the \emph{evenization} (linear, idempotent) operator
% \[
%   E(p)(x)=\sum_{d|n}\frac{R_{d}(x^{n/d})}{\varphi(d)}\sum_{(j,n)=d}p_j.
% \]  
% (or something similar) Clearly \(E(p)\) is \(n\)-even.  \(E\) can
% be seen as a map from \Trim{n} to the space of \(n\)-even
% polynomials.  Also, I think that \cref{rambound} implies that
% \(\il{E(p)}\leq\il{p}\), whith equality if \(p\) is interlace
% rational (by \cref{icorbit}).  Question: can \il{E(p)} be too far
% from \il{p}?

% What about \(\ker E\)?  It is a large subspace...

% BTW, this operator has a ``noble'' origin.  Identify polynomials
% of degree \(<n\) with the group algebra \(\R\Z_n\).  Let
% \(A_n=\Aut \Z_n\), and consider \(\R\Z_n\) as an \(\R A_n\)
% module.  Then \(E\) is just the idempotent
% \(\frac1{|A_n|}\sum_{g\in A_n}g\).

\subsection{Fekete polynomials}
\label{sec:Fekete-polynomials}
For prime \(n\), \(n\)-even polynomials are just geometric
polynomials and give us nothing new. Fekete polynomials are
interesting examples of prime degree polynomials whose interlace
number and certs have a nice description; they are defined by
\begin{equation} \label{Fekete}
  f_n(x) := \sum_{j=1}^{n-1} \left(
    \frac{ j }{ n } \right) x^j,
\end{equation}
for each prime \(n\), and we note that the darga of
\(f_n(x)\) is equal to \(n\).  Here $\left( \frac{ j }{ n } \right)$ is
the usual Legendre symbol, defined to be equal to $+1$ when $j$
is a quadratic residue mod $n$, and $-1$ when $j$ is a
non-quadratic residue mod $n$.  We recall that Gauss proved (in
his 1811 paper) the following closed form for the Gauss sum,
which is by definition equal to $f_n( \eipi{k;n} ) $
\begin{equation} \label{GaussFormula}
f_n( \eipi{k;n} ) = \epsilon_n \sqrt n   \left(  \frac{ k }{ n } \right),
\end{equation}
for all $k = 0, 1, 2, \dots, n-1$, where 
\[
\epsilon_n := 
\begin{cases}
1,   \text{ if  }  n \equiv 1  \text{ mod } 4,  \\
i,   \text{ if  }  n \equiv 3  \text{ mod } 4.
\end{cases}
\]

The polynomials $f_n(x)$ are trim, palindromic polynomials
provided \(n\equiv 1\mod 4\).  In that case, it follows from
\eqref{GaussFormula} that
$\left| f_n(x)\left( \eipi{k;n} \right) \right| = \sqrt n$, for
all $k = 0, 1, \dots, n-1$, and we may therefore read from the
latter identity the interlace number of the Fekete polynomials,
by using the \Ifor:
\begin{equation}  \label{eq:Fekete}
  \il{f_n(x)}= \frac{ \sqrt n}{2}\cdot
  \end{equation}

  It follows from \cref{FacesOfTheFan} that, as {\it every}
  primitive $n$th root of unity is a cert for $f_n(x)$, it lives in
  the same face of the FOIC as \(\text{tgcd}_n(x)\).

\subsection{Half monotonic polynomials}
\label{sec:unimodal-polynomials}

Consider a trim palindromic polynomial \(p\) of darga \(n\).  If
the sequence \(p_1, p_2,\ldots, p_{\floor*{n/2}}\) is monotonic,
increasing or decreasing, we say that the polynomial is \emph{half
monotonic}, increasing or decreasing, respectively, and strictly
so if all those coefficients are distinct.

\begin{pro}\label{increasing-upper}
  If \(p\) is strictly increasing, then
  \[
      \il{p}\leq \frac12p(1) -2\sum_{j=1}^{\lfloor n/4\rfloor-1}p_j-c(n)p_{\floor*{\frac{n}{4}}},
    \]
    where \(c(n)=1,2,2,3\) according to \(n\) being \(\equiv 0,1,2,3 \mod 4\).
  \end{pro}

\begin{proof}
  We refer to \cref{kwon-simple}, and notice that
  \(m(p)=p_{\floor*{\frac{n}{4}}}\),
  \(|M|=\floor*{\frac{n}{4}}-1\).  Hence,
  \begin{equation*}
    \il{p}\leq \frac12p(1)-2\sum_{j=1}^{\lfloor n/4\rfloor-1}p_j-
            \left(\floor*{\frac{n-1}{2}}-2\floor*{\frac{n}{4}}+2\right)p_{\floor*{\frac{n}{4}}}.
  \end{equation*}
  The result now follows from a case analysis. 
\end{proof}

A theorem of \citet{Chen}, extended by \citet[Theorem 7]{J2013}, can be stated as:
% \begin{teo}
%   For a trim palindromic \(p\), if for some \(m\leq n/2\),
%   \(p_1>p_2>\cdots>p_m>0\) and \(p_j=0\) for \(m<j\leq n/2\),
%   then \(cn(p)\leq p_1\).
% \end{teo}
\begin{teo}\label{chen}
  If \(p\) is decreasing and has nonnegative coefficients,
  then \(cn(p)\leq p_1\).
\end{teo}

A natural conjecture was that actually \(\il{p}\leq p_1\), at
least when the coefficients are nonnegative.  Alas, this is
false, in general.  A brute force search using \sage
produced several counterexamples.  Here are two examples, of odd
and even dargas:
\[
   15x^8 + 14x^7 + 12x^6 + 2x^5 + 2x^4 + 12x^3 + 14x^2 + 15x, 
 \]\[
   80x^9 + 75x^8 + 73x^7 + 11x^6 + 2x^5 + 11x^4 + 73x^3 + 75x^2 + 80x.
\]

For the first one, the interlace number is \(15.018885\ldots\)
and the circle number is \(\frac{23}{3}\); for the second one,
the values are \(90.6139\ldots\) and \(68\).

In spite of these examples, there is something positive that can
be said.  \cref{sumj2} implies that if \(p_1\) is sufficiently
larger than the remaining coefficients,
\(\il{p}=-\frac12p(-1)<\frac12p_1\).

In Joyner's article \cite{J2013}, the author asks about
nonnegative increasing polynomials; that is, for which values of
\(\alpha\in[0,p_1]\) \palpha is circle rooted.

We do not expect the concepts of circle and interlace number to
be of help here. Experiments show that \cn{p} is much larger than
\(p_1\), and we can actually prove that for \il{p}:

\begin{pro}\label{monotonic-lower}
  Suppose that \(p\in\Trim{n}\) is  monotonic increasing, and let \(m=\lfloor\frac{n}{2}\rfloor\). Then
  \begin{align*}
    \il{p}\geq &\; \sig{p}_m +(\sig{p}_{m-1}-\sig{p}_1)(1-5m^{-2})&&\text{if \(n\) is even},\\
    \il{p}\geq &\;  (\sig{p}_m-\sig{p}_1)(1-5n^{-2})&&\text{if \(n\) is odd}.
  \end{align*}
\end{pro}
\begin{proof}
  Let us write
  \(s_j=\sum_{k=1}^m -\sig{p}_k\cos\frac{2\pi j k}{n}\), for
  \(j=0,1,\ldots,m\).  The \Ifor says that
  \(\il{p}=\max_j s_j\), so, for any \(j\), \(\il{p}\geq s_j\).
  We will compute lower bounds for \(s_1\).

  \textsc{Case 1}: \(n=2m\) is even. 
  We can rewrite \(s_1=\sum_{k=1}^m -\sig{p}_k\cos\frac{\pi
    k}{m}\). Noticing that
  \(\cos\frac{\pi k}{m}=-\cos\frac{\pi (m-k)}{m}\) is positive
  for \(k\leq \frac{m}{2}\), we can rewrite:
  \(s_1=\sig{p}_m+\sum_{1\leq
    k\leq\frac{m}{2}}(\sig{p}_{m-k}-\sig{p}_k)\cos\frac{\pi
    k}{m}\), and, since \(p\) is increasing, all terms in the sum
  are nonnegative.  So,
  \(s_1\geq
  \sig{p}_m+(\sig{p}_{m-1}-\sig{p}_1)\cos\frac{\pi}{m}\).  From
  the Taylor series, \(\cos x\geq 1-x^2/2\), so
  \(\cos\frac{\pi}{m}\geq 1-\pi^2/{2m^2}\geq 1-5/m^2\), and this
  case is done.

  \textsc{Case 2}: \(n=2m+1\) is odd.  Notice that for
  \(1\leq k\leq\frac{m}{2}\) we have \\
  \(\cos 2\pi (m-k+1)/n\leq 0\) , and
  \(-\cos 2\pi (m-k+1)/n\geq \cos 2\pi k/n\).  Rewrite
  \(s_1=\sum_{1\leq k\leq\frac{m}{2}}\big(-\sig{p}_{m-k+1}\cos
  2\pi (m-k+1)/n-\sig{p}_k \cos 2\pi k/n\big)\) (actually, if
  \(m\) is odd, there is an additional negative term). Using
  the inequalities on cosines at the beginning of the paragraph,
  we can bound
  \(s_1\geq
  \sum_{1\leq
    k\leq\frac{m}{2}}(\sig{p}_{m-k+1}-\sig{p}_k)(-\cos2\pi
  (m-k+1)/n)\), and, as all summands are nonnegative, we can
  bound \(s_1\geq (\sig{p}_m-\sig{p}_1)\cos 2\pi/n\). We bound
  \(\cos 2\pi/n\) as in Case 1 to get the result.
\end{proof}

\subsection{Binomial polynomials}
\label{sec:binomial-polynomials}

An interesting case study of half-monotonic increasing polynomial is that of
\(B_n(x)=(1+x)^n-(1+x^n)\), trimmed binomial polynomials,
which will also be relevant in the next subsection.  Here we can
calculate the exact value
\begin{equation}  \label{eq:ilbinom}
  \il{B_n}= 2^{n-1} \cos^n \frac{\pi}{n}-1.
\end{equation}
Since \(\cos \frac{\pi}{n}=1-O(1/n^2)\), we have the asymptotics
\(\il{B_n}=2^{n-1}(1-O(1/n))\).
\begin{proof}[Proof of \eqref{eq:ilbinom}]
  We show that
  \(\min_{\w\in\vn} (1+\w)^n =-(2\cos\frac{\pi}{n})^n\), and then
  apply the \Ifor.  Let \(\w\in\vn\backslash\{1\}\), and let
  \(\zeta\in V_{2n}\) be such that \(\zeta^2=\w\).  Then,
  \((1+\w)^n=(1+\zeta^2)^n=2^n\zeta^n\Re(\zeta)^n\).  Since
  \(\zeta^n= \pm 1\), the last value is minimized at
  \(\zeta=\eipi{;n}\), as \(\zeta^n=-1\) and it has the largest
  real part, attaining the value
  \(-2^n\cos^n \frac{\pi}{n}\cdot\)
\end{proof}
Computation up to \(n=40\) suggests that \cn{B_n} is just above
\(2^{n-1} \cos^{n+3} \frac{\pi}{n}\), so
\(\be{B_n}< \frac1{\cos^3\frac{\pi}{n}}-1=O(\frac{1}{n^2})\).

The ``Hadamard inverse''
\(B_n^H=\sum_{k=1}^{n-1}\binom{n}{k}^{-1}x^k\) is decreasing, so
we have some information provided at the previous section;
actually, a much more precise results follow from \cref{sumj2}.
Writing \(p_k=\binom{n}{k}^{-1}\), one can show that for
\(k\geq 3\), \(k^2\sig{p}_k<3^2\sig{p}_3\); one way to verify this
is to note that \((k-1)^2\sig{p}_{k-1}<k^2\sig{p}_k\) for
\(k\leq n/2\) (the case \(k=n/2\) requires special attention).
Now,
\begin{align}
  \sig{p}_1 -\sum_{k=2}^{\floor*{n/2}} k^2|\sig{p}_k|
  &\geq p_1-4p_2-(\floor*{n/2}-2)\cdot3^2p_3 \label{abc}\\
  &=\frac1n-\frac8{n(n-1)}-\frac{27(n-4)}{n(n-1)(n-2)}\\
  &\geq \frac1n\left(1-\frac{35}{n-1}\right)
\end{align}
and this is positive if \(n\geq 36\). So, by \cref{sumj2}, for
those \(n\), \(-1\) is an interlace cert of \(B_n^H\) if \(n\) is
even, so \(B_n^H\) is exact.  For \(n\) odd, \(\tn^{\floor*{n/2}}\) %\eipi{m/(2m+1)}
is an interlace cert and it \emph{looks like} \(-1\) is a circle
cert.  Computer verification extends the same conclusions to
\(2\leq n\leq 35\).  Curiously, the expression on the left of
\cref{abc} is negative for \(6\leq n\leq17\), so \cref{sumj2}
cannot be applied to handle these cases.

% A curious thing happens with the ``Hadamard inverse''
% \(B_n^H=\sum_{k=1}^{n-1}\binom{n}{k}^{-1}x^k\).  For odd \(n\),
% it looks like \be{B_n^H} goes to zero, albeit much slowlier than
% \be{B_n}.  And, for even \(n\), it appears that -1 is an
% interlace cert, so \(B_n^H\) is exact; we have not found a proof
% of that.

\subsection{Cut polynomials}
\label{sec:cut-polynomials}

Let \(A\) be an \(n\times n\) Hermitian matrix, define its
\emph{cut polynomial} (see \cite{Barv}) by
\[
  \text{Cut}_A(x) =
  \sum_{S\subseteq\{1,\ldots,n\}}x^{|S|}\prod_{\substack{i\in S \\ j\not\in S}}a_{ij},
  \]
  and let \(\text{TCut}_A(x)=\trim \text{Cut}_A(x)\).  

  The celebrated Lee-Yang Theorem \cite{LY} states that
  \begin{quote}
  \emph{if \(|a_{ij}|\leq 1\) for all \(i,j\), then
    \(\text{Cut}_A(x)\) is circle rooted}.
  \end{quote}

  %Restricting \(A\) to symmetric real matrices,
  One easily sees
  that the cut polynomial is self-inversive, hence the question comes
  to mind as whether the Lee-Yang Theorem has any relation to
  either interlace number or circle number.

  The answer is, nothing much, in general.  Consider, for
  instance, the matrix \(J\) of all 1's.  It is easy to see that
  \(\text{Cut}_J(x)=(1+x)^n\), so \(\text{TCut}_J(x)=B_n(x\)).  As
  we have seen, both \il{B_n} and \cn{B_n} are very large, and
  say nothing about the monic cut polynomial.

  But all is not lost.  Suppose we restrict a little the set of
  matrices, requiring that all \(|a_{ij}|\leq \lambda\), for some
   \(0<\lambda(n)<1\).  Then, each product in the definition of
  \(\text{Cut}_A(x)\) has absolute value at most
  \(\lambda^{k(n-k)}\), where \(k=|S|\), and we can apply 
  \cref{LL}  to conclude that
  \[
      \il{\text{TCut}_A}\leq \frac12\sum_{k=1}^{n-1}\binom{n}{k}\lambda^{k(n-k)}\leq\sum_{k=1}^{\floor{n/2}}\binom{n}{k}\lambda^{k(n-k)}
    \]
 
    Denote
    \(LY(x)=\sum_{k=1}^{\floor{n/2}}\binom{n}{k}x^{k(n-k)}\); if
    \(LY(\lambda)\leq 1\), then it follows that
    \(\text{Cut}_A(x)\) angle interlaces \un, so in this case we
    have a stronger conclusion than in Lee-Yang. For \(n>2\),
    \(LY(1)>0\), \(LY(0)=-1\), and \(LY(x)\) is increasing for
    positive \(x\). So, the best possible value for \(\lambda\)
    is the only positive root of \(LY(x)-1\), which probably does
    not have a neat closed form.

    Let us show that \(\lambda=1/(cen)^{2/n}\) is a lower
    bound for the root, for some \(c>1\).

    Indeed, for \(1\leq k\leq n/2\),
    \[
      \binom{n}{k}\lambda^{k(n-k)}\leq \binom{n}{k}\lambda^{kn/2}\leq
      \left(\frac{ne}{k}\right)^k (cen)^{-k}=\frac{1}{c^kk^k}\cdot
      \]

      Hence,
      \(LY(\lambda)\leq \sum_{k\geq 1}\frac{1}{(ck)^k}\leq
      \frac{1}{c}\sum_{k\geq 1}\frac{1}{k^k}\leq
      \frac{1}{c}\sum_{k\geq 1}\frac{1}{4^{k-1}}\leq 1\) for
      \(c\geq 4/3\) (although \(5/4\) seems to be more accurate);
      that is, \(\lambda=\text{e}^{-2\frac{\log n}{n}-h}\), where \(h\)
      is a constant, slightly larger than 1.

      Anyway, we have shown that \(\il{\text{TCut}_A}\leq 1\) if
      \(|a_{ij}|\leq \text{e}^{-2\frac{\log n}{n}-h}\).  The actual root
      of \(LY(x)-1\) seems to grow as
      \(\text{e}^{-\alpha\frac{\log n}{n}}\), where \(1<\alpha<2\).

      % On one hand, since \(\lambda\rightarrow 1\) as \(n\)
      % increases, we can imagine that set expanding to snuggly fit
      % into \(n(n-1)/2\) dimensional cube of symmetric matrices
      % with \(|a_{ij}|\leq 1\); in fact its volume is but a tiny
      % part of the volume of that cube. On the other hand, our
      % result holds for any matrices, not just symmetric ones.

%%%%%%%%%%%%%%%%%%%%%%%%%%%%%%%%%%%%%%%%%%%%%
%%%%%%%%%%%%%%%%%%%%%%%%%%%%%%%%%%%%%%%%%%%%%
  
\bigskip
  
\section{Open problems}
\label{sec:open-problems}

Here we delve into some open questions that are scattered
throughout the paper, and we present future research directions in which we
hope more people can join us.

The general theory of the interlace number seems quite
satisfactory at this point.  The circle number is a different ball
game.  There is some hope that a concerted attack on the
semi-algebraic side will yield new insights.  The same comments hold for
exact polynomials.

\begin{problem}
  What is the geometry (maybe just the topology) of the set of
  exact polynomials, for each darga \(\geq6\)?  Is it connected?
  In particular, what happens within each \(C_j\), where the
  exact polynomials fall into distinct components; can these
  polynomials be parameterized?
\end{problem}

\begin{problem}
  What about the growth of \(\text{BE}(n)\)?  We have shown a
  linear lower bound, and the construction used for it is intuitively
  (for us) best possible, but we have no proof of that.  In
  \cref{sec:small-dargas} we have shown that \(\text{BE(4)}\) and
  \(\text{BE(5)}\) are not attained by any polynomial.  Does this
  happen for all larger values of \(n\)?
\end{problem}

Polynomials with rational coefficients give rise to several
questions; in most of them, there is no loss of generality in
considering only integer polynomials.

\begin{problem}
  What are the minimal faces of the FOIC that contain rational
  points?  Most rays clearly do not (except in darga 6).  On the
  other hand, \cref{sec:interl-rati-polyn} gives several examples
  of faces that do.  Are there interlace rational polynomials in
  those minimal faces?
\end{problem}

\begin{problem}
  We repeat the same question as above, but for exact
  polynomials.  Outside of the polyhedral cones \(C_0\) and
  \(C_{n/2}\) in the FOIC, what do exact rational polynomials
  look like?  In \cref{exapol}, we present a family of exact
  rational polynomials.  Does every component of the set of exact
  polynomial contain a rational point? Are there interlace
  rational exact polynomials?
\end{problem}

\begin{problem}
  \cref{sumj2} can be unwound into an inequalities description of a 
  rational cones contained in the cone \(C_{n/2}\).  Can we
  get, for each \(j\), interesting rational cones contained in
  \(C_j\)?
\end{problem}

\begin{problem}
  Is there an algorithm involving only rational arithmetic for
  finding out all the interlace certs of a given rational
  polynomial, describing them as appropriate powers of \tn?
\end{problem}

\begin{problem}
  What are the ranges of \il{}, \cn{} and \be{}, computed over
  integer polynomials?  Are all nonnegative real algebraic
  numbers attainable?  For \il{}, there is already a restriction:
  it must belong to a cyclotomic field.
\end{problem}

\begin{problem}
  The gcd polynomials empirically share the following properties: they
  are never exact, but their bounding error is \(O(1/n)\) for
  darga \(n\).  In particular, this includes the geometric
  polynomials of prime darga, for which this is a known fact
  (\cref{ex:geom-be}).  Is it always true?
\end{problem}

\begin{problem}
  A global view of rational polynomials could be provided by the
  development of Erhart theory for the interlace simplex, and for
  the cones in the FOIC.  That is highly developed and well
  understood in case of rational polyhedral cones, but little is
  known for cones that are not rational.  The FOIC cones are, in
  a sense, just shy of being rational, and may be still
  treatable.
\end{problem}

Finally, here is a different type of question, which was already hinted at in 
\cref{ex:twointerv}:

\begin{problem}
  Is there a family polynomials \(p\) for which the set of values
  \(\alpha\) for which \palpha is circle rooted fall into
  increasingly many disjoint intervals?
\end{problem}

\printglossary[title={Terminology and notation}]

% \bibliography{pols}{} \bibliographystyle{plain}
\printbibliography
\end{document}